\documentclass[10pt]{amsart}
\usepackage{silliman-header}
\usepackage[margin=1in]{geometry}

\author{Jesse Silliman}
\title{Irrational periods of Hilbert Eisenstein series via toroidal compactification}
\date{\today}

\address{Duke University, Department of Mathematics, 120 Science Dr, Durham, NC 27710}
\email{jksilliman@gmail.com}

\renewcommand{\todo}[1]{}

\begin{document}

\begin{abstract}
We show that the periods of the holomorphic Eisenstein series of level 1, parallel weight 2, on a Hilbert modular surface are not rational, even up to scaling. This is deduced from a study of the mixed Hodge structure on the cohomology of the Hilbert modular surface, where we find extension classes related to the units of the real quadratic field. We prove similar results for Hilbert modular varieties of all dimensions, and produce extensions of Galois representations in \'etale cohomology. The key point is to study the restriction of the canonical extension of the Hodge bundle to the boundary of a smooth toroidal compactification.
\end{abstract}

\maketitle

\setcounter{tocdepth}{1}
%\tableofcontents

\section{Introduction}
\subsection{Harder's Eisenstein cohomology}
Let $F/\Q$ be a totally real number field with integers $\O_F$. We assume that $\O_F$ has narrow class number one. Let $G = \Res_{F/\Q}\GL_2$, $G^0 = \Res_{F/\Q}\SL_2$, $G^0(\Z) = \SL_2(\O_F) \subset G^0(\Q)$, $\Gamma \subset G^0(\Z)$ a congruence subgroup. Let $V$ be a finite-dimensional algebraic representation of $G$, i.e.\ a $\Q$-vector space $V$ with a map $\rho \from G \to \GL(V)$. Harder \cite{Harder} defines a map \[ \del \from H^*(\Gamma, V) \to \oplus_{P} H^*(\Gamma_P, V), \] where $P$ runs over $\Gamma$-conjugacy classes of $\Q$-parabolic subgroups of $G^0$, $\Gamma_P = \Gamma \cap P(\Z)$. He proves that there is a unique $\mathcal{H} \times \pi_0(G(\R))$-equivariant section of $\del \from H^*(\Gamma, V) \to \im(\del)$, where $\mathcal{H}$ is the spherical Hecke algebra. In most cases, the spherical Hecke operators suffice to prove uniqueness, as they have distinct eigenvalues on the kernel and image of $\del$. However, in the case of $V = \Q$, the trivial representation of $G$, he also uses the action of $\pi_0(G(\R))$ (\cite{Harder}, pg. 86). Hecke operators do not suffice, because $\ker(\del)$ and $\im(\del)$ both could contain the trivial representation of $\mathcal{H}$, i.e.\ the representation where $T_{\q} \in \mathcal{H}$, for $\q$ a prime of $\O_F$ not dividing the level, acts as $T_{\q}(v) = (1 + N_{F/\Q}(\q))v$. %We let $H^*(\Gamma)_{triv}$ denote the summand of cohomology given by localizing at the trivial representation of $\mathcal{H}$.

\subsubsection{Hilbert modular varieties}
As $F/\Q$ is totally real, $H^*(\Gamma, \Q)$ may be identified with the singular cohomology of a complex algebraic variety $\H^d/\Gamma$, a Hilbert modular variety, and we can identify $H^*(\Gamma, \C)$ with its de Rham cohomology. This comes equipped with a Hodge filtration $\Fil^*$ from Deligne's mixed Hodge theory. The action of $\mathcal{H}$ preserves this filtration, for example sending holomorphic forms to holomorphic forms. However, the action of $\pi_0(G(\R)) = (\Z/2\Z)^d$ does not preserve holomorphicity. When $d = 2$, the element $c = c_1 \circ c_2 \in (\Z/2\Z)^2$ sends $f(z_1, z_2) dz_1 \wedge dz_2$ to $f(-\bar{z_1}, -\bar{z_2}) d\bar{z_1} \wedge d\bar{z_2}$. This raises the possibility that Harder's section does not behave well with respect to the Hodge filtration $\Fil^*$ on $H^*(\H^d/\Gamma, \C)$. In other words, his section might not send classes to their ``most holomorphic" lift.

For $F$ a real quadratic field, consider a Hilbert modular surface $Y = \H^2/\Gamma$, for $\Gamma \subset \SL_2(\O_F)$ a torsion-free congruence subgroup. Consider the following holomorphic Eisenstein series:
\[ E_{(2,2)} = \left( -\frac{\zeta_F(-1)}{4} + \sum_{\alpha \in \mathfrak{d}^{-1}, \alpha \gg 0} \sigma_1(\alpha \mathfrak{d}) e^{2 \pi i \cdot \tr(\alpha z)} \right) dz_1 \wedge dz_2. \]

It is a closed, holomorphic 2-form on $Y$, with logarithmic poles along the boundary $\del X$ of any smooth toroidal compactification $X$. It gives a class $[E_{(2,2)}] \in \Fil^2H^2(\H^2/\Gamma, \C)^{\SL_2(\O_F)}_{triv}$ whose image under $\del$ is a non-zero element of $\oplus_P H^2(\Gamma_P, \Q)$ (see \ref{eis-section-hodge}, \ref{prop:MHS-eis-cohom}). In particular, if it was contained in $H^2(\Gamma, \Q)$, it would be Harder's lift of $\del([E_{(2,2)}])$. However, it is not:
\begin{cor}\label{cor:irrat-periods}
No non-zero scalar multiple of $E_{(2,2)}$ has rational periods.
\end{cor}

This is a corollary to our Main Theorem \ref{thm:eis-periods-1}. It may be possible to find an actual cycle exhibiting these irrational periods, but for us the interest in this is as a concrete consequence of our general result about Hodge structures.

This has a ``motivic" consequence:
\begin{cor}
No non-zero scalar multiple of the form $E_{(2,2)}$ can be written as \[ \sum_i c_i \frac{df_i}{f_i} \wedge \frac{dg_i}{g_i} \] for $c_i \in \Q$, $f_i, g_i \in \C(Y)^*$.
\end{cor}

%I removed the factors of $(2 \pi i)^2$ from these statements.

In other words, no non-zero scalar multiple of the algebraic differential form $E_{(2,2)}$ is in the image of the map from Milnor K-theory, $K_2^M(\C(Y)) \to \Omega^2_{\C(Y)/\C}$.

We can prove Theorem \ref{thm:eis-periods-1}, as it applies to Corollary \ref{cor:irrat-periods}, in two essentially different ways:
\begin{enumerate}
\item The Hodge bundles $\omega_1$, $\omega_2$ on $Y$ have canonical extensions $\bar{\omega_i}$ to smooth toroidal compactifications $X$. We consider the restrictions $\bar{\omega_i}|_{\del X}$, and compute their Chern classes in Deligne-Beilinson cohomology. We relate this to the non-semisimplicity of the mixed Hodge structure on $H^2(Y, \Q)$.
\item Using $(\g, K)$-cohomology, the differential form $E_{(2,2)} - \bar{E_{(2,2)}}$ is shown to be cohomologous to a linear combination of $\frac{dz_i \wedge d\bar{z_i}}{y_i^2}$, $i=1,2$, which is non-zero in $H^2(Y, \C)$. This relation occurs because of a simple pole in the meromorphic continuation of the non-holomorphic Eisenstein series $\sum_{(c,d) \in \O_F^2/\O_F^*}^{'} \frac{1}{|(c_1z_1+d_1)(c_2 z_2 + d_2)|^{2s}}$ as $s = 1$, which in turn comes from the pole of $\zeta_F(s)$ at $s = 1$.
\end{enumerate}

As method (1) generalizes to prove analogous results in \'etale cohomology, it will be the focus of the paper. We sketch method (2) in Sections \ref{sec:gK-cohom}, \ref{sec:weak-eis-periods}. It would be interesting to imitate method (2) in the framework of $p$-adic modular forms.

\subsection{Statement of theorems}
Let $F/\Q$ be a totally real field of degree $d > 1$, with real embeddings $\sigma_i \from F \into \R$, $i = 1,\ldots, d$, and integers $\O \subset F$. We write $\Q(j) = (2 \pi i)^j \Q \subset \C$. We assume $F/\Q$ is Galois, and that $\O$ has narrow class number one. Both assumptions are for convenience - the former makes certain line bundles defined over $F$ as opposed to its Galois closure over $\Q$, and the latter lets us work with geometrically connected Shimura varieties.

We study the associated Hilbert modular variety $Y = \H^d/\SL_2(\O)$. Consider the subgroup \[ P(\Z) = \left\lbrace \begin{pmatrix} * & * \\ 0 & * \end{pmatrix}  \in \SL_2(\O) \right\rbrace. \] We have a map $i_{\infty} \from \H^d/\SL_2(\O) \to \H^d/P(\Z)$, which allows us to think of $\H^d/P(\Z)$ as the ``boundary" of the space $Y$. Its cohomology is well-understood: for example, 
\[ H^{2d-2}(\H^d/P(\Z), \Q(d)) \isom H_1(\H^d/P(\Z), \Q) \isom \O^* \tensor \Q \text{ (\ref{lem:borel-serre-units})}. \]

The ``restriction" map \begin{align}\label{singular-boundary-map}
H^{2d-2}(Y, \Q(d)) \nmto{i^*_{\infty}} H^{2d-2}(\H^d/P(\Z), \Q(d)) \isom \O^* \tensor \Q
\end{align} 
admits, after passing to $\C$ coefficients, a section $\Eis$, defined using Eisenstein series and uniquely characterized as a Hecke-equivariant section such that its image lies in $\Fil^d H^{2d-2}(Y, \C)$ (see Section \ref{subsec:Eis-MHS}). Our goal is to describe the obstruction to $\Eis$ being defined over $\Q$.

To do this, we introduce some other differential forms on $Y$. There are forms $\frac{dz_i \wedge d\bar{z_i}}{y_i^2}$, $i = 1,\ldots,d$, on $Y$, giving cohomology classes \[ \omega_i \in H^2(Y, \Q(1)) \subset H^2(Y, \C), \] where this rationality is because they are Chern classes of line-bundles. We define \[ \omega_i^* := \omega_1 \wedge \ldots \wedge \wh{\omega}_i \wedge \ldots  \wedge \omega_d \in H^{2d-2}(Y, \Q(d-1)). \]

\begin{thm}\label{thm:eis-periods-1} For all $u \in \O^*$, we have
\[ \Eis(u) = \pm\frac{1}{2|\zeta_F(-1)|} \sum_{i = 1}^d \log(|\sigma_i(u)|) \omega_i^* \mod H^{2d-2}(Y, \Q(d)). \]
%\in \frac{H^{2d-2}(Y, \C)}{(2 \pi i)^d H^{2d-2}(Y, \Q)}
In particular, for $u \neq 0$, $\Eis(u) \notin H^{2d-2}(Y, \R(d)),$ since $\omega_i^* \in H^{2d-2}(Y, \R(d-1))$.
\end{thm}
\begin{rmk}
The map $\Eis \from \O^* \to H^{2d-2}(Y, \C)$ has a sign ambiguity, depending on a choice of an orientation of the torus $(\O^* \tensor \R)/\O^*$. The right side has a sign ambiguity as well, in that we have chosen an ordering for the $d$ real embeddings $F \into \R$. As the latter choice can be used to determine the former, it should be possible to state a result without a $\pm$. However, there are many opportunities for sign errors in the paper, and we have not attempted to ensure that all signs are correct.
\end{rmk}

% $vol(X) \cdot s(u) = - \sum_{i = 1}^d \log(\sigma_i(u)) \omega_i^*$ 
This implies that our section $\Eis$, defined by compatibility with the Hodge filtration, does not agree with Harder's, defined by compatibility with rational structures. Moreover, it implies that $\Fil^d \cap H^{2d-2}(Y, \Q) = 0$.

%We can also consider it as a generalization of an easier fact about Hilbert modular varieties: \[ H^{2d-1}(Y, \C) = 0, \text{ while } H^{2d-1}(Y_0(p), \C) = \C, \] (when $h_{K,+} = 1$). \begin{rmk}It is known that this class on $Y_0(p)$ is in the image of motivic cohomology $H^{2d-1}_{\Mot}(Y_0(p), \Q(d)) \to H^{2d-1}(Y_0(p), \Q(d))$ (TODO: ref. Kings?).\end{rmk}

We have analogous results in \'etale cohomology. For simplicity, assume $d > 2$. %Fix $N \geq 3$, and let $L = F(\mu_N)$. 
There is a variety $Y/\Q$ such that $(Y_{\C})^{\an} \isom \H^d/\SL_2(\O)$\footnote{As we only state results for singular/\'etale cohomology with $\Q$ or $\Q_l$-coefficients, they hold both for the the smooth stack $Y$ and its singular coarse space $[Y]$. For simplicity, we will only discuss the singular coarse space. In our proofs, we will frequently work with auxiliary level $\Gamma(N)$, $N \geq 3$, where this distinction does not occur.}. If we base-change to $Y_F$, there are line-bundles $\L_i$, $i = 1,\ldots,d$, on $Y_F$, with Chern classes $2\omega_i := c_1(\L_i) \in H^2_{\Et}(Y_{\bar{F}}, \Q_l(1))$. Their $(d-1)$-fold cup products define a subspace $\oplus_i \Q_l \subset H^{2d-2}_{\Et}(Y_{\bar{F}}, \Q_l(d-1))$. There is an exact sequence
of $G_F$-modules \[ 0 \to \oplus_i \Q_l(1) \to H^{2d-2}_{\Et}(Y_{\bar{F}}, \Q_l(d)) \to \O^* \tensor \Q_l \to 0, \] where the map $H^{2d-2}_{\Et}(Y_{\bar{F}}, \Q_l(d)) \to \O^* \tensor \Q_l$ is the $\Q_l$-linear extension of the map (\ref{singular-boundary-map}) above (identifying singular and \'etale cohomology via a choice of $\bar{F} \subset \C$). This map is in fact $G_F$-equivariant, where $\O^* \tensor \Q_l$ is given the trivial action of $G_F$ (see Section \ref{subsec:nonsplit-extn}). Kummer theory identifies this extension of $G_F$-modules with a map \[ m \from \O^* \tensor \Q_l \to \oplus_i F^* \tensor \Q_l. \]

We compute this extension class completely:
\begin{thm}\label{thm:etale-extn-class}
For all $u \in \O^*$, \[ m(u) = \pm \frac{1}{2 |\zeta_F(-1)|} (\sigma_1(u),\ldots, \sigma_d(u)). \]
\end{thm}
In particular, this extension is non-trivial, and $H^{2d-2}_{\Et}(Y_{\bar{F}}, \Q_l(d))^{G_F} = 0$.

%We interpret this as saying that there is no distinguished ``Eisenstein series of level 1" in \'etale cohomology $H^{2d-2}_{\Et}(Y_{\bar{F}}, \Q_l(d))$. 

\subsection{Sketch of the proof}
We sketch the proof of the \'etale-theoretic Theorem \ref{thm:etale-extn-class}. For simplicity, we assume $d = 2$, $Y = \H^2/\Gamma(N)$, $N \geq 3$, and prove a result about extension classes occuring in $H^{2}(Y)$.  Even though we stated Theorem \ref{thm:etale-extn-class} only for $d > 2$ above, the purpose of this was only to avoid the contribution of cusp forms to $H^{2d-2}$. Thus, we also assume that $H^2(Y(\C), \C)$ has no contributions from cusp forms. This assumption is not realistic, but can be removed by localizing at the trivial representation of the Hecke algebra. 

We consider $Y$ as a variety over $L := F(\mu_N) \subset \C$. Let $X$ be a smooth, projective toroidal compactification of $Y$. The boundary divisor $\del X = X - Y$ is a SNCD, and for each $x \in \pi_0(\del X) = \Cusps(\Gamma(N))$, the connected component $(\del X)_x$ is isomorphic to a circle of $\P^1$'s, \[ (\del X)_x = (\coprod_{i \in \Z/n\Z} \P^1)/(\infty_i = 0_{i+1}). \] The boundary component corresponding to the cusp at infinity, $\del X_{\infty}(\C)$, has fundamental group identified with \[ T(\Z) := \left\lbrace \begin{pmatrix} u & 0 \\ 0 & u^{-1} \end{pmatrix} \mid u \in \O^* \right\rbrace \cap \Gamma(N) \subset \O^*. \]

We have an exact sequence
\[ 0 \to H^1(\del X_{\bar{L}}, \Q_l(1)) \to H^2((X, \del X)_{\bar{L}}, \Q_l(1)) \to H^2(X_{\bar{L}}, \Q_l(1))^{\perp \del X} \to 0,  \]
defining \[ H^2(X_{\bar{L}}, \Q_l(1))^{\perp \del X} := \ker(H^2(X_{\bar{L}}, \Q_l(1)) \to H^2(\del X_{\bar{L}}, \Q_l(1))).\]

Noting that $H^1(\del X_{\bar{L}}, \Q_l(1)) \isom \oplus_{\pi_0(\del X)} \Hom(\O^*, \Q_l(1))$, we see that, up to a twist by $\Q_l(1)$, this sequence is dual to the extension
\[ 0 \to \Q_l(1) \omega_1 \oplus \Q_l(1) \omega_2 \to H^2_{\Et}(Y_{\bar{F}}, \Q_l(2)) \to \oplus_{\pi_0(\del X)} \O^* \tensor \Q_l \to 0 \]
which we wanted to understand. Thus it suffices to compute the extension class of the former sequence. To do this, we use:

\begin{prop}\label{prop:line-bundle-extn}
Suppose we have a smooth projective variety $X$ over a field $L$, with an embedded circle of $\P^1$'s $Z \into X$, ``oriented" by $\gamma \in H_1^{\Sing}(Z(\C), \Z)$.
\begin{enumerate} 
\item Defining $\Pic^0(Z) := \{ \L \in \Pic(Z) \mid c_1(\L) = 0 \in H^2(Z_{\bar{L}}, \Q_l(1)) \}$, we have $\Pic^0(Z) = L^*$.
\item Consider a line-bundle $\L$ on $X$ such that $\L|_{Z}$ is contained in $\Pic^0(Z)$, corresponding to $a \in L^*$. We have an extension class in $\Ext^1_{G_L}(\Q_l, \Q_l(1))$, giving by pulling back/pushing out via the following maps:
\[ \begin{tikzcd}[column sep = small]
& & & \Q_l \arrow{d}{c_1(\L)} & \\
0 \arrow{r} & H^1_{\Et}(Z_{\bar{L}}, \Q_l(1)) \arrow{r} \arrow{d}{\gamma} & H^2_{\Et}((X, Z)_{\bar{L}}, \Q_l(1)) \arrow{r} & H^2_{\Et}(X_{\bar{L}}, \Q_l(1))^{\perp Z} \arrow{r} & 0 \\
 & \Q_l(1) & & &
\end{tikzcd} \]

This extension class equals the Kummer class $\kappa(a) \in \Ext^1_{G_L}(\Q_l, \Q_l(1))$.
\end{enumerate}
\end{prop}
We prove this, and a Hodge-theoretic analog, in Section \ref{sec:line-bundle-extn}, using cycle classes in absolute \'etale and Hodge cohomology.

We see that we need to:
\begin{enumerate}
\item find line-bundles $\L$ on $X$ whose Chern classes span $H^2(X_{\bar{L}}, \Q_l(1))^{\perp \del X}$,
\item compute $\L|_{\del X} \in \Pic^0(\del X) = \oplus_{\pi_0(\del X)} \Hom(\O^*, L^*)$. 
\end{enumerate}

We can perform both of these steps over $\C$, as long as the line-bundles themselves are defined over $L$. We consider the \emph{canonical extensions} to $X$ of the line-bundles $\L_i$ whose Chern classes gave us $2\omega_1, 2\omega_2$. These are line-bundles $\bar{\L_1}, \bar{\L_2}$, with Chern classes $2\bar{\omega_1}, 2\bar{\omega_2} \in H^2(X, \Q(1))$. We can describe these line-bundles as follows. There is an isomorphism
\[ \Omega^1_X(\log \del X) \isom \bar{\L_1} \oplus \bar{\L_2}. \] Away from the boundary, $\bar{\L_i}$ is locally spanned by 1-forms of the form $f(z)dz_i$, $f$ homomorphic. In a small enough (analytic) neighborhood $V$ of any point $x \in \del X_{\infty}(\C)$, 
\[ \bar{\L_i}(V) = \O(V) \cdot \L_i(V - V \cap \del X)^{\mathfrak{n} = 0} = \O(V) dz_i, \]
for $\mathfrak{n} = \R \cdot \frac{\del}{\del x_1} + \R \cdot \frac{\del}{\del x_2}$ (see Lemma \ref{lem:canonical-extn}).

Importantly, $dz_i$ is a \emph{non-vanishing} section of $\bar{\L_i}^{\an}$ along $\del X_{\infty}$. This implies that $\bar{\L_i}|_{\del X}(V \cap \del X) = \O_{\del X}(V \cap \del X) \cdot dz_i$, and $c_1(\bar{\omega_i})|_{\del X_{\infty}} = 0$. The group $T(\Z) \subset \O^*$ acts on $dz_i$ by
\[ u^*(dz_i) = 2\sigma_i(u) dz_i. \]
This completes step (2) for these line-bundles (using the $\SL_2(\O/N\O)$-invariance of $\L_i$ to see that it has the same behavior at all cusps, not just at $\infty$).

These line-bundles satisfy (1) as well. It is easy to see that $\bar{\omega_1}, \bar{\omega_2}$ then span $H^2(X_{\bar{L}}, \Q_l(1))^{\perp \del X}$, as they pair non-degenerately with $\im(H^2(X_{\bar{L}}, \Q_l(1)) \to H^2(Y_{\bar{L}}, \Q_l(1))) = \Q_l \omega_1 + \Q_l \omega_2$ (this follows from Mumford's Proportionality Theorem, see \ref{thm:mumford-proportionality}, as well as our assumption that $H^2(Y,\C)$ contains no contributions from cusp forms). 

This completes (1) and (2), and hence computes the extension class of $H^2(Y_{\bar{L}}, \Q_l(2))$ as a $G_L$-module. This completes the sketch of the proof of Theorem \ref{thm:etale-extn-class}. In the Hodge-theoretic setting, there is more to show - we must show not only compute an extension class occuring in the cohomology of $Y$, but also relate it to the periods of Eisenstein series. In the case $d = 2$, the Eisenstein series in question is a holomorphic 2-form with logarithmic poles along $\del X$, giving an element of $\Fil^2H^2(Y, \C)$, which would split the extension \[ 0 \to \Q(1) \omega_1 \oplus \Q(1)\omega_2 \to H^2(Y, \Q(2)) \to \Q(0) \to 0 \] precisely if it were contained in $H^2(Y, \Q(2)) \cap \Fil^2H^2(Y, \C)$.

\subsection{Other results}

These phenomena do \emph{not} occur for Eisenstein series in lower degrees. For $1 \leq j \leq d-1$, there is an exact sequence
\[ 0 \to \wedge^{d-(j-1)/2} (\oplus_i \Q)((j-1)/2) \to H^{2d-j-1}(Y(\C), \Q(d)) \to \wedge^j \O^* \tensor \Q \to 0, \]
where the subgroup is zero unless $j$ is odd. There is a section \[ \Eis \from \wedge^j \O^* \tensor \C \to H^{2d-j-1}(Y(\C), \C),\] explicitly given by Eisenstein series, characterized by its behavior with respect to Hecke operators and the Hodge filtration.

\begin{thm}\label{thm:lower-rationality}
The Eisenstein series \[ \Eis(\wedge^j \O^*) \subset H^{2d-1-j}(Y(\C), \C) \] are contained in $H^{2d-1-j}(Y(\C), \Q(d))$.
\end{thm}

\begin{thm}\label{thm:lower-semisimplicity}
The $G_F$-module $H^{2d-1-j}_{\Et}(Y_{\bar{F}}, \Q_l(d))$ is semisimple, with \[ H^{2d-1-j}_{\Et}(Y_{\bar{F}}, \Q_l(d))^{G_F} = \Q_l. \]
\end{thm}

These results only have content when $j$ is odd, and even then are fairly elementary (modulo the fact that $\Eis(\wedge^j \O^*) \subset \Fil^d H^{2d-1-j}(Y(\C), \C)$). With somewhat more difficulty, we can show that (for $j \neq d-1$) these Hodge/Galois invariant classes lift to Deligne-Beilinson/absolute \'etale cohomology, but we will not do so in this paper. We expect this to have applications to the study of p-adic Eisenstein series.

\begin{comment}
We expect that they fail to lift to Nekovar--Scholl's plectic absolute Hodge cohomology \cite{Nekovar-Scholl}, with obstruction classified by \[ \wedge^j \log \from \wedge^j \O^* \to \wedge^j \oplus_{v \mid \infty} \R. \] 

It should be possible to prove this using $(\g, K)$-cohomology. We do not know how to give a motivic proof/an \'etale analog, as the plectic structure on motivic/\'etale cohomology is still conjectural. 
\end{comment}

\subsection{Related work}\label{subsec:related-work}
\subsubsection{Hilbert modular varieties}
It has come to our attention that that case $d = 2$ of our main results were proven by Caspar \cite{Caspar}. Moreover, in the case $d > 2$, Scholl--Davidescu have announced (2018) a proof of the Galois-theoretic Theorem \ref{thm:etale-extn-class}, and their method would work with little modification in the Hodge-theoretic setting as well. Their techniques are similar to ours, which is not suprising, as our results on cycle classes in absolute cohomology (Section \ref{sec:extn-classes}) rely on older work of Scholl's \cite{Scholl}. However, we believe our technique for proving the Hodge-theoretic results using the residues in the analytic continuation of Eisenstein series (Sections \ref{sec:gK-cohom}, \ref{sec:weak-eis-periods}) is new.

\subsubsection{Other Shimura varieties}
After completing the bulk of this paper, we discovered a paper by Looijenga \cite{Looijenga} which studies the Chern classes of the Hodge bundle $\omega$ on Siegel moduli space $Y = \mathcal{A}_g$, and their relationship with extensions of mixed Hodge structures and Galois representations. 

There are toroidal and minimal compactifications $Y^{\mathrm{tor}}, Y^{\mathrm{min}}$. The vector bundle $\omega$ has a canonical extension to a vector bundle $\omega^{\mathrm{can}}$ to $Y^{\mathrm{tor}}$, but this does not descend along the map $\pi \from Y^{\mathrm{tor}} \to Y^{\mathrm{min}}$. Nevertheless, Goresky--Pardon \cite{Goresky-Pardon} showed that the Chern classes $c_i(\omega^{\mathrm{can}}) \in H^{2i}(Y^{\mathrm{tor}}, \Q(i))$ are in the image of $\pi^* \from H^{2i}(Y^{\mathrm{min}}, \Q(i)) \to H^{2i}(Y^{\mathrm{tor}}, \Q(i))$, by producing canonical lifts in $H^{2i}(Y^{\mathrm{min}}, \C)$. Looijenga shows that, while these lifts live in $\Fil^{i}H^{2i}(Y^{\mathrm{min}}, \C)$, as one expects from Chern classes of vector bundles, they do not live in $H^{2i}(Y^{\mathrm{min}}, \Q(i))$, and hence generate a non-trivial extension of mixed Hodge structures inside of $H^{2i}(Y^{\mathrm{min}}, \Q(i))$. His method uses:
\begin{enumerate}
\item embeddings $B\GL_g(\Z) = \GL_g(\Z) \backslash \GL_g(\R) / O(g) \subset \mathcal{A}_g$,
\item the fact that the restriction of the Hodge bundle $\omega$ to $B\GL_g(\Z)$ is a scalar extension of the standard local system $\rho \from \GL_g(\Z) \to \GL_g(\C)$,
\item the fact that the Borel regulator
\[ H_{2i-1}(B\GL_g(\Z), \Z) \to \R, \] 
is the secondary characteristic class (in the sense of Cheeger-Simons) of the local system $\rho$.
\end{enumerate}

His method could be used in our setting to prove the Hodge-theoretic claims about the sequence
\[ 0 \to H_1(\del X, \Q(1)) \to H^2((X, \del X), \Q(1)) \to H^2(X, \Q(1)), \]
as $H^2((X, \del X), \Q(1)) \isom H^2(X^{\mathrm{min}}, \Q(1))$, for $X^{\mathrm{min}}$ the minimal compactification of the Hilbert modular variety. He does not relate this structure to Eisenstein series or study the related Galois representions.

We describe a conjecture (Conjecture \ref{conj:toroidal-chern}) about Chern classes on toroidal embeddings which ought to be useful in generalizing Looijenga's results to \'etale cohomology.

\begin{question}
Is it possible to prove Looijenga's results using Eisenstein series as opposed to Chern classes, analogous to our Section \ref{sec:weak-eis-periods}?
\end{question}
The main difficulty with this is in finding a useful analytic intepretation of $H^*(\mathcal{A}_g^{\mathrm{min}})^{\dual}$ in which certain Eisenstein series can be verified to be non-zero.

See also Esnault--Harris \cite{Esnault-Harris}, which studies the Galois-theoretic properties of the map $\pi^*$ with regards to these Chern classes in \'etale cohomology. In particular, over $p$-adic fields and with $\Q_l$-coefficients, $l \neq p$, they use Scholze's Hodge-Tate period map to produce analogs of the Goresky-Pardon lifts. They prove that these lifts are $G_{\Q_p}$-invariant, but see no reason for them to be $G_{\Q}$-invariant. The above conjecture would imply that, for $\mathcal{A}_g$, they are not always $G_{\Q}$-invariant. For Hilbert modular varieties, we have shown the following:
\begin{enumerate}
\item It is not possible to define $G_F$-invariant lifts of $\omega_i$ with $\Q_p$-coefficients, and questions of $G_{F_v}$-invariance, for $v \mid p$, are related to Leopoldt's conjecture for $F$.
%\item It is possible, by choosing an isomorphism $\bar{K_v} \isom \C$, to define a $G_{K_v}$-invariant Eisenstein series $E \in H^{2d-2}_{\Et}(Y_{\bar{K_v}}, \Q_l(2))$ ($l \neq p$), using Harder's $E \in H^{2d-2}(Y_{\C}, \Q(d))$. We can use this to define canonical lifts of $\omega_i$ to $H^2((X, \del X)_{\bar{K_v}}, \Q_l(1))$ by requiring that such lifts pair trivially with $E$.
\item If the map $\O^* \tensor \Z_l \to (\O/(p))^* \tensor \Z_l$ is nonzero, it is not possible to define $G_{F_v}$-invariant lifts of $\omega_i$ with $\Z_l$-coefficients.
\end{enumerate}

\subsection{Outline of the paper}

In Section \ref{sec:hilbert-moduli}, we recall how the space $Y(\C) = \H^d/\SL_2(\O)$ is the $\C$-points of a (singular) variety $Y/\Q$. After adding auxiliary level structure, we have a smooth, projective compactification $X/\Q$. In Section \ref{sec:boundary-top}, we describe the topology of the boundary $\del X = X - Y$.

In Section \ref{sec:eisenstein-cohomology}, we review the cohomology of $H^*(Y)$, in particular its Eisenstein cohomology $H^*(Y)_{\Eis} \subset H^*(Y)$. We introduce the Eisenstein series and the Hodge line-bundles. We describe the Hodge structure on $H^*(Y)_{\Eis}$, and relate it to Theorems \ref{thm:eis-periods-1} and \ref{thm:lower-rationality}. 

In Section \ref{sec:line-bundle-extn}, we consider line-bundles which becomes local systems when restricted to a circle of $\P^1$'s, and study associated extension classes of Galois modules and mixed Hodge structures. We apply this in Section \ref{sec:extn-classes} to prove Theorems \ref{thm:eis-periods-1}, \ref{thm:etale-extn-class}, \ref{thm:lower-rationality}, \ref{thm:lower-semisimplicity}.

In Section \ref{sec:gK-cohom}, we review Eisenstein cohomology via the framework of $(\g,K)$-modules.  We use this in Section \ref{sec:weak-eis-periods} to give an alternate proof of a weak form of Theorem \ref{thm:eis-periods-1}, and in Section \ref{sec:harder-section} to study the relationship of our theorems with Harder's notion of Eisenstein cohomology.

\subsection{Acknowledgements}
This paper constituted part of my PhD thesis. I'd like to thank my advisor, Akshay Venkatesh, for teaching me many things about motives and Eisenstein series, and for suggesting that I study Hilbert modular varieties.

\tableofcontents

\section{Hilbert modular varieties}\label{sec:hilbert-moduli}

Let $F$ be a totally real field, $d = [F : \Q]$, $\O_F \subset F$ the ring of integers. We assume the narrow class number of $\O_F$ is one. 

% $U$ be the group of totally positive units, $D_K = disc(\O/\Z)$, $\mathfrak{d}_K$ the different ideal. 
\subsection{Shimura varieties}
For $N \geq 3$, the congruence subgroup $\Gamma(N) = \{ g \in \SL_2(\O) \mid g \equiv I \mod N \}$ has no elements of finite order, and the space $\H^d/\Gamma(N)$ is well-known to be a complex manifold (see \cite{Freitag}, \cite{VanDerGeer}). The quotient by $\SL_2(\O_F/N\O_F)$, $\H^d/\SL_2(\O_F)$, is a (possibly singular) complex analytic space. The space $\H^d/\Gamma(N)$ is in fact the analytification of an algebraic variety defined over $\Q(\mu_N)$. We can see this by identifying $\H^d/\Gamma(N)$ with a connected component of the analytification of a Shimura variety as follows.

We consider the Weil restriction $\Res_{F/\Q} \GL_2$. There is a determinant map \[ \det \from \Res_{F/\Q}\GL_2 \to \Res_{F/\Q} \G_m, \] as well as an inclusion $\G_m \subset \Res_{F/\Q} \G_m$. Let $G := \det^{-1}(\G_m)$, $Z \subset G$ its center, $K_N := \{ g \in G(\wh{\Z}) \mid g \equiv 1 \mod N \}$. The group $K_N$ is an open compact subgroup of $G(\A^f)$. The condition $N \geq 3$ implies that $K_N$ is \emph{neat}.

Associated to the reductive group $G$ and the neat level structure $K_N$ is a Shimura variety $\Sh(G, K_N)$, a smooth quasi-projective variety over $\Q$. Its analytification $\Sh(G, K_N)^{\an}_{\C}$ has, as one of its connected components, the space $\H^d/\Gamma(N)$. Moreover, this connected component is ``defined over $\Q(\mu_N)$", i.e.\ is the analytification of a connected component of $\Sh(G, K_N)_{\Q(\mu_N)}$ (we fix an embedding $\sigma \from \Q(\mu_N) \into \C$). 

\begin{rmk}
We will discuss Hecke operators on $\H^d/\Gamma(N)$ in Section \ref{subsec:rat-eis-cohom}. However, these Hecke operators are not those for the Shimura variety associated to $G$, but for the Shimura variety associated to $\Res_{F/\Q} \GL_2$. This distinction is not very important for us, as we restrict to connected components anyways.
\end{rmk}

The group $H_N := K_N /K_1\cdot(K_N \cap Z(\A_{\Q}))$ acts on $\Sh(G, K_N)_{\Q(\mu_N)}$. There is an exact sequence
\[ 0 \to \SL_2(\O_F/N\O_F) \to H_N \to (\Z/N\Z)^* \to 0. \] Then the quotient $\Sh(G, K_N)_{\Q(\mu_N)}/H_N$ exists as a (possibly singular) quasi-projective variety over $\Q(\mu_N)$.  

For $N_1, N_2 \geq 3$, there are canonical isomorphisms \[ (\Sh(G, K_{N_1})_{\Q(\mu_{N_1})}/H_{N_1})_{\Q(\mu_{N_1 N_2})} \isom (\Sh(G, K_{N_2})_{\Q(\mu_{N_2})}/H_{N_2})_{\Q(\mu_{N_1 N_2})}. \] Choosing $N_1, N_2$ coprime, we find that this variety descends to a variety over $\Q$, which we denote $\Sh(G, K_1)_{\Q}$. Its analytification $(\Sh(G, K_1))^{\an}$ is equal to $\H^d/\SL_2(\O_F)$. 

%The subgroup of $H_N$ which preserves the connected component $\H^d/\Gamma(N)$ equals $\SL_2(\O_F/N\O_F)$.

Thus:
\begin{thm}\label{thm:canon-model}
\
\begin{enumerate}
\item There is a smooth algebraic stack $Y(1)$ over $\Q$ whose coarse space $[Y(1)]$ is a (possibly singular) variety over $\Q$, with \[ [Y(1)]_{\C}^{\an} \isom \H^d/\SL_2(\O_F). \] 
\item There exists a smooth, geometrically connected algebraic variety $Y(N)$ over $\Q(\mu_N)$ such that \[(Y(N)_{\C})^{\an} := (Y(N) \tensor_{\sigma} \C)^{\an} \isom \H^d/\Gamma(N). \]
\end{enumerate}
\end{thm}

%We will often write $Y$ instead of $Y_L$.

%\begin{rmk}
%Rapoport worked over $\Q(\mu_N)$, actually. TODO: we might need this to use this to see that different embeddings $K \into \R$ give the same variety.
%\end{rmk}

%Note that this isomorphism $(Y_{\C})^{\an} \isom \H^d/\Gamma(N)$ depends on the embedding $\sigma$. In particular, the map $\O_F \into \R^d$ used to define $\H^d/\Gamma(N)$ agrees with 

%We let $L = F(\mu_N)$, and define $Y := Y_L = Y_{\Q(\mu_N)} \times_{Spec(\Q(\mu_N))} Spec(L)$. 

\subsection{Arithmetic toroidal compactification}\label{subsec:arith-toroidal}
%\cite{Emerton-Reduzzi-Xiao}
The Shimura variety $\Sh(G, K_N)$ is well-known to represent the Hilbert-Blumenthal moduli problem (\cite{Rapoport}). In other words, for $\Q$-algebras $S$, $\Sh(G,K_N)(S)$ equals the set of abelian schemes $A/S$ of relative dimension $d$, with
\begin{enumerate}
\item[(i)] real multiplication $\O_F \into \End(A)$,
\item[(ii)] full level-$N$ structure $A[N] \isom (\Z/N\Z)^d$,
\item [(iii)]  a polarization condition which we will not recall here.
\end{enumerate} 
As a PEL-type Shimura variety, the results of Lan \cite{Lan} may be used to construct algebraic compactifications of $\Sh(G, K_N)$. We will use this to compactify $Y(N)$. 

\begin{rmk}
The compactification of the Hilbert-Blumenthal moduli problem to a proper algebraic space was first done for $\Z[1/N, \mu_N]$ by Rapoport \cite{Rapoport}. However, we found it convenient to use the more general results of Lan, as he proves
\begin{enumerate}
\item  conditions for compactification by projective algebraic varieties (as opposed to proper algebraic spaces) in \cite{Lan},
\item  compatibilities between these compactifications and the analytic toroidal compactifications in \cite{Lan-Comparison}.
\end{enumerate}
\end{rmk}

Define $\Cusps(\Gamma(N)) := \P^1(F)/\Gamma(N)$. For the moment, let $Y = Y(N)$. Then:
\begin{thm}\label{thm:alg-tor-cpt}
 There exists a smooth, projective variety $X$ over $\Q(\mu_N)$, together with an embedding $Y \subset X$. It satisfies the following properties:
\begin{enumerate}
\item $Y$ is Zariski dense in $X$, and the complement $\del X := X - Y$, with its reduced structure, is a simple normal crossings divisor in $X$.
\item There exists a semi-abelian scheme $A' \to X$ extending the universal abelian scheme $A \to Y$, so that the universal real multiplication $\O_F \to \End(A)$ extends to $\O_F \to \End(A')$.%, and so that $A|_{\del X}$ is an algebraic torus over $\del X$.
\item At the level of complex-analytic spaces, $Y^{\an}_{\C} \subset X^{\an}_{\C}$ may be identified with the toroidal compactification of $Y$ constructed in \cite{AMRT}, Ch. 3.1.
\item The variety $\del X$ equals $\coprod_{x \in \Cusps(\Gamma(N))} \del X_x$, where $\del X_x$ is a geometrically connected variety over $\Q(\mu_N)$.
\item Let $Z_i$ denote the irreducible components of the $\del X$ over $\Q(\mu_N)$. The connected components of each $Z_i$, $Z_i \cap Z_j$ are normal and geometrically connected. 
\end{enumerate}
\end{thm}

The variety $X$ depends on an auxiliary choice $\Sigma$ for each cusp $x \in \Cusps(\Gamma(N))$. Let \[ T(\Z) := \left\lbrace \begin{pmatrix} * & 0 \\ 0 & * \end{pmatrix} \in \Gamma(N) \right\rbrace. \] For the cusp $\infty \in \Cusps(\Gamma(N))$, the auxiliary choice is a $T(\Z)$-admissible rational polyhedral cone decomposition $\Sigma$ of $(\O_F \tensor \R)_{+} \cup \{ 0 \}$, where $(\O_F \tensor \R)_{+}$ denotes the cone of totally positive elements of $\O \tensor \R \isom \R^d$. The above theorem will hold for any sufficiently fine $\Sigma$ which is \emph{smooth} and \emph{projective}.

\begin{proof}
We essentially just take results in \cite{Lan}, apply them to $\Sh(G, K_N)$, base-change to $\Q(\mu_N)$, and restrict to the connected component $Y$ of $\Sh(G, K_N)_{\Q(\mu_N)}$.

(1) See \cite{Lan} Thm. 6.4.1.4 for the existence of the smooth toroidal compactfication of $Y$ by an algebraic space $X$ so that the complement $X - Y$ (with its reduced structure) is a normal crossings divisor. See (loc. cit.) Prop. 7.3.1.2, Thm. 7.3.3.4 for the fact that $\Sigma$ may be chosen so that $X$ is a projective variety. By the same proposition, we may assume that $\Sigma$ is sufficiently fine, so that the normal crossings divisor $\del X$ is in fact simple.

(2) See \cite{Lan} Thm. 6.4.1.4 for the extension of the semi-abelian scheme equipped with extra structures. More precisely, this theorem constructs an extension of the universal abelian scheme into a semi-abelian algebraic space. However, it can be verified that this extension is in fact a scheme in our case, for example by the results of (\cite{Faltings-Chai}, V.6), using the fact that the universal abelian scheme has unipotent (as opposed to quasi-unipotent) monodromy locally on $\del X$.

%\todo{Is it true that we can get what we need from Faltings-Chai? The theory of Neron models should show that extensions to semi-abelian schemes are unique, if they exist, along codimension 1 points.}

%  \url{https://mathoverflow.net/questions/8918/is-an-algebraic-space-group-always-a-scheme}, \url{https://mathoverflow.net/questions/178230/is-a-semiabelian-algebraic-space-a-scheme}

(3) See \cite{Lan-Comparison} for this result. The complex-analytic compactification will be described in Section \ref{sec:boundary-top}.

(4), (5) These follow from Theorem 5.1 of Rapoport \cite{Rapoport}, which identifies a formal neighborhood of $\del X$ in $X$ with a formal neighborhood in a ``locally toric variety", base-changed from $\Z$ to $\Q(\mu_N)$, whose connected components are indexed by $\Cusps(\Gamma(N))$. (4) follows immediately, while (5) can then be checked using standard toric techniques.

Technically, in order to use results of \cite{Rapoport}, we ought to prove the compatibility of Lan's compactifications with Rapoport's. We have not done this - however, it is not hard to see that (\cite{Rapoport} Theorem 5.1) may be proven for Lan's compactifications, using an analog of (\cite{Faltings-Chai} Theorem 5.7 (5)) for Hilbert modular varieties.

\end{proof}

\section{Topology of toroidal compactifications}\label{sec:boundary-top}
% What do we actually use from this section?
%3) Cell decomposition of $\del X$, so that we can produce $U \to H_1^{\Mot}((\del X)_{\Z}, \Q(0))$
%4) Verification that the $\Q$-structure on $H^{2d-1-i}(\H^d/P(\Z), \Q(d))$ agrees with the $\Q$-structure on $H_i(\del X, \Q(d))$ under $Res$.

In this section, we will recall the analytic construction of the toroidal compactification $X$ and describe the topology of its boundary divisor $\del X$. This compactification depends on a certain auxiliary choice $\Sigma$ for every cusp. We will focus on the connected component of $\del X$ corresponding to the standard cusp $\infty \in \Cusps(\Gamma(N))$, as the compactification at other cusps is similar. Abusing notation, we will denote this connected component by $\del X$. 

Associated to the cusp at infinity, we have the following subgroups of $\Gamma(N)$:
\[ P(\Z) = \left\lbrace \begin{pmatrix} * & * \\ 0 & * \end{pmatrix} \in \Gamma(N) \right\rbrace, \ T(\Z) = \left\lbrace \begin{pmatrix} u & 0 \\ 0 & u^{-1} \end{pmatrix} \in \Gamma(N) \right\rbrace, \ N(\Z) = \left\lbrace \begin{pmatrix} 1 & n \\ 0 & 1 \end{pmatrix} \in \Gamma(N) \right\rbrace, \]
fitting into an exact sequence
\[ 0 \to N(\Z) \to P(\Z) \to T(\Z) \to 0. \]
We will frequently identify $N(\Z)$ and $T(\Z)$ with the subgroups $\{ n \in \O_F \mid n \equiv 0 \text{ mod } N\O \} \subset \O_F$ and $\{ u \in \O_F^* \mid u \equiv 1 \text{ mod } N\O \} \subset \O_F^*$, respectively. % $N(\Z) \subset \O_F$, $T(\Z) \subset \O_F^*$.

\subsection{Toric varieties}
Let $V_{\Z} := N(\Z) \subset \O_F,\ V_{\R} = N(\Z) \tensor \R = \O_F \tensor \R \isom \R^d$. Let $V_+ \subset V_{\R}$ correspond to $\R_+^d \subset \R^d$. The group $T(\Z)$ acts on $V_{\Z}$, preserving the cone $C := V_+ \cup \{0\} \subset V_{\R}$, via $\begin{pmatrix}u & 0 \\ 0 & u^{-1}\end{pmatrix} \mapsto (u^2) \in \O_{> 0}^* \subset (\R_+)^d$.

We consider a $T(\Z)$-invariant fan $\Sigma$ on $V_{\Z}$, such that $\Sigma/T(\Z)$ is finite and $C = \bigcup_{\sigma \in \Sigma} \sigma$.  Moreover, we assume that the fan $\Sigma$ is smooth. %, i.e.\ the intersection $\sigma \cap V_{\Z}$ equals $\N \cdot v_1 + \cdots + \N \cdot v_k$, for $\{v_1,\ldots, v_k \}$ a subset of a basis of $V_{\Z}$. 
We let $\sigma_0$ denote the cone $\{0\} \subset C$.

\begin{rmk}
Such $\Sigma$ corresponds to a $T(\Z)$-admissible smooth rational polyhedral cone decomposition of $C$ (\cite{AMRT}, Ch. 2, Defn. 4.10), (\cite{Lan}, Defn. 6.1.1.14). Without the smoothness condition, such $\Sigma$ exists by the results of (\cite{AMRT}, Ch. 2). That $\Sigma$ may be assumed to be smooth is shown in (\cite{AMRT}, Ch. 3, Cor. 7.6).
\end{rmk}

\begin{lem}\label{lem:action-on-fan}
$T(\Z)$ acts freely on $\Sigma - \sigma_0$.
\end{lem}
\begin{proof}
The group $T(\Z)$ acts freely and properly discontinuously on the hyperbola \[ V_+^{N=1} = \{ (y_1,\ldots,y_d) \in \R_{>0}^d \isom V_+ \mid \prod y_i = 1 \} \] by Dirichlet's unit theorem. The proper discontinuity implies that, for any compact set $K \subset V_+^{N=1}$, the set $\{ g \in T(\Z) \mid g(K) \cap K \neq \emptyset \}$ is finite. In particular, for $K \neq \emptyset$, the subset $\{ g \in T(\Z) \mid g(K) = K \}$ is a finite subgroup of $T(\Z)$, hence trivial as $T(\Z)$ is torsion-free.

 As $\sigma \cap V_+^{N=1}$ is compact (and non-empty, by the assumption $\sigma \neq \sigma_0$), this implies that the collection of $g \in T(\Z)$ such that $g \cdot (\sigma \cap V_+^{N=1}) \cap (\sigma \cap V_+^{N=1})$ is non-zero is finite. In particular, the stabilizer of $\sigma$ in $T(\Z)$ is trivial.
%It would suffice to show that for any cone $\sigma \subset C$ of the form $\R_{\geq 0} v_1 + \cdots + \R_{\geq 0} v_k$, $\sigma \neq \sigma_0$, and any element $u \in T(\Z)$, there exists an $n$ such that $u^n \cdot \sigma \cap \sigma = \emptyset$. 
\end{proof}

Associated to the data $(V_{\Z}, \Sigma)$ is a toric variety $W_{\Sigma}$\footnote{As the fan $\Sigma$ is not finite, $W_{\Sigma}$ is not of finite type over $\C$, so ``variety" is a misnomer. However, as it is locally of finite type, we need not worry much about this distinction.}, equipped with an action of the group $T(\Z)$. As a toric variety, $W_{\Sigma}$ has an action of the algebraic torus $T'$ with cocharacter lattice $V_{\Z}$, as well as a basepoint $e \in W_{\Sigma}$ whose $T'$-orbit is Zariski dense in $W_{\Sigma}$\footnote{Our convention for toric varieties is that, if $v \in C \cap V_{\Z}$, then $\lim_{t \to \infty} v(t) \cdot e$ exists in $W_{\Sigma}(\C)$. This choice is not the standard one, but is convenient for our notation.}. Thus there is a canonical embedding $T' \subset W_{\Sigma}$ whose image is Zariski dense. Passing to complex-analytic spaces, we obtain an embedding  \[ \C^d/N(\Z) \subset W_{\Sigma}(\C). \] 
Note that $\C^d/N(\Z)$ is dense in $W_{\Sigma}(\C)$, and that, as the fan $\Sigma$ is smooth, $W_{\Sigma}(\C)$ is a complex manifold.

Consider $\H^d/N(\Z) \subset \C^d/N(\Z) \subset W_{\Sigma}(\C)$. There is a map $N \from \C^d/N(\Z) \to \R$, given by $(z_1,\ldots,z_d) \mapsto \prod y_i$, for $y_i = \Im(z_i)$. We define
\[ \wt{S} := \bar{\H^d/N(\Z)} - \bar{N^{-1}(0)} \cap \bar{\H^d/N(\Z)}, \]
where all closures are taken inside of $W_{\Sigma}(\C)$. The map $N$ extends to a map $N \from \wt{S} \to (0, \infty]$. We will use superscripts to denote the intersection of various sets with the preimages of intervals and points along $N$, e.g.\ $\wt{S}^{(\epsilon^{-1}, \infty)} = \wt{S} \cap N^{-1}((\epsilon^{-1}, \infty))$.

\begin{comment}We want to compactify $\H^d/N(\Z)$ ``at infinity", by adding the limits in $W_{\Sigma}(\C)$ of sequences of points $p_i$ with $\lim_{i \to \infty} N(p_i) = \infty$. However, this is not quite right - we only want to add those points $p \in W_{\Sigma}(\C)$ such that for \emph{all} sequences $p_i \in \H^d/N(\Z)$, with $p = \lim p_i$, we have $\lim_{i \to \infty} N(p_i) = \infty$. This ensures that the map $N$ extends, on this closure, to a map to $(0,\infty]$.\end{comment}

%We now require that the fan $\Sigma$ is $T(\Z)$-admissible in the sense of (\cite{AMRT}, Ch. 3). In particular, we require that $T(\Z)$ acts freely on the set $\Sigma - \sigma_0$. Such $\Sigma$ exist by the results of (\cite{AMRT}, Ch. 2). For such $\Sigma$, the analytic toroidal compactification $X(\C)$ of $Y(\C)$ is defined in (\cite{AMRT}, Ch. 3).

%an algebraic toroidal compactification $X = X_{\Sigma}$ satisfying the results of Section \ref{subsec:arith-toroidal}, whose analytification $X(\C)$ agrees with the analytic toroidal compactification constructed in \cite{AMRT}. The existence of such $\Sigma$ is guaranteed by the results of (\cite{AMRT}, Ch. 2) \todo{Clarify why such $\Sigma$ exists}. 

We state some facts describing how $\wt{S}/T(\Z)$ is glued to $Y(\C)$ in the construction of $X(\C)$:
\begin{thm}\label{thm:basic-tor}
\
\begin{enumerate}
\item $\wt{S}$ is an open submanifold of $W_{\Sigma}(\C)$, hence a complex manifold.
\item The group $T(\Z)$ acts freely and properly discontinuously on $\wt{S}$, with quotient $S := \wt{S}/T(\Z)$, and $N$ descends to a map $N \from S \to (0,\infty]$.
\item For $\epsilon > 0$ small enough, the map $S^{(\epsilon^{-1}, \infty)} \subset \H^d/P(\Z) \surj \H^d/\Gamma(N)$ is an embedding. 
\item For $\epsilon > 0$ small enough, there is an embedding $S^{(\epsilon^{-1}, \infty]} \into X(\C)$ extending the embedding of (3), identifying $S^{(\epsilon^{-1}, \infty]}$ with an open neighborhood of $\del X(\C)$.
\item $S^{\infty} = \cap_{\epsilon \to 0} S^{(\epsilon^{-1}, \infty]} = \del X(\C)$. 
\item $S^{(\epsilon^{-1}, \infty]}$ has a strong deformation retract to $\del X(\C)$, as does $S^{[\epsilon^{-1},\infty]}$.
%\item $S^{[1,\infty]}$ has a strong deformation retract to $\del X(\C)$.
\end{enumerate}
\end{thm}
\begin{proof}

(1), (3)-(5) follow easily from the construction of analytic toroidal compactifications of Hilbert modular varieties, see (\cite{AMRT}, Ch. 3.1). 

(2) The orbits of $\C^d/N(\Z)$ on $W_{\Sigma}(\C)$ are in bijection with $\Sigma$, compatible with the action of $T(\Z)$. As $T(\Z)$ acts freely on $\Sigma - \sigma_0$, $T(\Z)$ acts freely and properly discontinuously on $W_{\Sigma}(\C) - \C^d/N(\Z)$, hence on $\wt{S} - \H^d/N(\Z)$. The group $T(\Z)$ acts freely and properly discontinuously on $\H^d/N(\Z)$ as well (similarly to Lemma \ref{lem:action-on-fan}, this follows from Dirichlet's unit theorem.).

%By Lemma \ref{lem:action-on-fan}, the only possibility of $T(\Z)$ acting non-freely on $W_{\Sigma}(\C)$ is via the action of $T(\Z)$ on the orbit corresponding to $\sigma_0$, equal to $\C^d/N(\Z)$ itself. 

(6) We use the fact that $\del X(\C) \subset X(\C)$ is a strong neighborhood deformation retract. This follows from the following two results:
\begin{thm}[\cite{Lojasiewicz}]
Let $Z$ be a closed analytic subset of a complex analytic space $X$. Then $(X, Z)$ is a CW pair.
\end{thm}

\begin{lem}[\cite{Strom}, pg. 105]
Closed (Hurewicz) cofibrations $Z \into M$ are \emph{strong neighborhood deformation retracts}: there exists a function $\phi \from M \to [0, 1]$ such that $Z = \phi^{-1}(0)$, $U = \phi^{-1}([0,1))$ is open in $M$,  and $Z$ is a strong deformation retract of $U$.
\end{lem}
Note that CW pairs are closed Hurewicz cofibrations.

There is also a strong deformation retract of $S^{[1,\infty]}$ to $S^{[t,\infty]}$, $t > 1$. By (5), for $t$ large enough, $S^{[t,\infty]}$ is contained in the neighborhood $U$ which has a strong deformation retract to $\del X(\C)$. Composing these deformation retracts gives the result.

\end{proof}

\subsection{Topology of real points}

%We want to study the ``non-negative real points" of $S$ and $\del X$. 
We define $\wt{S}_{\R}$ to be the closure of $(i \R_{>0})^d \subset \H^d/N(\Z) \subset \wt{S}$ inside $\wt{S}$, and define $S_{\R} := \wt{S}_{\R}/T(\Z)$. We want to understand the topology of $S_{\R}$ and relate this to the topology of $S$.

Let $\sigma$ be a top-dimensional cone of $\Sigma$, with $\sigma \cap V_{\Z}$ generated by $v_1,\ldots, v_d$. Define \begin{align*}
C_{\sigma} &= \{ \sum t_i v_i \mid t_i \in [0,1] \text{ for } i = 1,\ldots, d \} \subset V_{\R} \\
C'_{\sigma} &= \{ \sum t_i v_i \mid t_i \in [0,1] \text{ for } i = 1,\ldots, d,\ t_j = 1 \text{ for some $j$ } \}.
\end{align*}

We consider the subspace $(\cup_{\dim \sigma = d} C_{\sigma}) - \{0\} \subset V_{\R}$.

%\todo{Can one use a $T(\Z)$-invariant polarizing function $h$ to give a canonical deformation retract $S_{\R}^{\infty} \subset S_{\R}$?}

\begin{prop}\label{prop:real-top-cover}
\

\begin{enumerate}
\item There is a homeomorphism $(\cup_{\dim \sigma = d} C_{\sigma}) - \{0\} \isom \wt{S}_{\R}$.
\item This homeomorphism identifies $\wt{S}_{\R}^{\infty}$ with $\cup_{\dim \sigma = d} C'_{\sigma}$. 
\item The pair $(\wt{S}_{\R}, \wt{S}_{\R}^{\infty})$ is homeomorphic to the pair $((0,1] \times \R^{d-1}, \{1 \} \times \R^{d-1})$.  
\end{enumerate}
\end{prop}
\begin{proof}
(1) We only sketch the proof, as the details are tedious. For any $\sigma \in \Sigma$, we have a toric chart $U_{\sigma} \subset W_{\Sigma}$. As the fan $\Sigma$ is smooth, if $\dim \sigma = d$ then $U_{\sigma} \isom \A^d$, canonically up to permuting the coordinates. If we order the generators $v_1,\ldots, v_d$ of $\sigma \cap V_{\Z}$, then the isomorphism $\A^d \isom U_{\sigma}$ is canonical. Therefore we have a cube $Z_{\sigma} := [0,1]^d \subset \A^d(\C) \isom U_{\sigma}(\C) \subset W_{\sigma}(\C)$. 

The map is given by sending $C_{\sigma} \isom [0,1]^d$ to $Z_{\sigma} \isom [0,1]^d$, via the map sending $x \mapsto 1-x$ on each factor of $[0,1]^d$, ordering these coordinates by the same choice of generators of $\sigma \cap V_{\Z}$. The facts that
\begin{itemize}
\item[(i)] these maps $C_{\sigma} \to W_{\Sigma}(\C)$ glue into a map $(\cup_{\dim \sigma = d} C_{\sigma}) \to W_{\Sigma}(\C)$,
\item[(ii)] the map $(\cup_{\dim \sigma = d} C_{\sigma}) - \{0 \} \to W_{\Sigma}(\C)$ is a homeomorphism onto its image,
\item[(iii)]  the image of this map is $\wt{S}_{\R}$,
\end{itemize}
 are a combinatorial exercise.

(2) follows easily from the description of (1). 

(3) follows from (2) by scaling $\cup_{\dim \sigma = d} C'_{\sigma}$ onto the hyperbola $V_+^{N=1} \subset V_{\R}$.
\end{proof}

\begin{comment}
\begin{rmk}
The map of (1) is not $T(\Z)$-equivariant for the action of $T(\Z)$ on $(\cup_{\dim \sigma = d} C_{\sigma}) - \{0\}$ from restricting the linear action on $V_+$. However, there are isomorphism $\phi_{\sigma, g} \from C_{\sigma} \isom C_{g \cdot \sigma}$ for all $\sigma \in \Sigma$, $g \in T(\Z)$, giving a different $T(\Z)$-action on $(\cup_{\dim \sigma = d} C_{\sigma}) - \{0\}$, such that the map of (1) is a $T(\Z)$-equivariant homeomorphism. In other words, except for the fact that $C_{\sigma} - \{ 0 \}$ is not a cell, there is roughly the structure of a $T(\Z)$-CW complex on $\wt{S}$.
\end{rmk}
\end{comment}

\begin{cor}\label{cor:real-top}
The inclusions $S^{1}_{\R},\ S^{\infty}_{\R} \subset S^{[1,\infty]}_{\R}$ are weak homotopy equivalences.
\end{cor}
\begin{proof}
By Theorem \ref{thm:basic-tor} (2), the map $\wt{S}_{\R} \to S_{\R}$ is a covering space map. Applying Proposition \ref{prop:real-top-cover}, we find that the spaces $S^{1}_{\R},\ S^{\infty}_{\R},\ S^{[1,\infty]}_{\R}$ all have contractible universal covers $\wt{S}^1_{\R},\ \wt{S}^{\infty}_{\R},\  \wt{S}^{[1,\infty]}_{\R}$, and the inclusions induce an isomorphism of their fundamental groups.
\end{proof}

%\item The map $N \from \wt{S}_{\R} \to (0,\infty]$ is a fiber bundle with fiber $\R^{d-1}$.
%\item The map $N \from S_{\R} \to (0,\infty]$ is a fiber bundle with fiber $\R^{d-1}/T(\Z)$.
%\item The inclusions $S^{1}_{\R},\ S^{\infty}_{\R} \subset S^{[1,\infty]}_{\R}$ are homotopy equivalences with neighborhood deformation retracts.

\begin{comment}
For every $\sigma \in \Sigma$, with the monoid $\sigma \cap V_{\Z}$ being generated by vectors $v_1,\ldots, v_k$, consider the cube $C_{\sigma} = \{ \sum t_i v_i \mid t_i \in [0,1] \} \subset V_{\R}$. For any top-dimensional cone $\sigma \in \Sigma$, we have a toric chart $\A^d = U_{\sigma} \subset W_{\Sigma}$. We consider the cube $C'_{\sigma} = [0,1]^d \subset W_{\Sigma}(\C)$. There is a $T(\Z)$-equivariant homeomorphism $\cup_{\dim \sigma = d} C'_{\sigma} = \cup_{\sigma} C_{\sigma}$, given by the map $C_{\sigma} \to C'_{\sigma}$ which reverses the intervasl $[0,1]$. That this map is a homeomorphism can be verified combinatorially. Also $\cup_{\dim \sigma = d} C'_{\sigma} \subset W_{\Sigma}(\C)$ is equal to the closure of $(i \R)^d \subset W_{\Sigma}(\C)$.

We obtain $\cup_{\sigma} C_{\sigma} - \{ 0 \} \isom \wt{S}_{\R}$.

Moreover, $\cup_{\sigma} C_{\sigma} - \{ 0 \}$ is homeomorphic to $[0,\infty) \times \R^{d-1}$.
\end{comment}

Toric varieties can be evaluated on the non-negative real numbers $\R_{\geq 0} \subset \C$ (see \cite{Cox}, 12.2). We have $W_{\Sigma}(\R_{\geq 0}) \subset W_{\Sigma}(\C)$, and maps 
\[ p \from W_{\Sigma}(\C) \surj W_{\Sigma}(\R_{\geq 0}),\  s \from W_{\Sigma}(\R_{\geq 0}) \into W_{\Sigma}(\C), \]
satisfying $p \circ s = \mathrm{id}$, where $s$ is the natural inclusion and $p$ is the quotient by the action of $(S^1)^d \isom N(\R)/N(\Z)$, the maximal compact subgroup of the torus $\C^d/N(\Z)$. These maps are constructed in a more general setting in (\cite{Cox}, 12.2). Concretely, these maps are the continuous extensions of maps $\ImP \from \C^d/N(\Z) \surj (i\R)^d$ and $(i \R)^d \subset \C^d/N(\Z)$. 

%\begin{lem}
%The intersection $W_{\Sigma}(\R_{\geq 0}) \cap \wt{S}$ equals $\wt{S}_{\R}$.
%\end{lem}
% $W_{\Sigma}(\R_{\geq 0})$ is the closure of $(i\R)^d$, while $\wt{S}_{\R}$ is a subset of the closure of $(i\R_{>0})^d$.

The intersection $W_{\Sigma}(\R_{\geq 0}) \cap \wt{S}$ equals $\wt{S}_{\R}$. Thus the maps $p$ and $s$ restrict/descend to maps
\[ p \from S \surj S_{\R},\ s \from S_{\R} \into S. \]
Moreover, these maps are compatible with the maps $N \from S \to (0, \infty],\ N \from S_{\R} \to (0,\infty]$ in the obvious way. Thus we obtain maps
\[ p \from S^{\infty} \surj S^{\infty}_{\R},\ s \from S^{\infty}_{\R} \into S^{\infty}. \]

\begin{comment}
On $\del X(\C) = S^{\infty}$, $s$ induces split injections
\begin{align*} 
s_* \from \pi_1(S_{\R}^{\infty}) \into \pi_1(S^{\infty}), \\
s_* \from H_*(S^{\infty}_{\R}) \into H_*(S^{\infty}).
\end{align*}
\end{comment}

\subsubsection{A CW-structure on $S^{\infty}_{\R}$}
We put a CW structure on $S^{\infty}_{\R}$ as follows. 

The orbits of $\C^d/N(\Z)$ on $\wt{S}^{\infty}$ are in bijection with $\Sigma -  \sigma_0$. Moreover, the closure of the orbit associated to $\sigma$ is a toric variety we will denote $W_{\sigma}$. Standard toric techniques show that $W_{\sigma}$ is the toric variety associated to a smooth, proper fan $\Sigma_{\sigma}$ on $V_{\sigma} := V_{\Z}/(\Span_{\R} \sigma \cap V_{\Z})$. Defining $D_{\sigma} := W_{\sigma}(\R_{\geq 0})$, we find:
\begin{lem}\label{lem:cw-str-univ}
\

\begin{enumerate}
\item $D_{\sigma}$ is homeomorphic to a closed disk of dimension $\dim D_{\sigma} = d - \dim \sigma$.
\item $\wt{S}^{\infty}_{\R}$ is a CW complex with cells $D_{\sigma}$, $\sigma \in \Sigma - \sigma_0$.
\end{enumerate}
\end{lem}
\begin{proof}
(1) The claim about dimension is clear, as each torus orbit is of the form $(\C^*)^{d - \dim \sigma}$, with positive real points $(\R_{\geq 0})^{d - \dim \sigma}$.

That $D_{\sigma}$ is homeomorphic to a closed disk follows from the theory of the moment map for proper toric varieties, see (\cite{Fulton} 4.1), where $W_{\sigma}(\C)/(S^1)^d$ is identified with a compact convex polytope, the dual polytope. However, an alternate proof follows from the same construction used to prove Proposition \ref{prop:real-top-cover} using the fact that for the proper fan $\Sigma_{\sigma}$, $V_{\sigma, \R} = \cup_{\sigma' \in \Sigma_{\sigma}} \sigma'$.

(2) follows from the following easy facts:
\begin{enumerate}
\item[(i)] $\bigcup_{\sigma \in \Sigma - \sigma_0} \interior D_{\sigma} = \wt{S}^{\infty}_{\R}$ is a (set-theoretic) disjoint union,
\item [(ii)] $\del D_{\sigma} = \bigcup_{\sigma \subset \sigma'} \interior D_{\sigma'}$.
\end{enumerate}
%\todo{Is there a good criterion to apply here?}
\end{proof}

Taking the quotient of $\wt{S}^{\infty}_{\R}$ by $T(\Z)$, we obtain:
\begin{prop}\label{prop:cw-str-boundary}
\ 
\begin{enumerate}
\item $S^{\infty}_{\R}$ is a CW complex with cells $D_{\sigma}$, for $\sigma \in (\Sigma - \sigma_0)/T(\Z)$.
%\item The preimage $p^{-1}(\interior D_{\sigma})$ is a complex torus $(\C^*)^{d - \dim \sigma}$, with \[ s(\interior D_{\sigma}) \isom (\R_+)^d. \]
\item The preimage $p^{-1}(D_{\sigma}) \subset S^{\infty}$ is (the $\C$-points of) the proper smooth toric variety $W_{\sigma}$.
\item The subvarieties $W_{\sigma}$, $\dim \sigma = 1$, are the irreducible components of  $\del X_{\C}$.
\item The map $s_* \from \pi_1(S^{\infty}_{\R}) \to \pi_1(S^{\infty})$ is an isomorphism.
\end{enumerate}
\end{prop}
\begin{proof}
(1), (2), (3) follow easily from Lemma \ref{lem:cw-str-univ} and the fact that, by Lemma \ref{lem:action-on-fan}, $T(\Z)$ acts freely on $\Sigma - \sigma_0$.
%For (1) and (2), we simply take the closures of orbits of $(\C^*)^d$ on $\del W_{\Sigma}(\C) := W_{\Sigma}(\C) - \C^d/N(\O)$, and look at their image under $p \from S^{\infty} \to S^{\infty}_{\R}$. 

(4) The fiber of $p$ over $D_{\sigma}$ is identified with the compact torus $T_{\sigma} := \frac{V_{\R}}{\vecspan_{\R}(\sigma) + V_{\Z}}$. It is possible to make $T_{\sigma}$ into a CW complex such that the maps $T_{\sigma} \surj T_{\sigma'}$, for $\sigma \subset \sigma'$, are cellular (see \cite{Fischli}). This makes $S^{\infty}$ a CW complex whose cells are products of $D_{\sigma}$ with a cell of $T_{\sigma}$.

Since $s$ is a section of $p$, the map $s_* \from \pi_1(S^{\infty}_{\R}) \to \pi_1(S^{\infty})$ is a split injection. But the only 1-cells in $S^{\infty}$ are of the form $D_{\sigma}$, which is a cell of $S^{\infty}_{\R}$. Therefore $s_*$ is also surjective.
\end{proof}

\begin{cor}\label{cor:fund-group-bdry}
\
\begin{enumerate}
\item $\pi_1(S^{(\epsilon^{-1}, \infty)}) = P(\Z)$,
\item $\pi_1(S^{(\epsilon^{-1}, \infty]}) = T(\Z)$.
\end{enumerate}
\end{cor}
\begin{proof}
(1) is clear.

(2) follows from $\pi_1(S^{\infty}_{\R}) \isom \pi_1(S^{\infty}) \isom \pi_1(S^{(\epsilon^{-1}, \infty]})$, where the first isomorphism follows from Proposition \ref{prop:cw-str-boundary}, and the second follows from the homotopy equivalence $S^{\infty} \subset S^{(\epsilon^{-1}, \infty]}$.
\end{proof}

Proposition \ref{prop:cw-str-boundary}, or more precisely its proof, has the following consequence which is essential for us:
\begin{cor}\label{boundary-circles}
Let $u \in H_1(\del X(\C), \Z)$. There exists a ``circle of $\P^1$'s" mapping to $\del X$, \[Z = (\coprod_{i \in \Z/n\Z} \P^1)/(\infty_i = 0_{i+1}) \to \del X, \] such that $u$ is in the image of $H_1(Z(\C)) \to H_1(\del X(\C))$.
\end{cor}
\begin{proof}
We use the CW structure on $\del X(\C)$ described in Proposition \ref{prop:cw-str-boundary}. We lift $u \in H_1(\del X(\C), \Z)$ to $u \in \pi_1(\del X(\C), x)$ for some 0-cell $x \in \del X(\C)$. Every element $u \in \pi_1(\del X(\C), x)$ is represented by a loop of the form $\bigcup_{i \in \Z/n\Z} D_{\sigma_i}$ for some 1-cells $D_{\sigma_i} \isom [0,\infty]$ (choosing orientations on the $D_{\sigma_i}$ correctly), such that $\del_0 D_{\sigma_i} = \del_{\infty} D_{\sigma_{i-1}}$ for all $i \in \Z/n\Z$. The preimage $p^{-1}(D_{\sigma})$ is (the $\C$-points of) a proper 1-dimensional toric variety, hence is isomorphic to $\P^1$, identifying $D_{\sigma_i}$ with $[0,\infty] \subset \P^1(\C)$. Therefore the loop $u$ is in the image of a map $(\coprod_{i \in \Z/n\Z} \P^1)/(\infty_i = 0_{i+1}) \to \del X(\C)$.
\end{proof}

\section{Eisenstein cohomology}\label{sec:eisenstein-cohomology}
In this section, we describe the Eisenstein cohomology of $\H^d/\SL_2(\O_F)$. In Section \ref{subsec:hodge-bundle} we will discuss certain line-bundles on the variety $Y_F$ defined over $F$. In the remainder of the section, we will only need the complex variety $Y_{\C}$. Our goal is to describe the Eisenstein series which contribute to the singular cohomology of $Y(\C)$, and relate their periods to mixed Hodge theory.

\subsection{Hodge line-bundles}\label{subsec:hodge-bundle}
In this section, we let $Y$ denote the variety $Y(N),$ defined over $L := F(\mu_N)$, as in Theorem \ref{thm:canon-model}. We fix an embedding $L \into \C$, so that we may consider the $\C$-points/analytification of the variety $Y_L$.

\subsubsection{Moduli-theoretic definition}
We recall the definition of the Hodge bundles in terms of the Hilbert-Blumenthal moduli (see Section \ref{subsec:arith-toroidal}). 

Consider the Hodge bundle $\omega := e^*\Omega^1_{A/Y}$ on $Y$, where $A \to Y$ is the universal abelian scheme and $e \from Y \to A$ is the identity section. Consider a $F$-algebra $S$, and suppose $A/S$ is an abelian scheme of relative dimension $d$, with real multiplication by $\O_F$. The Galois automorphisms $\{ \sigma_1, \ldots, \sigma_d \} = \Gal(F/\Q)$ give maps $\sigma_i \from \O_F \to F$. The eigenspace decomposition for the action of $\O_F$ on $\omega_S$ defines $\omega_S = \oplus \L_i$ for line-bundles $\L_i$ on $S$, with $\O_F$ acting via $\sigma_i$ on $\L_i$. In particular, $\omega = \oplus_i \L_i$ over $Y_L$.

%TODO: Comment on the fact that $\omega$ is not really defined over the coarse space, only over the stack, due to $\pm 1 \in \SL_2(\O_F)$.

The analytification $\L_i^{\an}$ has a canonical section on the universal cover $\H^d$ of $Y^{\an}$, described as follows: an abelian variety $A_{\tau} = (\O_F \tensor \C)/(\O_F + \tau \O_F)$ over $\C$ has $e^*\Omega^1_{A_{\tau}/\C} = \Hom_{\C}(\O_F \tensor_{\Z} \C, \C) = \Hom_{\C}(\R^d \tensor_{\R} \C, \C)$, and $\alpha_i := e^*(dw_i)$ is the dual basis element to the i-th coordinate (where $w_i$ is the $i$-th coordinate map $\O_F \tensor \C \to \C$. We  write $\alpha_i := e^*(dw_i)$, and will not refer to the coordinates on $A$ again. Holomorphic sections of $\L_i^{\an}$ pull back to $\Gamma(N)$-invariant sections $f(z)\alpha_i$ with $f(z)$ a holomorphic function on $\H^d$. 
%To avoid confusing the coordinates $(z_i)$ on $A$ with the coordinates $(z_i)$ on $\H^d$, we

The section $\alpha_i$ transforms as $g^*(\alpha_i) = (c_i z_i + d_i)^{-1} \alpha_i$ under the action of $G(\R) = \SL_2(\R)^d$ on $\H^d = \{(z_1,\ldots, z_d)\}$. Via this action, the Lie algebra $\mathfrak{g} = \Lie G(\R)$ acts on the $C^{\infty}$-sections of $\L_i^{\an}$. 

The line-bundles $\L_i$ have a \emph{canonical extension} (see \cite{Faltings-Chai}, VI.4) to line-bundles $\bar{\L_i}$ on $X$. This comes from the fact that there is a semi-abelian scheme $A \to X$ with real multiplication by $\O_F$ such that $e^*\Omega^1_{A/X}$ is a rank $d$-vector bundle on $X$, which can be decomposed into eigenspaces exactly as we did for $A/Y$. Near the cusp $\infty$, there is an alternate description in terms of $\mathfrak{n} = \Lie N(\R)$-invariant sections of $\L_i$ (loc. cit.). By invariant, we mean those sections on which $\mathfrak{n}$ acts as zero. Identifying $\mathfrak{n} = \oplus \R \cdot \frac{\del}{\del x_i}$, with $x_i = \Re(z_i)$, $\mathfrak{n}$ acts on a section $f(z) \alpha_i$ of $\L_i^{\an}$ by taking derivatives of the function $f(z)$, as $\alpha_i$ is $N(\R)$-invariant. Then \[ \bar{\L_i}^{\an}(U) = \O(U) \cdot \L_i^{\an}(U - (\del X \cap U))^{\mathfrak{n}} \subset \L_i^{\an}(U - (\del X \cap U)). \]

%In particular, for any $U \subset X$, we have $\bar{\L_i}(U) \subset \L_i(U - U \cap \del X)$. 
%Concretely, we extend $\Omega^1_Y$ to $\Omega^1_X(\log \del X)$ and extend $\L_i$ as a sub-bundle $\bar{\L_i} \subset \Omega^1_X(\log \del X)$. 

\subsubsection{Restriction to the boundary}

We consider the restriction $\bar{\L_i}|_{\del X}$. We will show the following result:

\begin{prop}\label{hodge-on-boundary}
\

\begin{enumerate}
\item The analytification of $\bar{\L_i}|_{\del X}$ is the scalar extension of a rank 1 $\C$-local system on $\del X(\C)$, corresponding to a representation $\rho_i \from \pi_1(\del X(\C)) \to \C^*$.
\item Via the isomorphism $\pi_1(\del X(\C)) \isom T(\Z)$, the representation $\rho_i$ is equal to
\[ T(\Z) \subset \O_F^* \nmto{\sigma_i} L^* \subset \C^*. \]
\end{enumerate}
\end{prop}

We have some preparatory lemmas. 

\begin{lem}\label{lem:canonical-extn}
\
\begin{enumerate}
\item Near any point on $\del X$, the section $\alpha_i$ is a non-vanishing section of the analytic line-bundle $\bar{\L_i}^{\an}$.
\item The line-bundle $\bar{\L_i}$ is trivial when restricted to any irreducible component of $\del X$.
\end{enumerate}
\end{lem}
\begin{proof}
(1) For any $U \subset Y$, it is easy to that $\L_i(U)^{\an,\mathfrak{n}} \subset \C \alpha_i$, and that for any open set $U$ with trivial fundamental group, they are equal. For $\alpha_i$ is $\mathfrak{n}$-invariant, and if we had 2 $\mathfrak{n}$-invariant sections, their ratio $f$ would be an $\mathfrak{n}$-invariant meromorphic function. At any point where $f$ is holomorphic, this implies that $\del f/\del x_i = 0$ for $i = 1, \ldots, d$, which implies that $f$ is constant.

(2) The section $\alpha_i$ of $\bar{\L_i}^{\an}$, while not globally defined, is defined on any irreducible component of $\del X$. For these irreducible components are simply connected, by the following remark, and hence an open neighborhood of them in $X(\C)$ is as well.
\end{proof}

\begin{rmk}Note that the irreducible components of $\del X_{\bar{L}}$ are proper toric varieties by Proposition \ref{prop:cw-str-boundary}. This implies that  their complex points are simply-connected (\cite{Fulton}, 3.2).
\end{rmk}
%, so that homologically trivial line-bundles are trivial. Thus $Pic^0(\del X)(\C)$ is the group of homologically trivial line-bundles on $\del X$.

For a variety $Z$ over $L$, let $\Pic^0(Z)$ denote the group of line-bundles on $Z$ whose restriction to each irreducible component of $Z$ is trivial. 

\begin{lem}\label{lem:line-bundle-rational-var}
Let $Z$ be a proper variety over a field $L \subset \C$. Suppose that:
\begin{enumerate}
\item the irreducible components $Z_i$ of $Z$ are normal and geometrically irreducible,
\item the intersections $Z_i \cap Z_j$, $Z_i \cap Z_j \cap Z_k$, etc., of distinct irreducible components are transverse (i.e.\ the scheme-theoretic intersection is reduced),
\item the connected components of each $Z_i \cap Z_j$ are geometrically connected.
\end{enumerate} Then
\[ \Pic^0(Z) \isom \ker(H^1(Z(\C), L^*) \to \oplus_i H^1(Z_i(\C), L^*)). \]
In particular, every line-bundle in $\Pic^0(Z)$ is the scalar extension of a rank 1 $L$-local system on $Z(\C)$, corresponding to a representation $\pi_1(Z(\C)) \to L^*$.
\end{lem}
\begin{proof}
%We consider $\del X$ as a variety over $K$.
Using the normalization map $p \from \wt{Z} \to Z$, we have an exact sequence
\[ H^0(\wt{Z}, \G_m) \to H^0(Z, p_*(\G_m)/\G_m) \to H^1(Z, \G_m) \to H^1(\wt{Z}, \G_m). \] By the assumption that the irreducible components are themselves normal, $\wt{Z} \isom \coprod Z_i$, and so $H^1(\wt{Z}, \G_m) = \oplus_i \Pic(Z_i)$. This implies \begin{align}\label{a1} \ker(\Pic(Z)(L) \to \oplus_i \Pic(Z_i)(L)) = H^0(Z, p_*(\G_m)/\G_m)/H^0(\wt{Z}, \G_m).
\end{align} A similar argument for the constant sheaf $\underline{L}^*$ in the classical topology gives \begin{align}\label{a2} \ker(H^1(Z(\C), L^*) \to \oplus_i H^1(Z_i(\C), L^*)) = H^0(Z(\C), p_*(L^*)/L^*)/H^0(\wt{Z}(\C), L^*).
\end{align}

Let $\phi_{i} \from Z_i \into Z,\ \phi_{ij} \from Z_i \cap Z_j \into Z$, etc., denote the inclusions of irreducible components and their intersections. \begin{claim}
There is an exact sequence of sheaves
\[ 0 \to \G_m \to \oplus_i \phi_{i,*} \G_m \to \oplus_{i, j} \phi_{ij,*} \G_m \to \oplus_{i, j, k} \phi_{ijk,*} \G_m \to \ldots. \]
\end{claim}
\begin{proof}
Via the inclusion $\G_m = \O_Z^* \subset \O_Z$, it is not hard to see that it suffices check the analogous statement for the structure sheaf $\O_Z$. This follows easily from the exactness of the sequences
\[ 0 \to \O_Z \to \oplus_i \O_{Z_i} \to \oplus_{i, j} \O_{Z_i \cap Z_j}, \]
\[ \O_{Z_i} \to \oplus_{j} \O_{Z_i \cap Z_j} \to \oplus_{j, k} \O_{Z_i \cap Z_j \cap Z_k} \]
etc.

\begin{comment}
The exactness of $0 \to \O_Z \to \oplus_i \phi_{i,*} \O_Z \to \oplus_{i, j} \phi_{ij,*} \O_Z$ is clear (functions on $Z$ are the functions on the irreducible components which agree on intersections). Fixing an $i$, the exactness of
\[ 0 \to \phi_{i,*} \O_Z \to \oplus_{j} \phi_{ij,*} \O_Z \to \oplus_{j, k} \phi_{ijk,*} \O_Z \]
follows from the exactness of
\[ 0 \to O_{Z_i} \to \oplus_{j} \phi_{j,*} \O_{Z_i \cap Z_j} \to \oplus_{j, k} \phi_{jk,*} \O_{Z_i \cap Z_j \cap Z_k} \]
and the exactness of the functor $\phi_{i,*}$. Taking the direct sum of these for all $i$, we obtain the exactness  This proves the exactness of the sequence 
\end{comment}

\end{proof}

\begin{comment}
The morphism $\wt{Z} \to Z$ is universally of cohomological descent. The results of \cite{HodgeIII} imply this for abelian sheaves in the classical topology, but this can be used to deduce it for sheaves in the Zariski topology. \todo{I think I am overcomplicating things}
\end{comment}

Rewriting the first few terms as \[ 0 \to p_*(\G_m)/\G_m \to \oplus_{i < j} \phi_{ij,*} \G_m \to \oplus_{i < j < k} \phi_{ijk,*} \G_m \]
and taking global sections gives an exact sequence \[ 0 \to H^0(Z, p_*(\G_m)/\G_m) \to \oplus_{i < j} H^0(Z_i \cap Z_j, \G_m) \to \oplus_{i < j < k} H^0(Z_i \cap Z_j \cap Z_k, \G_m). \] As $Z_i \cap Z_j$ is proper, reduced, and its connected components are geometrically connected, we have $H^0(Z_i \cap Z_j, \G_m) = \oplus_{\pi_0(Z_i \cap Z_j)} L^*$. Comparing to a similar exact sequence for the constant sheaf $\underline{L}^*$ in the classical topology, the inclusion $\underline{L}^* \subset \underline{\C}^* \subset \G_m^{\an}$ induces isomorphisms 
\[ H^0(Z, p_*(\G_m)/\G_m) \isom H^0(Z^{\an}, p_*(\G_m^{\an})/\G_m^{\an}) \isom H^0(Z(\C), p_*(L^*)/L^*), \] and hence, as $H^0(\wt{Z}, \G_m) \isom H^0(\wt{Z}(\C), L^*)$, an isomorphism \begin{align}\label{a3}H^0(Z, p_*(\G_m)/\G_m)/H^0(\wt{Z}, \G_m) = H^0(Z(\C), p_*(L^*)/L^*)/H^0(\wt{Z}(\C), L^*).
\end{align}

Combining (\ref{a1}), (\ref{a2}), (\ref{a3}) gives the result.

%In factTo see that this is identified with $H^1_{sing.}(Z(\C), F^*)$, note that there is no distinction between the Zariski and classical global sections $H^0(\wt{Z}, F^*)$ and $H^0(Z, p_*(F^*)/F^*)$, so the same argumentfor the constant sheaf $F^*$ in the classical topology identifies $\frac{H^0(Z, p_*(F^*)/F^*)}{H^0(\wt{Z}, F^*)} = \ker(H^1(Z(\C), F^*) \to \oplus_i H^1(Z_i(\C), F^*))$.

%A standard argument in coherent cohomology shows that any line-bundle in $Pic^0(\del X)(K)$ is determined by its transition functions along the intersection of any two irreducible components. As these irreducible components
% Why can we describe line-bundles via transitions on the closed subsets? We have a map $p \from \coprod Z_i \to \del X$. We have $H^1(\del X, \G_m) \to H^1(\del X, p_*(\G_m)) 
%TODO: All gluings on intersecting irred. components are functions on a proper variety defined over $K$, hence are constant, in $F^*$. We see that the line-bundle is locally constant, in the topological sense, hence corresponds to an element of $H^1(\del X, F^*)$. Conversely, any locally constant line-bundle (with transition functions in $F^*$) is trivial on each irreducible component, as the irreducible components are simply-connected.
\end{proof}

%\begin{rmk}
%Note that the isomorphism $\Pic^0(Z) \isom \ker(H^1(Z(\C), L^*) \to \oplus_i H^1(Z_i(\C), L^*))$ depends on the embedding $L \subset \C$. 
%\end{rmk}

\begin{comment}
With the same hypotheses as this lemma, we immediately obtain the following:
\begin{cor}\label{simply-conn}
If $\pi_1(Z_i(\C)) = 0$ for all irreducible components $Z_i$, then
\[ \Pic^0(Z) \isom H^1(Z(\C), L^*) \isom \Hom(\pi_1(Z(\C)), L^*), \]
and every line-bundle in $\Pic^0(Z)$ is the scalar extension of a rank 1 $L$-local system on $Z(\C)$, corresponding to a representation $\rho \from \pi_1(Z(\C)) \to \GL_1(L)$.
\end{cor}
\end{comment}

We now prove Proposition \ref{hodge-on-boundary}. 
\begin{proof}[Proof of Prop. \ref{hodge-on-boundary}]
(1) The variety $Z = \del X$ satisfies the hypotheses of Lemma \ref{lem:line-bundle-rational-var}: it satisfies hypotheses (1) and (3) by Proposition \ref{thm:alg-tor-cpt} (6). It satisfies hypothesis (2) by the fact that $\del X \subset X$ is a NCD. As the irreducible components of $\del X$ are simply-connected (by the remark after \ref{lem:canonical-extn}), we obtain
\[ \Pic^0(\del X)(L) \isom H^1(\del X(\C), L^*) = \Hom(T(\Z), L^*), \]
 and that the line-bundle $\bar{\L_i}|_{\del X}$ is a scalar extension of a rank 1 local system on $\del X(\C)$ corresponding to a representation $\rho_{i} \from \pi_1(\del X(\C)) \to \GL_1(L)$.

 (2) It is not hard to see what this representation $\rho_i$ is. Applying Corollary \ref{cor:fund-group-bdry}, we see that $\pi_1(\del X(\C)) \isom T(\Z)$, compatible with the action of $T(\Z) \subset \SL_2(\O)$ on $\H^d$. For $g = \begin{pmatrix}u & 0 \\ 0 & u^{-1}\end{pmatrix} \in T(\Z)$, we have $g^*(\alpha_i) = \sigma_i(u)\alpha_i$. We conclude that the element $\bar{\L_i}|_{\del X} \in \Pic^0(\del X)(L)$ corresponds to the map \[ T(\Z) \subset \O_F^* \nminj{\sigma_i} F^* \subset L^*.\] 

\end{proof}

\subsubsection{Chern classes and volumes}

Consider the differential forms $\omega_i = \frac{dz_i \wedge d\bar{z_i}}{y_i^2}$, $i = 1,\ldots,d$, on $Y(\C)$.

\begin{lem}\label{hodge-bundle}
The line-bundle $\L_{i,\C}$ on $Y_{\C}$ satisfies \[ [\omega_i] = c_1(\L_i) \in H^2(Y(\C), \Q(1)). \]
\end{lem}
\begin{proof}
It is clear from looking at transition functions that $(\L_i^{\tensor 2})^{\an}$ is isomorphic to the line-bundle of holomorphic 1-forms which, when pulled back to $\H^d$, are of the form $f(z)dz_i$. For sections $s_1 = f(z)dz_i, s_2 = g(z)dz_i$, we define $(s_1, s_2) := \frac{s_1 \wedge \bar{s_2}}{\omega_i}$. This defines a Hermitian metric on $(\L_i^{\tensor 2})^{\an}$. Locally, we have the form $dz_i$, with $(dz_i, dz_i) = |y_i|^2$. The curvature 2-form of the Chern connection associated to this Hermitian metric is then $\bar{\del}\del  \log (dz_i, dz_i) = \bar{\del} \del \log |y_i|^2 = 2\frac{dz_i \wedge \bar{dz_i}}{y^2}$. We obtain $c_1(\L_i^{\tensor 2}) = 2[\omega_i]$. \todo{Sign?}
% It seems like a lot of sources say that $\frac{1}{2 \pi i} \del \bar{\del} \log h(s,s)$ is the integral Chern class.
\end{proof}

We also consider the Chern classes of the canonical extensions $\bar{\L_i}$, $c_1(\bar{\L_i}) \in H^2(X(\C), \Q(1))$. As $\bar{\L_i}|_{\del X}$ is the scalar extension of a local system (Proposition \ref{hodge-on-boundary}), and the Chern classes of local systems are well-known to be trivial in singular cohomology with $\Q$-coefficients, we find:
\begin{cor}
\[ c_1(\bar{\L_i}) \in \ker(H^2(X(\C), \Q(1)) \to H^2(\del X(\C), \Q(1))).\]
\end{cor}

%\todo{For line-bundles is this even true integrally? Unsure, as Chern classes not defined by exponential exact sequence for singular varieties, AFAIK.}

% is trivial when restricted to each irreducible component of $\del X(\C)$ (Lemma \ref{lem:canonical-extn}), or because it 
% $H^2(\del X, \Z) \to H^2(\del X, p_*(\Z)) \to H$?

We will also need the following consequence of the Hirzebruch-Mumford Proportionality theorem and Siegel's computation of the volume of $\H^d/\SL_2(\O_F)$:
\begin{thm}[\cite{Mumford}, \cite{Siegel}]\label{thm:mumford-proportionality}
We have
\[ \int_{X(N)} c_1(\bar{\L_1}) \cup \cdots \cup c_1(\bar{\L_d}) = \int_{Y(N)} \omega_1 \wedge \cdots \wedge \omega_d = \begin{cases}
2|\zeta_F(-1)|(2 \pi i)^d & \text{for } N = 1 \\
|\zeta_F(-1)|\cdot|\SL_2(\O/N\O)|(2 \pi i)^d  & \text{for } N \geq 3\\
\end{cases}
  \] \todo{Sign?}
\end{thm}
We denote this quantity by $\vol(Y(N))$.
% For d = 1, we have $\int_{X(1)} dz \wedge d\bar{z}/y^2 = -2i \int dx \wedge dy/y^2 = -2\pi i/6 = 2\zeta(-1)(2 \pi i)$. 

%, compatible with considering the map $Y(N) \to Y(1)$ as a degree $|\PSL_2(\O/N\O)|$ branched cover of $Y(1)$.

% Is it $2^d|\zeta_F(-1)|\cdot|\SL_2(\O/N\O)|$ or $|\zeta_F(-1)|\cdot|\SL_2(\O/N\O)|$?

% Check on modular curve. Degree of the Hodge bundle looks like 1/12. On $Y_0(p)$, it looks like (p+1)/12. The size of $\SL_2(\Z)/\Gamma_0(p)$ is $p+1$. This makes it look like the coarse $Y(1)$ should have volume $2|\zeta(-1)|$. 
% The Chern classes descend to $Y(1)$ even though the Hodge bundles do not.

Let $H^2(X)_{\univ}$ denote the subspace of $H^2(X, \C)$ spanned by the classes $\bar{\omega_i} :=  c_1(\bar{\L_i})$. The cup-products of these Chern classes given an embedding 
\[ \wedge^* H^2(X, \C)_{\univ} \into H^*(X, \C), \]
and we denote this subspace as $H^*(X)_{\univ}$. This defines the \emph{universal cohomology} of $X$. We have already seen that $H^*(X)_{\univ}$ is contained in $\ker(H^*(X, \C) \to H^*(\del X, \C))$. The image of $H^*(X)_{\univ} \into H^*(X,\C) \to H^*(Y,\C)$ defines a subspace $H^*(Y)_{\univ}$. 

\begin{lem}
The map $H^*(X)_{\univ} \to H^*(Y)_{\univ}$ is an isomorphism, except in degree $2d$, where $H^{2d}(X)_{\univ} = \C$, $H^{2d}(Y)_{\univ} = 0$.
\end{lem} 
\begin{proof}
In (\cite{Freitag}, Ch. 3), a similar result is shown for the cup-products of the classes $[\omega_i] \in H^2(Y, \C)$. We then use the identification of $c_1(\bar{\L_i})|_{Y} = c_1(\L_i) = [\omega_i]$.
\end{proof} 

By the $\SL_2(\O)$-invariance of the forms $\omega_i$, the subspace $H^*(Y)_{\univ}$ is actually pulled back from a subspace of $H^*(\H^d/\SL_2(\O))$, which we similarly denote $H^*(\H^d/\SL_2(\O))_{\univ}$.

\begin{comment}
\todo{Where do we use thiss? It needs justification}
\begin{lem}
For $i \neq d$, we have an isomorphism \[ H^i(X)_{\univ} \isom \ker(H^*(X, \C) \to H^*(\del X, \C)) \isom \im(H^*_c(Y,\C) \to H^*(X,\C)). \]
\end{lem}
For $i = d$, we define $H^d(X)_{\cusp} := \frac{\ker(H^*(X, \C) \to H^*(\del X, \C))}{H^d(X)_{\univ}}$. This is Poincare dual to $\im(H^d(X) \to H^d(Y))^{\perp H^d(X)_{\univ}}$, which is in fact the subspace of $H^d(Y)$ generated by cusp-forms (see Section \ref{subsec:rat-eis-cohom}).
\end{comment}

%The above \ref{hodge-on-boundary} and \ref{thm:mumford-proportionality} are the main numerical computations we will need to prove our results. 

\subsection{Eisenstein series in de Rham cohomology}\label{subsec:eis-series}
In this section, let $Y = \H^d/\SL_2(\O)$. By a differential form on $Y$, we will mean an $\SL_2(\O)$-invariant differential form on $\H^d$.

%As discussed in the introduction, there is a ``restriction" map \[ i^*_{\infty} \from H^*(Y) \to H^*(\H^d/P(\Z)). \] 

We first recall the computation of $H^*(\H^d/P(\Z), \C)$, for \[ P(\Z) = \left\lbrace \begin{pmatrix} * & * \\ 0 & * \\ \end{pmatrix} \in \SL_2(\O) \right\rbrace \] the parabolic subgroup corresponding to the cusp $\infty \in \P^1(F)/\SL(\O)$. We have an exact sequence
\[ 0 \to N(\Z) \to P(\Z) \to T(\Z) \to 0, \]
with $N(\Z) = \left\lbrace \begin{pmatrix} 1 & * \\ 0 & 1 \\ \end{pmatrix} \in \SL_2(\O) \right\rbrace$, $T(\Z) = \left\lbrace \begin{pmatrix} * & 0 \\ 0 & * \\ \end{pmatrix} \in \SL_2(\O) \right\rbrace$. %We have an inclusion \[ \{ z \in H^d \mid y_1\cdots y_d = 1 \}/P \subset \H^d/P, \] which is easily seen to be a homotopy equivalence. 
To a subset $J = \{j_1, \ldots, j_k \} \subset [d] := \{1, \ldots, d \}$, we associate two differential forms on $\H^d$: \[ dz_{J} = dz_{j_1} \wedge \cdots \wedge dz_{j_k},\ \frac{dy_J}{y_J} = \frac{dy_{j_1}}{y_{j_1}} \wedge \cdots \wedge \frac{dy_{j_k}}{y_{j_k}}. \]

\begin{prop}[\cite{Freitag}, Ch. 3]\label{boundary-ranks}
The differential forms \[ \frac{dy_J}{y_J}, \ dz_{[d]} \wedge \frac{dy_{J}}{y_{J}} \text{ for } J \subset [d-1] \] are closed and $P(\Z)$-invariant, descending to closed forms on $\H^d/P(\Z)$. Their cohomology classes form a basis of $\oplus_{i = 0}^{2d-1} H^i(\H^d/P(\Z), \C)$. In particular, the basis elements living in degrees $d \leq i \leq {2d-1}$ are $[dz_{[d]} \wedge \frac{dy_J}{y_J}]$, for $J \subset [d-1]$, and \[ \dim H^{i}(\H^d/P(\Z)) = \binom{d-1}{i-d}. \]
\end{prop}

Note that $\H^d/P(\Z)$ deformation-retracts to \[ (\H^d/P(\Z))^{N = 1} := \{ z \in \H^d \mid N(y) = y_1 \cdots y_d = 1 \}/P(\Z), \] a closed, orientable $(2d-1)$-manifold. In particular, this implies that $\frac{d(y_1 \cdots y_d)}{y_1 \cdots y_d}$ is exact. This necessitates the restriction to $J \subset [d-1]$, as opposed to $J \subset [d]$, in \ref{boundary-ranks}. 

The space $E := (\H^d/P(\Z))^{N = 1}$ is a fiber bundle over $B := T(\R)^{N=1}/T(\Z)$ with fiber $F := N(\R)/N(\Z)$. We have a map $p \from E \to B$ given by $z = x + iy \mapsto y$, with a section $s \from B \to E$ given by $y \mapsto iy$.

\begin{lem}\label{lem:homology-from-base}
For $0 \leq i \leq d-1$, the map $s_* \from H_i(B,\Q) \to H_i(E,\Q)$ is an isomorphism. 
\end{lem}
\begin{proof}
Since $p_* \circ s_* = \id$, the map $s_*$ is injective. Use the Poincare duality isomorphism \[ H_i((\H^d/P(\Z))^{N = 1}) \nmisom{\PD} H^{2d-1-i}((\H^d/P(\Z))^{N = 1}) \] and \ref{boundary-ranks} to see that \[ \dim_{\Q} H_i(\R^{d-1}/T(\Z)) = \dim_{\Q} H_i((\H^d/P(\Z))^{N = 1}, \Q). \] 
\end{proof}

We want to define, for $0 \leq k \leq d-2$, sections $\Eis \from H^{d+k}(\H^d/P(\Z), \C) \to H^{d+k}(Y, \C)$. By Theorem \ref{boundary-ranks}, it suffices to find, for all $J \subset [d-1]$, $|J| = k$, a closed, $\Gamma$-invariant differential $(d + k)$-form whose restriction to $\H^d/P(\Z)$ is cohomologous to $dz_{[d]} \wedge \frac{d\bar{z}_J}{y_J}$. If we average the form $dz_{[d]} \wedge \frac{d\bar{z}_J}{y_J}$ over $P(\Z) \backslash \SL_2(\O)$, we obtain an \emph{Eisenstein series}

\begin{align*}
E_J & := \sum_{\gamma \in P(\Z) \backslash \SL_2(\O)} \gamma^*(dz_{[d]} \wedge \frac{d\bar{z}_J}{y_J}) \\ & = \left(\sum_{(c,d) \in (\O)^2/\O^*, (c,d) = 1} \prod_{i \notin J}\frac{1}{(c_i z_i + d_i)^2} \cdot \prod_{i \in J}\frac{1}{|c_i z_i + d_i|^2} \right) dz_{[d]} \wedge \frac{d\bar{z}_J}{y_J}.
\end{align*}

However, this summation does not absolutely converge. Following (\cite{Freitag}, Ch. 3), we actually define these forms via ``Hecke-summation":
\begin{align*} E_J := \lim_{s \to 0^+} & \sum_{(c,d) \in (\O)^2/\O^*, (c,d) = 1} \frac{1}{|N(cz+d)|^s} \\ & \cdot\prod_{i \notin J}\frac{1}{(c_i z_i + d_i)^2} \cdot \prod_{i \in J}\frac{1}{|c_i z_i + d_i|^2} dz_{[d]} \wedge \frac{d\bar{z}_J}{y_J}.
\end{align*}

\begin{thm}[\cite{Freitag}, Ch. 3]\label{thm:eis-series}
\begin{enumerate}
\item For $0 \leq |J| \leq d-1$, the differential form $E_J$ is $\SL_2(\O)$-invariant.
\item For $0 \leq |J| \leq d-1$, the restriction of $E_J$ to $(\H^d/P(\Z))^{N=1}$ is closed, cohomologous to \[ dz_{[d]} \wedge \frac{d\bar{z}_J}{y_J} \in H^{d+|J|}(\H^d/P(\Z), \C). \]
\item For $0 \leq |J| \leq d-2$, the differential form $E_J$ is closed.  
\end{enumerate}
\end{thm}

Fix $0 \leq k \leq d-2$. We define the map 
\begin{align*} 
\Eis \from H^{d+k}(\H^d/P(\Z), \C) \to H^{d+k}(Y, \C)
\end{align*} by \[ [dz_{[d]} \wedge \frac{d\bar{z}_J}{y_J}] \mapsto [E_J] \] for $J \subset [d-1]$, $|J| = k$. Theorem \ref{thm:eis-series} implies that $\Eis$ is a section of $i^*_{\infty} \from H^{d+k}(Y, \C) \to H^{d+k}(\H^d/P(\Z), \C)$.

\begin{rmk}
At the level of forms, this section $\Eis$ is a bit ad-hoc, as we singled out the subset $[d-1] \subset [d]$ in order to choose a basis of $H^{d+k}(\H^d/P(\Z))$. However, in Section \ref{subsec:hodge-theory} we will characterize $\Eis$ by its behavior with respect to Hecke operators and the Hodge filtration. %Any subspace of $V_j := \sum_{J \subset [d], |J| = k} \C \cdot dz_{[d]} \wedge \frac{d\bar{z}_J}{y_J}$ which maps isomorphically to $H^{d+k}(\H^d/P(\Z))$ would have worked.  %In particular, we have a well-defined map $V_j \to H^{d+k}(Y, \C)$. %This suggests that the exact forms in $V_j$ %This implies that any form in $V_j$ which is exact produces a sum of Eisenstein series which is exact. % (presumably the derivative of other Eisenstein series). 
%It should be possible to verify this directly using $(\g, K)$-cohomology.
\end{rmk}

\subsection{Rationality of Eisenstein cohomology}\label{subsec:rat-eis-cohom}
\

In this section, $Y = \H^d/\SL_2(\O_F)$. We define the \emph{Eisenstein cohomology}: \[ H^*(Y)_{\Eis} := H^*(Y)_{\univ} \oplus \im(\Eis) \subset H^*(Y,\C).\] 
\begin{rmk}
This does not agree with Harder's notion of Eisenstein cohomology, as we will see in Section \ref{sec:harder-section}.
\end{rmk} There is also a subspace $H^d(Y)_{\cusp} \subset H^d(Y)$, defined via cusp-forms (\cite{Freitag}, Ch. 3). We let $H^i(Y)_{\cusp} = 0$ for $i \neq d$.

All cohomology is either cuspidal or Eisenstein:

\begin{thm}[\cite{Freitag}, Ch. 3]
\[ H^*(Y) = H^*(Y)_{\Eis} \oplus H^*(Y)_{\cusp}. \]
\end{thm}

These subspaces are distinguished by their eigenvalues under the action of the spherical Hecke algebra $\T$:
\begin{prop}\label{prop:hecke-on-eis}
The Hecke operator $T_{\p}$, for $\p$ a prime of $\O_F$, acts on $H^*(Y)_{\Eis}$ as multiplication by  $1 + N(\p)$. On $H^*(Y)_{\cusp}$, $T_{\p}$ is diagonalizable, with eigenvalues $\alpha$ of absolute value $|\alpha| < 1 + N(\p)$.
\end{prop}
\begin{proof}We only give a sketch, as this is well-known. The proof is essentally the same as for modular curves.

The Hecke operator $T_{\p}$ acts as $1 + N(\p)$ on $H^*(Y)_{\univ}$ using the fact that the forms $\omega_i$ on $\H^d$ are $\SL_2(\R)^d$-invariant.

The Eisenstein series $E_J$ (\ref{subsec:eis-series}) satisfy $T_{\p}E_J = (1 + N(\p))E_J$, as can be checked by direct computation with their description in terms of summation. Another proof goes as follows: $E_J$ is known to be a Hecke eigenvector, with Hecke eigenvalues the same as the holomorphic Eisenstein series in degree $d$. These can be read off from its Fourier expansion, which can be computed via Poisson summation:
\[ -\zeta_F(-1)/2^d + \sum_{\alpha \in \mathfrak{d}^{-1}, \alpha \gg 0} \sigma_1(\alpha \mathfrak{d}) q^{\alpha}, \]
with $\sigma_1(I) = \sum_{J \mid I} N(J)$ the sum over the norms of ideals $J$ dividing $I \subset \O_F$. Note $\sigma_1(\p) = 1 + N(\p)$.

Using the Poincare inner product on cusp forms, for which the operators $T_{\p}$ are self-adjoint, we find that $T_{\p}$ acts diagonalizably on $H^d(Y)_{\cusp}$. The \emph{holomorphic} cusp forms, by the density of $\SL_2(\O_F[1/\p]) \subset \SL_2(F \tensor \R)$ and a maximum principle argument, satisfy the desired eigenvalue bounds. The space $H^d(Y)_{\cusp}$ is spanned by the ``partial complex conjugates" of holomorphic cusp forms (see Section \ref{subsec:partial-conjugation}). These are the automorphic forms which are the orbit of a holomorphic cusp form under the action of $\pi_0(\GL_2(F \tensor \R)) \isom (\Z/2\Z)^d$.  The action of $\pi_0(\GL_2(F \tensor \R))$ arises from considering automorphic forms as functions on $(Z(\A_F) \GL_2(F)) \backslash (\GL_2(\A^f_F) \times \GL_2(F \tensor \R))$. In particular, the action of $\pi_0(\GL_2(F \tensor \R))$ commutes with all Hecke operators at finite places.
\end{proof}

As the Hecke operators act on the rational cohomology $H^*(Y, \Q)$, we obtain:
\begin{cor}
The subspaces $H^d(Y)_{\cusp}$ and $H^d(Y)_{\Eis}$ are defined over $\Q$, giving a decomposition
\[ H^d(Y,\Q) = H^d(Y, \Q)_{\cusp} \oplus H^d(Y, \Q)_{\Eis}. \]
\end{cor}

This leads to an alternate definition of $H^*(Y,\Q)_{\Eis}$:
\begin{cor}
$H^*(Y,\Q)_{\Eis}$ is the summand of $H^*(Y,\Q)$ given by localizing at the ideal $I_{\Eis} \subset \T$ generated by $T_{\p} - (1 + N(\p))$ for all primes $\p \subset \O$.
\end{cor}

\begin{rmk}
When we work with auxiliary level $\Gamma(N)$, we will define $H^*(\H^d/\Gamma(N),\Q)_{\Eis}$ by the localization at the ideal generated by $T_{\p} - (1 + N(\p))$ for all primes $\p$ coprime to $N$. The analogs of the above results hold in this setting.
\end{rmk}

The subspace $H^*(Y)_{\univ}$ descends to a $\Q$-subspace $H^*(Y,\Q)_{\univ} \subset H^*(Y, \Q)_{\Eis}.$ However, the subspace $\im(\Eis) \subset H^*(Y)_{\Eis}$ need not be defined over $\Q$. Define \[ H^*(Y, \Q)_{\del} := H^*(Y,\Q)_{\Eis}/H^*(Y,\Q)_{\univ}, \] so that we have a short exact sequence
\begin{align*}
0 \to H^*(Y, \Q)_{\univ} \to H^*(Y, \Q)_{\Eis} \to H^*(Y, \Q)_{\del} \to 0. 
\end{align*} For $0 \leq k \leq d-1$, t here is an injection $H^{d+k}(Y, \Q)_{\del} \subset H^{d+k}(\H^d/P(\Z), \Q)$, an isomorphism for $0 \leq k \leq d-2$. We may consider the section $\Eis$ as a map $\Eis \from H^{d+k}(Y, \Q)_{\del} \to H^{d+k}(Y, \C)_{\Eis}.$

\begin{comment}
Assuming Theorem \ref{thm:eis-periods-1}, we find:
\begin{cor}
The subspace $\im(s) \subset H^{2d-2}(Y)$ is not defined over $\Q$. For $j = 2, \ldots, d$ the subspace $\im(s) \subset H^{2d-1-j}(Y)$ is defined over $\Q$.
\end{cor}
\end{comment}
 
\subsection{Hodge filtration of Eisenstein series}\label{subsec:hodge-theory}
Let $Y = \H^d/\Gamma(N)$ for the moment.

%Ziegler computed the mixed Hodge numbers of Hilbert modular varieties in \cite{Freitag}, Chapter 3. However, they do not verify the behavior of the Eisenstein sections $s_j$ with respect to the Hodge filtration. We do this now, although it should also follow from the very general results of Harris--Zucker \cite{Harris-Zucker}.

Since $Y$ is a smooth algebraic variety over $\C$ and $X - Y = \del X$ is a simple normal crossing divisor, Deligne's mixed Hodge structure \cite{HodgeII} on $H^*(Y, \Q)$ is defined via the logarithmic holomorphic de Rham complex $\Omega^*_X(\log \del X)$, and the fact that its hypercohomology $\H^i(X, \Omega^*_X(\log \del X))$ computes $H^i(Y, \C)$. %(In this section we will use $\Omega^q_X$ to denote the holomorphic differential forms on $X$.)

\textbf{Rational structure:}
$H^*(Y, \Q) \subset H^*(Y, \C)$
%\[ C^*(Y,\Q) \subset C^*(Y,\Q) \tensor \C \isom \Omega^*_X(\log \del X), \] where this quasi-isomorphism comes from Poincare's theorem.

\textbf{Weight filtration:} We will not recall the definition of the weight filtration. We will only need the fact that there is an increasing filtration $W_{*,\Q}$ on $H^i(Y, \Q)$, such that $W_{j,\Q}H^i(Y,\Q) = H^i(Y,\Q)$ for $j \geq 2d$.

%There is an increasing filtration $W_{*,\C}$ on the complex $\Omega^*_X(\log \del X)$, where $W_{i, \C} \Omega^k_X(\log \del X)$ denotes forms which have ``at most $i$ terms of the form $\frac{dz}{z}$" along $\del X$.

 %, i.e.\ the $\O_X$-module locally generated by $\frac{dq_1}{q_1} \wedge \ldots \wedge \frac{dq_j}{q_j} \wedge dq_{j+1} \wedge \ldots dq_k$, where $q_1, \ldots, q_d$ are local coordinateshe  Count forms $dz/z$ as weight 1. 

%Shift the resulting filtration on cohomology up by cohomological degree: \[ W_{i+j, \C}H^i(Y, \C) := \im(\H^i(W_{j,\C} \Omega^*(\log \del X)) \to \H^i(\Omega^*(\log \del X))). \]

%TODO: Is this the same as the filtration induced by the spectral sequence? Is that the correct one to use?

%This gives a increasing filtration $W_{*,\C}$ on $H^i(Y, \C)$. This filtration is in fact defined over $\Q$: $W_{*,\C} = W_{*, \Q} \tensor \C$ for a filtration $W_{*,\Q}$ on $H^i(Y, \Q)$. It is known that $W_{j, \Q}H^i(Y, \Q)$ for $j < i$, $W_{j,\Q}H^i(Y,\Q) = H^i(Y,\Q)$ for $j > 2d$.

\textbf{Hodge filtration:} We have a filtration on the complex: \[ \Fil^p(\Omega^{p'}_X(\log \del X)) = \begin{cases} 0 & \text{ if } p' < p, \\ \Omega^{p'}_X(\log \del X) & \text{ if } p \geq p. \end{cases}\] This induces a descreasing filtration $\Fil^p$ on $H^i(Y, \C)$, via
\[ \Fil^p H^i(Y,\C) := \im(\H^i(\Fil^p \Omega^{p'}_X(\log \del X)) \to \H^i(\Omega^{p'}_X(\log \del X))). \]

%This agrees with the filtration defined by the spectral sequence for a filtered complex, although this is not obvious.

Following Ziegler (\cite{Freitag} Ch. 3), we resolve $\Omega^*_X(\log \del X)$ (as a sheaf in the classical topology) by the double complex of acyclic sheaves

\[ \Omega^*_X(\log \del X) \tensor_{\O_X} A^{0,*}_X. \]

Here, $A^{0,q}_X$ is the sheaf of $C^{\infty}$ differential forms of Hodge type $(0,q)$. We denote the total complex by \[ A^n_{X}(\log \del X) := \sum_{p+q = n} \Omega^p_X(\log \del X) \tensor_{\O_X} A^{0,q}_X. \]

%TODO: If we want to use $C^{\infty}$, need to use Burgos-Gil?

%ISSUE: Can we take this tensor-product to be over $\O_X$? Then we would have \[ \Omega^p_X(\log \del X) \tensor_{\O_X} A^{0,q}_X \subset A^{p,q}_X. \]

\begin{lem}[\cite{Freitag} Ch. 3, \cite{Navarro}]\ The inclusion $A^*_X(\log \del X) \subset A^*_Y$ is a quasi-isomorphism of complexes of sheaves. The filtration $\Fil^*$ on $A^i_{X}(\log \del X)$ computes the filtration $\Fil^*$ on $H^i(Y,\C)$.
\end{lem}

%\todo{Can we get this from the strictness of the Hodge filtration, the computation of the Hodge filtration of the boundary, and the computation of the image of the Eisenstein series? Sure.}
\begin{prop}\label{eis-section-hodge} Let $0 \leq k \leq d-2$. The image of \[ \Eis \from H^{d+k}(\H^d/\SL_2(\O), \C)_{\del} \to H^{d+k}(\H^d/\SL_2(\O_F), \C)_{\Eis} \] is contained in \[ (\Fil^d \cap W_{2d,\C}) H^{d+k}(\H^d/\SL_2(\O_F), \C) = \Fil^d H^{d+k}(\H^d/\SL_2(\O_F), \C). \] %More generally, for $J \in {-2,0,2}^d$, the differential form $E_J \eta_J \omega_{[d] - \supp(J)}$
%A most-holomorphic Eisenstein series $E$ in degree $d+j$ is contained in $Fil^d = \Omega^d_X(\log \del X) \tensor_{\O_X} A^{0,j}_X$ and $W_{d-j} = W_{d-j}(\Omega^d_X(\log \del X)) \tensor_{\O_X} A^{0,j}_X$.
\end{prop}
\begin{proof}
As the map $H^*(\H^d/\SL_2(\O_F), \Q) \to H^*(\H^d/\Gamma(N), \Q)$ is injective and compatible with mixed Hodge structures, it suffices to check this on the smooth complex variety $Y = \H^d/\Gamma(N)$. We will study the behavior of the Eisenstein series $E_J$ near the cusp $\infty \in \Cusps(\Gamma(N))$, corresponding to the connected component $\del X_{\infty}$ of $\del X = X - Y$. Due to its $\SL_2(\O/N\O)$-invariance, the form $E_J$ will have the same behavior at all cusps.

  As noted above, it is a general fact that $W_{2d}H^i(Y, \Q) = H^i(Y, \Q)$.
Thus we only need to check that the forms $E_J$ defined in Section \ref{subsec:eis-series} are contained in $\Fil^d H^{d+|J|}(Y, \C)$. Write the Eisenstein series as \[ E_J = F_J dz_{[d]} \wedge \frac{d\bar{z}_J}{y_J}, \] for some $J \subset [d]$.

We first check that the form $E_J \in \Omega^d_Y \tensor_{\O_Y} A^{0,|J|}_Y$ is in fact contained in $\Omega^d_X(\log \del X) \tensor_{\O_Y} A^{0,|J|}_Y$. Observe that near $\del X_{\infty}$ (i.e.\ for $N(y) \to \infty$), $E_J$ approaches $dz_{[d]} \wedge \frac{d\bar{z}_J}{y_J}$. The form $dz_{[d]}$, defined in a neighborhood of $\del X_{\infty}$, is contained in $\Omega^d_X(\log \del X)$. One way of seeing this is to note that $\Omega^d_X(\log \del X)$ is the canonical extension of $\Omega^d_Y$, and to use the analytic description of canonical extensions described in Section \ref{subsec:hodge-bundle}. 

If $|J| = 0$, we conclude that $E_J \in \Fil^d = \Omega^d_X(\log \del X)$. However, for $|J| > 0$, the form $dz_{[d]} \wedge \frac{d\bar{z_J}}{y_J}$ is not contained in $\Fil^d = \Omega^d_X(\log \del X) \tensor_{\O_X} A^{0,|J|}_X$. To see this, fix a point $P \in \del X_{\infty}$, and consider a small open $U \ni P$. The form $\frac{d\bar{z_J}}{y_J}$ is a multiple of the form $d\bar{z_J} \in A^{0,|J|}_U$ by a function $\frac{1}{y_J} = \frac{(-2 \pi)^{|J|}}{\log |q^J|}$ on $U - (U \cap \del X)$ which has a continuous extension to $U$. This extension is not a $C^{\infty}$ function!

To remedy this, we must show that the form $E_J$ is \emph{cohomologous} to a form contained in $\Fil^d$. We do this as follows. Using the Dolbeault resolution, we have a class $[E_J] \in H^{|J|}(Y^{\an}, \Omega^d_Y)$, the analytic coherent cohomology. We must show that $[E_J]$ is in the image of \[ H^{|J|}(X^{\an}, \Omega^d_X(\log \del X)) \to H^{|J|}(Y^{\an}, \Omega^d_Y). \] For if $E_J \in \bar{\del}\eta + \Fil^d$ for some $\eta \in A^{d, |J|}_Y$, then it is also true that $E_J \in d\eta + \Fil^d$, simply because $\bar{\del}\eta = d \eta$.

We reduce this to a question in a neighborhood of the boundary as follows. We cover $X$ by open sets $U, V$, where $U = X - \del X$, $V \supset \del X$.
\begin{claim}
To show that $[E_J]$ is in the image of \[ H^{|J|}(X^{\an}, \Omega^d_X(\log \del X)) \to H^{|J|}(Y^{\an}, \Omega^d_Y), \]
it suffices to show that $[E_J|_{U \cap V}] \in H^{|J|}(U \cap V, \Omega^d_{U \cap V})$ is in the image of $H^{|J|}(V, \Omega^d_V(\log \del X))$
\end{claim} 
\begin{proof}[Proof of claim.]
We give a Mayer-Vietoris argument. Let $j \from Y \into X$ be the inclusion. As $U \coprod V \to X$ is a cover of $X$, $H^{*}(Y^{\an}, \Omega^d_Y) = H^{*}(X^{\an}, j_*\Omega^d_Y) $ is the cohomology of (the total complex of) the bi-complex \[ R\Gamma(U, \Omega^d_U) \oplus R\Gamma(V, j_*\Omega^d_{V - \del X}) \to R\Gamma(U \cap V, \Omega^d_{U \cap V}), \] and $H^{*}(X^{\an}, \Omega^d_X(\log \del X))$ is the cohomology of the bi-complex\footnote{We are not working in the derived category - the individual complexes should be computed by functorial $\Gamma$-acyclic resolutions of the sheaves, such as Godemont or Dolbeault resolutions.}
\[ R\Gamma(U, \Omega^d_U) \oplus R\Gamma(V, \Omega^d_{V}(\log \del X)) \to R\Gamma(U \cap V, \Omega^d_{U \cap V}). \]

We have an element $(E_J|_U, E_J|_V, 0)$ in the first bi-complex. By assumption, there exists a class $\alpha \in R\Gamma(V, \Omega^d(\log \del X))$ so that $\alpha|_{U \cap V} - E_J|_{U \cap V} = d\eta.$ We find that, as elements in the first bi-complex, $(E_J|_U, E_J|_V, 0)$ is cohomologous to $(E_J|_U, \alpha, \eta)$. The latter element is in the image of the second bi-complex.
\end{proof}

For $V$ small enough, $\pi_0(V) = \pi_0(\del X)$, and the above claim reduces us to a local question at each cusp. Again, by $\SL_2(\O/N\O)$-invariance of $E_J$, it will suffice to study the cusp $\infty$. 

Let $V_{\infty}$ be the neighborhood $S^{(\epsilon^{-1}, \infty]}$ of $\del X_{\infty}$ defined in Theorem \ref{thm:basic-tor}, for $\epsilon$ sufficiently small. As $[E_J|_{U \cap V_{\infty}}] = [F_J dz_{[d]} \wedge \frac{d\bar{z}_J}{y_J}]$, and $[F_J dz_{[d]}] \in H^0(V_{\infty}, \Omega^d_V(\log \del X))$, it suffices to show that $[\frac{d\bar{z_i}}{y_i}] \in H^1(U \cap V_{\infty}, \O_{U \cap V_{\infty}})$ is in the image of $H^1(V_{\infty}, \O_{V_{\infty}})$. We have a commuting diagram
\[ 
\begin{tikzcd}
\Hom(T(\Z), \C) \arrow[equal]{r} \arrow{d} & H^1(V_{\infty}, \C) \arrow{r} \arrow{d} & H^1(V_{\infty}, \O_{V_{\infty}}) \arrow{d} \\
\Hom(P(\Z), \C) \arrow[equal]{r} & H^1(U \cap V_{\infty}, \C) \arrow{r} & H^1(U \cap V, \O_{U \cap V_{\infty}}), \\
\end{tikzcd}
\]
noting that $\pi_1(U \cap V_{\infty}) = P(\Z)$ and $\pi_1(V_{\infty}) = T(\Z)$ (Corollary \ref{cor:fund-group-bdry}). Thus it will suffice to show that $[\frac{d\bar{z_i}}{y_i}] \in H^1(U \cap V_{\infty}, \O_{U \cap V_{\infty}})$ equals the image of a homomorphism $P(\Z) \to \C$ which factors through $T(\Z)$.

On sufficiently small open subsets of $U \cap V_{\infty}$, we have $\frac{d\bar{z_i}}{y_i} = \bar{\del}(\log y_i)$, for some branch of $\log y_i$. As $\log y_i$ is a global function on the cover $\H^d/N(\Z)$ of $\H^d/P(\Z) \supset U \cap V_{\infty}$, and $\begin{pmatrix} u & n \\ 0 & u^{-1}\end{pmatrix}(\log y_i) - \log y_i = 2\log \sigma_i(u)$, we find that $[\frac{d\bar{z_i}}{y_i}]$ is the image of the homomorphism $\phi_i \from P(\Z) \to \C$ defined by $\phi_i(\begin{pmatrix} u & n \\ 0 & u^{-1}\end{pmatrix}) := 2\log \sigma_i(u).$

\end{proof}

\subsection{Mixed Hodge theory of Eisenstein cohomology}\label{subsec:Eis-MHS}
We now switch back to writing $Y = \H^d/\SL_2(\O_F)$. 

Recall that $\Q(d)$ denotes the mixed Hodge structure whose underlying abelian group is $(2 \pi i)^d \Q$, and is a pure Hodge structure of type $(-d,-d)$. In the following proposition, we define/study the MHS on the sequence
\[ 0 \to H^{d+k}(Y,\Q(d))_{\univ} \to H^{d+k}(Y,\Q(d))_{\Eis} \to H^{d+k}(Y, \Q(d))_{\del} \to 0. \]

\begin{prop}\label{prop:MHS-eis-cohom}
Let $0 \leq k \leq d-2$.
\begin{enumerate}
\item There is a Hecke-equivariant decomposition of MHS
\[ H^{d+k}(Y,\Q(d)) \isom H^{d+k}(Y,\Q(d))_{\Eis} \oplus H^{d+k}(Y,\Q(d))_{\cusp}. \]
\item There is a short exact sequence of MHS \[ 0 \to H^{d+k}(Y,\Q(d))_{\univ} \to H^{d+k}(Y,\Q(d))_{\Eis} \to H^{d+k}(Y, \Q(d))_{\del} \to 0, \] where the MHS on $H^{d+k}(Y, \Q(d))_{\del}$ is defined by the sequence.
\item $\Fil^0(H^{d+k}(Y,\C(d))_{\univ}) = 0$, inducing an isomorphism \[ \Fil^0(H^{d+k}(Y,\C(d))_{\Eis}) \isom \Fil^0(H^{d+k}(Y,\C(d))_{\del}) = H^{d+k}(Y,\C(d))_{\del}. \] The inverse of this map is unique section of $H^{d+k}(Y,\C(d))_{\Eis} \to H^{d+k}(Y, \Q(d))_{\del} $ with image in $\Fil^d H^{d+k}(Y,\C(d))_{\Eis}$, and is compatible with the Hodge filtration. This section equals the section $s$ defined in Section \ref{subsec:eis-series}.
\item The MHS on $H^{d+i}(Y, \Q(d))_{\del}$ is trivial, i.e.\ a direct sum of $\Q(0)$'s.
\end{enumerate}
\end{prop}
\begin{proof}
\

(1) We can replace $Y = \H^d/\SL_2(\O)$ by $Y = \H^d/\Gamma(N)$ in this step, as \[ H^*(\H^d/\SL_2(\O),\Q) = H^*(\H^d/\Gamma(N), \Q)^{\SL_2(\O/N\O)},\] and $\SL_2(\O/N\O)$ acts via automorphisms of mixed Hodge structures.

%\todo{Compatibility of moduli-theoretic Hecke with classical ones. Does the classical Hecke operator on HMV have a lattice description? It sends lattices $\O_K^2 \subset \O_K \tensor \C$ to $\O_K$-stable sublattices $L \subset \O_K^2$ of determinant $\mf{p}$}
The spherical Hecke operator $T_{\p}$ on $H^{d+k}(Y, \Q)$ is an endomorphism of mixed Hodge structures, as it is induced by a finite correspondence $Y \leftarrow Y_0(\p) \to Y$ of complex algebraic varieties. While it should be a general fact that push-forward along finite morphisms is compatible with mixed Hodge structures, we do not know a down-to-earth reference (i.e.\ one avoiding mixed Hodge modules), so we give a proof in our case as follows.

The Hodge and weight filtrations for smooth varieties are computed in terms of the log de Rham complexes on ``good compactifications" (i.e.\ smooth compactifications, with boundary a SNCD). The map $f \from Y_0(\p) \to Y$ extends to a map $f \from X_0(\p) \to X$ for some toroidal compactifications (which may be assumed to be good). The properties of canonical extensions implies that $f^* \Omega^*_{X}(\log \del X) = \Omega^*_{X_0(\p)}(\log \del X_0(\p))$. Therefore the map $f_*\Omega^*_{Y_0(\p)} = f_*f^*\Omega^*_{Y} \nmto{\tr} \Omega^*_{Y}$ restricts to a map $f_*\Omega^*_{X_0(\p)}(\log \del X_0(\p)) = f_*\Omega^*_{X_0(\p)}(\log \del X_0(\p)) \to \Omega^*_{X}(\log \del X)$. %While this should be a special case of the compatibility of the Leray-Serre spectral sequence with mixed Hodge structures, we can more concretely \todo{Reference?}

(2) The map $H^{d+k}(Y,\Q(d))_{\univ} \into H^{d+k}(Y,\Q(d))$ is a morphism of MHS, as $\Q \cdot c_1(\L_i) \subset H^2(Y, \Q(1))$ is a sub-MHS.

(3) $H^{2j}(Y,\C)_{\univ}$ is spanned by forms which are of Hodge type $(j,j)$. In particular, unless $k = d$, $\Fil^d(H^{d+k}(Y,\C)_{\univ}) = 0$, but we have assumed $k \neq d$. Thus there is a unique section of $H^{d+k}(Y,\C)_{\Eis} \oplus H^{d+k}(Y,\C)$ with image contained in $\Fil^d$, and this section is compatible with the Hodge filtration. We proved in Proposition \ref{eis-section-hodge} that $\im(\Eis) \subset \Fil^d H^{d+k}(Y,\C)_{\Eis}$.

(4)  As $s$ is compatible with the Hodge filtration by (3), we find that $\Fil^0_{\C} H^{d+k}(Y, \C(d))_{\del} = H^{d+k}(Y, \C(d))_{\del}$. This implies that $\Fil^0_{\C} \cap W_{0,\Q} = H^{d+k}(Y, \Q(d))_{\del}$.

\end{proof}

By Proposition \ref{prop:MHS-eis-cohom} (3), the map $\Eis$ is the unique splitting of \[ H^{d+k}(Y,\C(d))_{\Eis} \to H^{d+k}(Y, \C(d))_{\del} \] compatible with the Hodge filtration. This implies:
\begin{cor}
Theorem \ref{thm:lower-rationality} is equivalent to the splitting of the extension of MHS
\[ 0 \to H^{d+k}(Y,\Q(d))_{\univ} \to H^{d+k}(Y,\Q(d))_{\Eis} \to H^{d+k}(Y, \Q(d))_{\del} \to 0 \]
for $0 \leq i < d-2$.
\end{cor}

In order to give a similar formulation of Theorem \ref{thm:eis-periods-1}, it is convenient to identify the mixed Hodge structures $H^{2d-2}(Y, \Q(d))_{\univ}$ and $H^{2d-2}(Y, \Q(d))_{\del}$. We have

\begin{align*}
& \oplus_{i=1}^d \Q(1) \cdot \omega_i^* \isom H^{2d-2}(Y,\Q(d))_{\univ}, \\
& \del \from H^{2d-2}(Y, \Q(d))_{\del} \isom \O^* \tensor \Q.
\end{align*}
The first isomorphism is essentially by definition, as the universal cohomology consists of cup-products of Chern classes $\omega_i$, whose Hodge-theoretic properties are well-known. The second map is given by \[ H^{2d-2}(Y, \Q(d)) \nmto{i_{\infty}^*} H^{2d-2}(\H^d/P(\Z), \Q) \nmisom{\PD} H_1(P(\Z), \Q) \nmisom{p_*} \O^* \tensor \Q \] (compare with Lemma \ref{lem:borel-serre-units} below). Actually, we normalize this, defining $\del := \frac{1}{\vol(Y(1))} p_* \circ \PD \circ i_{\infty}^*$.

Via these identifications, we obtain an exact sequence
\[ 0 \to \oplus_{i=1}^d \Q(1) \to H^{2d-2}(Y,\Q(d))_{\Eis} \to \O^* \tensor \Q \to 0. \]
Using the isomorphism $\Ext^1_{\MHS}(\Q(0), \Q(1)) \isom \C^* \tensor \Q$ (see \ref{extn-classes}), we obtain a map $m \from \O^* \tensor \Q \to \oplus_{i=1}^d \C^* \tensor \Q.$

\begin{prop}\label{prop:period-to-hodge}
Theorem \ref{thm:eis-periods-1} is equivalent to the equality \[ m = -(\sigma_1, \ldots,\sigma_d) \in \Hom(\O^* \tensor \Q, \oplus_i \C^* \tensor \Q). \]
\end{prop}

Once we prove Theorems \ref{thm:eis-periods-1}, \ref{thm:lower-rationality}, we will obtain:
\begin{cor}\label{cor:rational-sections}
\
\begin{enumerate}
\item Any rational section $H^{2d-2}(Y, \Q)_{\del} \to H^{2d-2}(Y, \Q)_{\Eis}$ is not compatible with the Hodge filtration.
\item There exists a unique section $H^{2d-1-j}(Y, \Q)_{\del} \to H^{2d-1-j}(Y, \Q)_{\Eis}$ which is compatible with the Hodge filtration, for $1 < j \leq d-1$.
\end{enumerate}
\end{cor}

\subsection{Constant terms and toroidal residues}\label{subsec:const-residue}
In this section, we fix $N \geq 3$, and consider the space $Y = \H^d/\Gamma(N)$. We will define a residue map
\[ \Res \from H^{2d-2}(Y, \Q(d)) \to H_1(\del X, \Q(0)), \]
and compare it to the map
\begin{align*}
 H^{2d-2}(Y, \Q(d)) &\to \oplus_{\Cusps(\Gamma(N))} H^{2d-2}((\H^d/P(\Z))^{N=1}, \Q(d)) \\
 & \nmto{\frac{1}{(2 \pi i)^d} \PD} \oplus_{\Cusps(\Gamma(N))} H_1(\R^{d-1}/P(\Z)) \nmto{p_*} \oplus_{\Cusps(\Gamma(N))} T(\Z) \tensor \Q. 
 \end{align*}

It is easy to see (using Lemma \ref{lem:homology-from-base}) that this induces an isomorphism
\begin{lem}\label{lem:borel-serre-units}
$H^{2d-2}(Y, \Q(d))_{\del} \isom \oplus_{\Cusps(\Gamma(N))} H_1(\H^d/P(\Z), \Q(d)) \isom \oplus_{\Cusps(\Gamma(N))} T(\Z) \tensor \Q$.
\end{lem}

 We will then prove
\begin{prop}\label{prop:extn-compat}
\
\begin{enumerate}
\item The map
\[ \oplus_{\Cusps(\Gamma(N))} H_1(\R^{d-1}/T(\Z), \Q) \isom H^{2d-2}(Y, \Q(d))_{\del} \nmto{\Eis} H^{2d-2}(Y, \C(d))_{\Eis} \nmto{\Res} H_1(\del X, \C) \]
induces an isomorphism $\oplus_{\Cusps(\Gamma(N))} H_1(\R^{d-1}/T(\Z), \Q) \isom H_1(\del X, \Q(0))$, compatible with the identifications of each with $\oplus_{\Cusps(\Gamma(N))} T(\Z) \tensor \Q$ (\ref{lem:borel-serre-units}, \ref{prop:cw-str-boundary}).
\item 
The isomorphism of (1) identifies the following extensions of MHS:
\[ \begin{tikzcd}[column sep = small]
0 \arrow{r} & H^{2d-2}(Y, \Q(d))_{\univ} \arrow{r}\arrow[equal]{d} & H^{2d-2}(Y, \Q(d))_{\Eis} \arrow{r}\arrow[equal]{d} & \oplus_{\Cusps(\Gamma(N))} H_1(\R^{d-1}/T(\Z), \Q) \arrow{r}\arrow{d}{\isom} & 0 \\
0 \arrow{r} & H^{2d-2}(Y, \Q(d))_{\univ} \arrow{r} & H^{2d-2}(Y, \Q(d))_{\Eis} \arrow{r}{\Res} & H_1(\del X, \Q(0)) \arrow{r} & 0 \\
\end{tikzcd} \]
\end{enumerate}
\end{prop}

In particular, to prove Theorem \ref{thm:eis-periods-1} via Proposition \ref{prop:period-to-hodge}, we may use the latter extension, which is better suited to algebraic methods. It is not be difficult to see that there is some isomorphism making (2) hold. However, to prove Theorem \ref{thm:eis-periods-1} exactly, as opposed to up to a non-zero rational multiple, we need to identify this isomorphism with that of (1).

In fact, we will prove a more general compatibility between the restriction of forms to the Borel-Serre compactification and the residues of forms along the toroidal compactification. We believe that this compatibility could be deduced from work of Harris--Zucker \cite{Harris-Zucker}.

%We will use the latter to make sense of the Hodge theory/Galois theory of the former. 

For $0 \leq i \leq d-1$, we will define two maps
\begin{align*}
\int_{N(\R)/N(\Z)} & \from H^{2d-1-i}(Y, \Q(d)) \to H_i(\R^{d-1}/T(\Z), \Q(d)) \\
\Res_{\infty} & \from H^{2d-1-i}(Y, \Q(d)) \to H_i(\del X_{\infty}, \Q(0)). 
\end{align*}
The former is defined topologically/group-cohomologically, related to the inclusion $P(\Z) \subset \Gamma(N)$ and the quotient $P(\Z) \surj T(\Z)$, and as such its relation with the Hodge/Galois theory of $Y$ is unclear. The latter is defined using the algebraic toroidal compactification $X$, and as such can be seen to be a map of Hodge structures, or, after tensoring with $\Q_l$ and choosing an embedding $\bar{\Q(\mu_N)} \subset \C$, of $G_{\Q(\mu_N)}$-modules.

We will relate these two maps as follows. We have a map $s \from \R^{d-1}/T(\Z) \to \del X_{\infty}$ (Section \ref{sec:boundary-top}), and the induced map $s_* \from H_i(\R^{d-1}/T(\Z), \Q) \to H_i(\del X_{\infty}, \Q)$ is injective.

\begin{prop}\label{prop:residue-compat}
For $0 \leq i \leq d-1$, there is an equality
\[ \frac{1}{(2 \pi i)^d} s_* \circ \int_{N(\R)/N(\Z)} = \Res_{\infty} \]
of maps $H^{2d-1-i}(Y, \Q(d)) \to H_i(\del X, \Q(0))$.  
\end{prop}
%Note that this implies, but is stronger than, the equality $\frac{1}{(2 \pi i)^d}\int_{N(\R)/N(\Z)} = p_* \circ \Res$.
Proposition \ref{prop:extn-compat} follows immediately from this.

\subsubsection{The map $\int_{N(\R)/N(\Z)}$}
We have a ``restriction" map \[ i^*_{\infty} \from H^{2d-1-i}(Y, \Q(d)) \to H^{2d-1-i}(\H^d/P(\Z), \Q(d)). \] The manifold $\H^d/P(\Z)$ deformation retracts to the closed oriented $(2d-1)$-manifold $(\H^d/P(\Z))^{N=1}$. Using the Poincare duality isomorphism \[ \PD_{2d-1} \from H^{2d-1-i}((\H^d/P(\Z))^{N=1}) \isom H_i((\H^d/P(\Z))^{N=1}) \] (the subscript indicating the dimension of the manifold), we obtain
\[ \PD_{2d-1} \circ i^*_{\infty} \from H^{2d-1-i}(Y, \Q) \to H_i(\H^d/P(\Z), \Q). \]
%Really, this map is ambiguous by a sign, as we did not specify an orientation on $(\H^d/P(\Z))^{N=1}$. %We give it whichever orientation makes Proposition \ref{prop:residue-compat} hold. \todo{Which orientation is that?}

For $0 \leq i \leq d-1$, recall that $s_* \from H_i(\R^{d-1}/T(\Z)) \to H_i(\H^d/P(\Z))$ is an isomorphism (Lemma \ref{lem:homology-from-base}). Composing with the inverse of this map, we obtain
\[ (s_*)^{-1} \circ \PD_{2d-1} \circ i^*_{\infty} \from H^{2d-1-i}(Y, \Q) \to H_i(\R^{d-1}/T(\Z),\Q). \]
We then define \begin{align*}
\int_{N(\R)/N(\Z)} := (s_*)^{-1} \circ \PD_{2d-1} \circ i^*_{\infty}. 
\end{align*}

\begin{rmk}
If we further compose with the Poincare duality \[ H_i(\R^{d-1}/T(\Z)) \isom H^{d-1-i}(\R^{d-1}/T(\Z)),\] we obtain a map \[ H^{2d-1-i}(Y, \Q) \to H^{d-1-i}(\R^{d-1}/T(\Z), \Q) \] which should essentially be the ``constant term" map in the theory of automorphic forms. This would really be the more appropriate map to denote $\int_{N(\R)/N(\Z)}$. 
\end{rmk}

\subsubsection{The map $\Res_{\infty}$}\label{subsubsec:residue}
We define the total residue map
\[ \Res \from H^{2d-1-i}(Y, \Q(d)) \to H_i(\del X, \Q(0)) \] to be the composition of the Poincare duality \[ 
\frac{1}{(2 \pi i)^d} \PD_{2d} \from H^{2d-1-i}(Y, \Q(d)) \isom H_{i+1}((X, \del X), \Q(0)) \] and the boundary map \[ \del \from H_{i+i}((X, \del X), \Q(0)) \to H_i(\del X, \Q(0)). \] Equivalently, $\Res$ is dual to the map $H^i_c(\del X, \Q(0)) \nmto{\delta} H^{i+1}_c(Y, \Q(0))$ occuring in the LES of compactly supported cohomology. The maps $\frac{1}{(2 \pi i)^d} \PD_{2d}$, $\delta$, and the duality between cohomology and compactly supported cohomology are known to be compatible with mixed Hodge structures (see \cite{Fujiki}).

We compose $\Res$ with the map $H_i(\del X, \Q(0)) = \oplus_{x \in \Cusps(\Gamma(N))} H_i(\del X_x, \Q(0)) \to H_i(\del X_{\infty}, \Q(0))$, giving a map \[ \Res_{\infty} \from H^{2d-1-i}(Y, \Q(d)) \to H_i(\del X_{\infty}, \Q(0)). \]

We describe this map in a slightly different way.
For an open set $U \supset \del X_{\infty}$, there is an isomorphism $H^{2d-i}(U,U-\del X_{\infty}, \Q(d)) \isom H_i(\del X_{\infty}, \Q(0))$ (using Poincare duality on $U$, $U-\del X_{\infty}$). We obtain maps $\Res'_{\infty} \from H^{2d-i-1}(Y,\Q(d)) \to H^{2d-i-1}(U, \Q(d)) \to H_i(\del X_{\infty}, \Q(0))$.

\begin{lem}\label{lem:excision}
$\Res'_{\infty}$ is independent of the open set $U$, and $\Res_{\infty} = \Res'_{\infty}$.
\end{lem}
\begin{proof} 
This is a standard argument using excision, which we omit.
\end{proof}

\subsubsection{The proof of Proposition \ref{prop:residue-compat}} 
We use the topological description of the toroidal compactification given in Section \ref{sec:boundary-top}. In particular, we have spaces $S_{\R} \subset S$ and a map $N \from S \to (0, \infty]$, so that $S^{(1,\infty)} \isom \H^d/P(\Z)^{(1,\infty)}$, $S^{\infty} \isom \del X(\C)$, and $S^{(1,\infty]}$ is an open neighborhood of $\del X_{\infty}(\C)$.

 We define $V := S^{[1,\infty]} = N^{-1}([1,\infty]), W := S_{\R}^{[1,\infty]} = S_{\R} \cap N^{-1}([1,\infty])$. We have $\del V := V^{1} \coprod V^{\infty}, \del W := W^{1} \coprod W^{\infty}$. Note that the space $\del V$ is not actually the topological boundary of the space $V$ - it is a topological boundary on one side, and a compactification on the other.

The maps $s$, $p$ are compatible with these boundaries:
\begin{align*}
((\del W)^1 \subset W \supset (\del W)^{\infty}) & \nmto{s} ((\del V)^1 \subset V \supset (\del V)^{\infty}), \\
((\del V)^1 \subset V \supset (\del V)^{\infty}) & \nmto{p} ((\del W)^1 \subset W \supset (\del W)^{\infty}).
\end{align*}

We have a map $\Res_{\infty} \from H^{2d-1-i}(V, \Q(d)) \to H_i(\del V^{\infty}, \Q(0))$ defined in \ref{subsubsec:residue}. Under the Poincare duality \[ 
\frac{1}{(2 \pi i)^d}\PD_{2d} \from H^{2d-1-i}(V - \del V, \Q(d)) \isom H_{i+1}((V, \del V), \Q(0)), \] the map $\Res_{\infty}$ corresponds to the boundary map \[ H_{i+1}(V, \del V^1 \coprod \del V^{\infty}) \nmto{\del_{\infty}} H_i(\del V^{\infty}). \]

By Lemma \ref{lem:excision}, we find that $\Res_{\infty} \circ i^*_{\infty} = \Res_{\infty}$.

% $\del_{\infty} \circ PD_{2d} \circ i^*_{\infty} = \Res_{\infty} \circ i^*_{\infty} = \Res ?= s_* \circ \int_{N(\R)/N(\Z)} = s_* \circ (s_*)^{-1} \circ PD \circ i^*_{\infty}%

Via the identification of $(\H^d/P(\Z))^{N=1}$ with $\del V^1$, we are reduced to showing a compatibility between two Poincare dualities:
\[ \del_{\infty} \circ \PD_{2d} =^? s_* \circ (s_*)^{-1} \circ \PD_{2d-1}, \]
for
\begin{align*}
\PD_{2d} \from & H^{2d-1-i}(V - \del V, \Q) \isom H_{i+1}((V, \del V), \Q), \\ \PD_{2d-1} \from & H^{2d-1-i}(\del V^1, \Q) \isom H_{i}(\del V^1, \Q).
\end{align*}
Note the abuse of notation: $s_* \circ (s_*)^{-1}$ not the identity, but a map
\[ H_i(\del V^1) \nmto{s_*^{-1}} H_i(\R^{d-1}/T(\Z)) \nmto{s_*} H_i(\del V^{\infty}). \]

% using the boundary map $H_{i+i}((U_{\epsilon}, \del X), \Q(0)) \to H_i(\del X, \Q(0))$, for $\del X \subset U_{\epsilon} \supset \H^d/P(\Z)$ as in Section \ref{subsec:boundary-top}.

\begin{comment}
\begin{lem}
\
\begin{enumerate}
 \item We have an isomorphism of relative homology of $(V, \del V)$ with the Borel-Moore homology of $V - \del V$ (locally finite singular chains): \[ H_{i+1}(V, \del V) \isom H_{i+1}^{\BM}(V). \]
 \item The homeomorphism $V - \del V \isom (1, \infty) \times (\H^d/P(\Z))^{N=1}$ induces
 \[ H_{i+1}^{\BM}(V) \isom H_i((\H^d/P(\Z))^{N=1}). \]
\end{enumerate}
\end{lem}
\begin{proof}
1) The relation between relative homology and Borel-Moore homology follows from the fact that $(V, \del V)$ is a ``good pair". See Appendix \ref{appendix:topology}. 2) follows immediately.
\end{proof}

\begin{cor}
The map $\del_{1} \from H_{i+1}(V, \del V) \to H_i(\H^d/P(\Z))$ is an isomorphism.
\end{cor}
\end{comment}

\begin{lem}\label{lem:borel-moore}
\
\begin{enumerate}
\item The map $\del_{1} \from H_{i+1}(V, \del V) \to H_i(\del V^1)$ is an isomorphism.
\item $\del_1 \circ \PD_{2d} = \PD_{2d-1}$
\end{enumerate}
\end{lem}
\begin{proof}
To prove these, we use Borel-Moore homology, i.e.\ the homology of the complex of locally finite singular chains. 

1) There is an isomorphism $H_{i+1}(V, \del V) \isom H_{i+1}^{\BM}(V - \del V)$. The homeomorphism $V - \del V \isom (1, \infty) \times \del V^1$ induces
 \[ H_{i+1}^{\BM}(V - \del V) \isom H_i(\del V^1). \]
 
2) For the moment, for an oriented manifold $M^m$ we consider Poincare duality as an isomorphism of chain complexes $\PD_M \from H^*(M) \isom H^{\BM}_{m-*}(M)$. 
For products of oriented manifolds $M^m, N^n$, it is known that $\PD_M \tensor \PD_N = \PD_{M \times N}$, via the Kunneth theorem on cohomology and the exterior product of simplices on Borel-Moore homology. More precisely, the map $\PD_{M \times N}$ factors as \begin{align*} 
H^*(M \times N) \isom H^*(M) \tensor H^*(N) \nmto{\PD_M \tensor \PD_N}& H^{\BM}_{m-*}(M) \tensor H^{\BM}_{n-*}(N) \\ \nmto{\boxtimes}& H^{\BM}_{(m+n)-*}(M \times N).
\end{align*} 

In our situation, we use the homeomorphism $V-\del V \isom (1,\infty) \times \del V^1$ to obtain \[ \del_1 \circ \PD_{V-\del V} = (\del_1 \circ \PD_{(1,\infty)}) \tensor \PD_{\del V^1} = \PD_{(\H^d/P(\Z))^{N=1}}, \] since $H^0((1,\infty)) \nmto{\PD_{(1,\infty)}} H^{\BM}_1((1,\infty)) \nmto{\del_1} \Q$ is an isomorphism. 
\end{proof}

This reduces us to showing that $\del_{\infty} \circ \del_1^{-1} \from H_i((\H^d/P(\Z))^{N=1}) \to H_i(\del X)$ is equal to $s_* \circ p_*.$ In other words, we need to show:

\begin{comment}
We have maps \begin{align}
s_* \from & H_i(\R^{d-1}/T(\Z)) \to H_i(\H^d/P(\Z)) \\ 
s_* \from & H_i(\R^{d-1}/T(\Z)) \to H_i(\del V^{\infty}).
\end{align} To prove Proposition \ref{prop:residue-compat}, it remains to show:
\end{comment}

\begin{lem}
The following diagram commutes:
\[ \begin{tikzcd}
 & \arrow{dl}{s_*} H_i(\R^{d-1}/T(\Z)) \arrow{dr}{s_*} \\
H_i(\del V^1) & \arrow{l}{\del_1} H_{i+1}(V, \del V) \arrow{r}{\del_{\infty}} & H_i(\del V^{\infty}). \\
\end{tikzcd} \]
\end{lem}
\begin{proof}
Consider the composition 
\[ H_i(\del V^1) \nmfrom{\del_1} H_{i+1}(V, \del V) \nmto{\del_{\infty}} H_i(\del V^{\infty}), \]
where the first map is an isomorphism by Lemma \ref{lem:borel-moore}. 

\begin{claim}$\del_{\infty} \circ \del_1^{-1}$ is equal to
\[ H_i(\del V^1) \to H_i(V) \isom H_i(\del V^{\infty}), \]
with maps induced by the inclusions $\del V^1, \del V^{\infty} \subset V$.
\end{claim}
\begin{proof}[Proof of claim.]
Using the deformation retract of $\del V^{\infty} \into V$ (see Theorem \ref{thm:basic-tor}), we have a map $h \from [1, \infty] \times \del V^1 \to V$, such that $h_1(\del V^1) \subset \del V^1, h_{\infty}(\del V^1) \subset \del V^{\infty}$. We obtain a map of pairs $h \from ([1, \infty] \times \del V^1 , \{1,\infty\} \times \del V^1) \to (V, \del V).$ The map $h_* \from H_*([1, \infty] \times \del V^1, \{1, \infty\} \times \del V^1) \to H_*(V, \del V)$ is compatible with the boundary maps $\del_1$, $\del_{\infty}$. We have the following diagram:
\[ \begin{tikzcd}
H_i(\del V^1) & \arrow{l}{\del_1} H_{i+1}(V, \del V) \arrow{r}{\del_{\infty}} &  H_i(\del V^{\infty}) \\
H_i(\del V^1) \arrow[equal]{u}{h_*} \arrow[equal]{d} & \arrow{l}{\del_1} H_{i+1}([1,\infty] \times \del V^1, \{1,\infty\} \times \del V^1) \arrow{d}{\isom} \arrow{r}{\del_{\infty}} \arrow{u}{h_*} & H_i(\del V^1) \arrow{u}{h_*} \arrow[equal]{d} \\
H_i(\del V^1) \arrow{r}{\isom} \arrow[equal]{d}{h_*} & H_{i}([1,\infty] \times \del V^1) \arrow{d}{h_*} & \arrow{l}{\isom} H_i(\del V^1) \arrow{d}{h_*} \\
H_i(\del V^1) \arrow{r} &  H_{i}(V) & \arrow{l}{\isom} H_i(\del V^{\infty}) \\
\end{tikzcd} \]

It is easy to check that the middle two rows are isomorphic in this way, similar to Lemma \ref{lem:borel-moore}. This diagram then proves the claim.

\end{proof}

\begin{comment}
We have a commuting diagram
\[ \begin{tikzcd}
H_i(\del W^1) \arrow{d}{s_*} & \arrow{l}{\del_1} H_{i+1}(W, \del W) \arrow{d}{s_*} \arrow{r}{\del_{\infty}} & H_i(\del W^{\infty}) \arrow{d}{s_*} \\
H_i(\del V^1) & \arrow{l}{\del_1} H_{i+1}(V, \del V) \arrow{r}{\del_{\infty}} & H_i(\del V^{\infty}) \\
\end{tikzcd}.
\]
\end{comment}

We have a commuting diagram
\[ \begin{tikzcd}
H_i(\del W^1) \arrow{r} \arrow{d}{s_*} & H_i(W) \arrow{d}{s_*} & \arrow{l} H_i(\del W^{\infty}) \arrow{d}{s_*} \\
H_i(\del V^1) \arrow{r} & H_i(V) & \arrow{l} H_i(\del V^{\infty}) \\
\end{tikzcd},
\]
where the horizontal maps are induced by the obvious inclusions of spaces. By Corollary \ref{cor:real-top} that the inclusions $\del W^{1}, \del W^{\infty} \subset W$ are homotopy equivalences. Collapsing the top row via this identification, we obtain \[ \begin{tikzcd} & \arrow{dl}{s_*} H_i(\R^{d-1}/T(\Z)) \arrow{dr}{s_*} & \\
H_i((\H^d/P(\Z))^{N=1}) \arrow{rr}{\del_{\infty}\del_1^{-1}} & & H_i(\del V^{\infty}). \\
\end{tikzcd} \]

%The fundamental groups of $\del W^1 = (\R^{d}/T(\Z))^{N=1}$ and $\del W^{\infty} = \del X_{\infty}(\R_{\geq 0})$. 
 %In Section \ref{subsec:boundary-top}, we described a contractible space $W_{\Sigma}([0,1])$ with a map $N \from W_{\Sigma}([0,1]) \to \R_{\geq 0} \cup \infty$. The group $T(\Z)$ acts freely on $W_{\Sigma}([0,1])$, giving a map $N \from W_{\Sigma}([0,1])/T(\Z) \to \R_{\geq 0} \cup \infty$. It is easy to see that 
%\[ W = N^{-1}([0,\infty]),\ \del W^1 = N^{-1}(1),\ \del W^{\infty} = N^{-1}(\infty). \]

% This is not difficult using the description of the toric compactification of Section \ref{subsec:boundary-top}, but even easier, and sufficient for our purposes, is to 
\end{proof} 
 
\section{Extensions associated to line-bundles}\label{sec:line-bundle-extn}
Consider a smooth projective variety $X$ of dimension $d$ over a field $L \subset \C$. Let \[ Z := (\coprod_{i \in \Z/n\Z} \P^1)/(\infty_i = 0_{i+1})\]be a circle of $\P^1$'s, equipped with an ``orientation" of this circle i.e.\ a cyclic ordering of the irreducible components $Z_i$. This orientation corresponds to a generator $\gamma \in H_1(Z, \Z)$ (the singular homology of the complex points), by taking a loop which goes in the same cyclic direction.  We choose a coordinate $z_i$ on $Z_i$ so that $\div(z_i) = [0_i] - [\infty_i]$.

Suppose that we are given:
\begin{enumerate}
 \item a closed immersion $f \from Z \to X$ such that $f_*(\gamma) \in H_1(f(Z), \Q)$ is non-zero,
 \item a line-bundle $\L$ on $X$ with non-trivial Chern class $c_1(\L) \in H^2(X, \Q(1))$,
 \item that the restriction of $\L$ to each irreducible component $Z_i$ of $Z$ is trivial. 
\end{enumerate} 
 
This implies that $c_1(\L)|_Z = c_1(\L|_Z) = 0 \in H^2(Z, \Z(1))$. There is an isomorphism
\[ H^1(Z, L^*) \isom \Pic^0(Z), \]
coming from the inclusion of $L^* \subset \G_m$ (see Lemma \ref{lem:line-bundle-rational-var}). Via the orientation $\gamma$, we obtain an isomorphism \[ \varphi \from \Pic^0(Z) \isom L^*. \]

We have an exact sequence of singular cohomology groups (of the complex points, via $L \subset \C$)
\[ \ldots \to H^1(f(Z), \Q(1)) \to H^2_c(X-f(Z), \Q(1)) \to H^2(X, \Q(1)) \to \ldots. \]
The map $H^1(X, \Q(1)) \to H^1(f(Z), \Q(1))$ is zero for weight reasons (see \ref{lem:weight-of-circle}). Considering the Chern class $c_1(\L)$ as a map $\Q \cdot c_1(\L) \to H^2(X, \Q(1))$, we pull-back to obtain an extension
\begin{align}\label{extn-in-geometry}
0 \to H^1(f(Z), \Q(1)) \to E \to \Q \cdot c_1(\L) \to 0. 
\end{align}

This extension is endowed with a canonical mixed Hodge structure. As $X$ and $\L$ are defined over $L$,  this sequence, after tensoring with $\Q_l$, is an extension of $G_L$-modules (using our chosen $L \into \C$). 

Consider $0 \neq f_*(\gamma) \in H_1(f(Z), \Q(1))$ as a map $f_*(\gamma) \from H^1(f(Z), \Q(1)) \surj \Q(1)$. This surjection is compatible with mixed Hodge structures/the action of $G_L$ (after tensoring with $\Q_l$) by Lemma \ref{lem:cohom-of-circle} below. We obtain an extension of the form $0 \to \Q(1) \to E \to \Q \to 0$, giving elements $[E] \in \Ext^1_{\MHS}(\Q, \Q(1))$, $[E_{\Q_l}] \in \Ext^1_{G_L}(\Q_l, \Q_l(1))$. Such extensions are well-understood (see Lemma \ref{extn-classes}), and we have the following maps/isomorphisms:
\begin{align*}
\kappa_{\MHS} \from & L^* \to \C^* \tensor \Q \isom \Ext^1_{\MHS}(\Q, \Q(1)) \\
\kappa_{G_L} \from & L^* \to L^* \tensor \Q_l \isom \Ext^1_{G_L}(\Q_l, \Q_l(1)). 
\end{align*}
For example, $\kappa_{G_L}$ is simply the Kummer map.

We will prove:
\begin{thm}\label{thm:line-bundle-extn}
\
\begin{enumerate}
\item $[E_{\Q}] = \kappa_{\MHS}(\varphi(\L|_{Z}))$. 
\item $[E_{\Q_l}] = \kappa_{G_L}(\varphi(\L|_{Z}))$.
\end{enumerate}
\end{thm}

\begin{rmk}
This result easily generalizes to subvarieties $W \subset X$ such that $H_1(W, \Q)$ is generated by the images of maps from circles of $\P^1$'s $f \from Z \to W$. We essentially check this as part of Theorem \ref{thm:classify-eis-extns} below.
\end{rmk}

For simplicity of notation, we will assume below that the map $f \from Z \to X$ is an closed embedding $Z \subset X$, but the proofs do not require this.

We introduce some notation. Let $\cC$ denote $G_L$-mod $:= \mathrm{Rep}^{\mathrm{cts}}_{G_L}(\Q_l)$ or $\MHS_w$ (weak mixed Hodge structures, as discussed in Appendix \ref{appendix:weak-MHS}). As a tensor category, $\mathcal{C}$ has unit object denoted $1_{\mathcal{C}}$. We denote the Tate objects in these categories ($\Q_l(j)$ or $\Q(j)$) by $1_{\mathcal{C}}(j)$. We use $H^i(-, j)$ to denote either $H^i_{\Et}(-_{\bar{L}}, \Q_l(j))$ considered as a Galois module or $H^i(-(\C), \Q(j))$ as a weak mixed Hodge structure. We let $H^i_{\cC}(X, j)$ to denote absolute \'etale cohomology $H^i_{\Et}(-_{L}, \Q_l(j))$ or Deligne-Beilinson cohomology $H^i_{\DB}(X_{\C}, \Q(j))$. See Appendix \ref{appendix:weak-MHS} for a definition, following \cite{Beilinson-AH}, of Deligne-Beilinson cohomology. Note that these cohomology theories have maps to singular cohomology with $\Q_l$ or $\Q$ coefficients, respectively.

As $X$ is assumed projective, we may write $\L$ as a difference of very ample line-bundles. This allows us to choose a meromorphic section $s$ of $\L$ such that the divisor $D = \div(s)$ is disjoint from the singular locus of $Z$: $D \cap Z^{\mathrm{sing}} = \emptyset$. We consider its cycle class $[D] \in H^2_{\cC}(X, 1)$, whose image in singular cohomology is the Chern class $c_1(\L)$. Cycle classes are defined by their compatibility with fundamental classes of subvarieties, in this case by the isomorphism
\[ H^2_{\cC}((X, X-|D|), 1) \isom H^2((X, X-|D|), 1_{\cC}) = 1_{\cC}^{\mathrm{irred}(|D|)}, \]
where $|D|$ is the reduced subvariety on which $D$ is supported, and $\mathrm{irred}(|D|)$ denotes the set of irreducible components of $|D|$. 
 We may restrict this cycle class to $Z$, giving a class $[D]|_{Z} \in H^2_{\cC}(Z, 1)$. 

Scholl \cite{Scholl} explains how cycle classes in the absolute cohomology of simplicial smooth varieties may be defined, if one knows that the cycles satisfy a transversality property with respect to the face maps. It is not hard to see that the same method defines cycle classes for singular varieties, as long as the cycle meets the singular locus in a controlled manner. By the assumption that $D \cap Z^{\mathrm{sing}} = \emptyset$, the cycle class $[D \cap Z] \in H^2_{\cC}(Z, 1)$ is defined. 

\begin{lem}
The classes $[D]|_{Z}$ and $[D \cap Z]$ are equal.
\end{lem}
\begin{proof}
This can be checked in singular cohomology with supports.
%We have $[D] \in H^2_{|D|}(X, \Q(1))$, and a restriction map $H^2_{|D|}(X, \Q(1)) \to H^2_{|D| \cap Z}(Z, \Q(1))$. 
If the support $|D|$ of the divisor $D$ were a sub-manifold of $X-Z^{\mathrm{sing}}$, and the intersection $D \cap (Z-Z^{\mathrm{sing}})$ were transverse, this would follow from standard facts about the relation between fundamental classes and the intersection product for transversely intersecting submanifolds. More generally, see (\cite{Fulton}, Cor. 19.2) for the compatibility between the refined intersection product of algebraic cycle with the cup-product of cycle classes with support.

\end{proof} 
 
There is a Leray-Serre spectral sequence
\[ E_2^{i,j} = \Ext^i_{\cC}(1_{\cC}, H^{j-i}(Z, k)) \imp H^j_{\cC}(Z, k). \] As $[D]|_Z = c_1(\L)|_Z = 0 \in H^2(Z,1)$, we obtain a class $[D \cap Z] \in \Ext^1_{\cC}(1_{\cC}, H^1(Z, 1))$.

\begin{lem}\label{lem:cohom-of-circle}
The map $\gamma \from H^1(Z,\Q(1)) \to \Q$ of abelian groups is in fact an isomorphism in $\cC$, \[ \gamma \from H^1(Z, 1) \isom 1_{\cC}(1). \]
\end{lem} 
\begin{proof}
It suffices to prove that, for $Z$ a circle of $\P^1$'s, $H_1(Z,0) \isom 1_{\cC}$. In fact, it suffices to consider the nodal curve $\P^1/(0 = \infty)$. We can identify $H^1(\P^1/(0 = \infty)) \isom H^1((\P^1, \{ 0 , \infty \}))$. The sequence
 \[ 0 \to H^0(\P^1) \to H^0(\{ 0 , \infty \}) \to H^1((\P^1, \{ 0 , \infty \})) \to 0 \]
 is exact, and $H^0(\{0, \infty \})/H^0(\P^1) = 1_{\cC}$.
\end{proof}

\begin{lem}\label{lem:weight-of-circle}
The map $H^1(X,1) \to H^1(Z,1)$ is zero.
\end{lem}
\begin{proof}
Although we have been working in weak mixed Hodge structures, the proof of the previous lemma actually implies that $H^1(Z,1)$ is pure of Hodge weight 0. Since $H^1(X,1)$ is pure of Hodge weight -1, the map is zero.
\end{proof}

Lemmas \ref{lem:cohom-of-circle}, \ref{lem:weight-of-circle} were used above (\ref{extn-in-geometry}) to define an extension $E \in \Ext^1_{\cC}(1_{\cC}, 1_{\cC}(1))$ as a subspace of $H^2_c(X-Z, 1)$. We then apply the following result in homological algebra: 
 
\begin{lem}\label{lem:geom-extn}
$[D \cap Z]$ equals the extension $E$ defined in (\ref{extn-in-geometry}).
\end{lem}

%In the case $\cC = G_L\text{-mod}$,
Roughly, this result identifies certain extension classes created via the Leray-Serre spectral sequence for $R\Hom_{\cC}(1_{\cC},-)$ (the class $[D \cap Z]$) with those arising from long exact sequences associated to objects in $\Ch^+(\cC)$ (the extension $E$).

\begin{proof}
\textbf{\'Etale case:}

This follows immediately from a homological result due to Jannsen (\cite{Jannsen-MixedRealizations} 9.4). It is easier to refer to Scholl (\cite{Scholl} 2.7), which gives a more concise statement of this result and an example of its application. Our application is slightly different than in \cite{Scholl} - we use the same functor $\Phi := (-)^{G_L} \from G_L\text{-mod} \to \mathrm{Vec}_{\Q_l}$, but different complexes. We use the complexes \[ A^* = R\Gamma(X_{\bar{L}}, \Q_l(1)),\ B^* = R\Gamma(Z_{\bar{L}}, \Q_l(1)),\ C^* = R\Gamma_c(Y_{\bar{L}}, \Q_l(1))[1],\]
for $Y = X - Z$ (really, as in \cite{Scholl}, representatives of these complexes such that there is an exact sequence of complexes $0 \to A^* \to B^* \to C^* \to 0$).
 
%\todo{Where is it verified that $Rep^{cts}_{G_L}(\Z_l)$ has enough injectives, and that etale cohomology with $l$-adic coefficients lands in $K^+(Rep_{G_L}(\Z_l))$? We could use the homological algebra result with mod $l^n$ coefficients, and take limits} 
 
We have a class $[D] \in H^2(X_L, \Q_l(1)) = H^2(R\Phi(A^*))$ contained in \[ \ker(H^2(R\Phi(A^*)) \to H^2(R\Phi(B^*)) \to \Phi(H^2(B^*))), \] as its image along \[ H^2(R\Phi(A^*)) \to H^2(R\Phi(B^*)) \to \Phi(H^2(B^*)) = H^2(Z_{\bar{L}}, \Q_l(1))^{G_L} \nmisom{\deg} \oplus_{\mathrm{irred(Z)}} \Q_l \] is zero by the assumption that $D$ is component-wise degree zero on $Z$. Then (\cite{Scholl} 2.7) implies that an extension obtained from $[D]$ by the Leray-Serre spectral sequence (the map $\sigma$ in loc. cit.) agrees with an extension obtained from $[D]$ via the following extension of $G_L$-modules (the map $\tau$ in loc. cit.):
\begin{align*} 0 \to \coker(H^1(X_{\bar{L}},\Q_l(1)) &\to H^1(Z_{\bar{L}}, \Q_l(1))) \to H^2(X_{\bar{L}}, \Q_l(1)) \\ 
& \to \ker(H^2(X_{\bar{L}}, \Q_l(1)) \to H^2(Z_{\bar{L}}, \Q_l(1))) \to 0. 
\end{align*}

\textbf{Hodge case:}

Scholl \cite{Scholl} suggests that an analog of Jannsen's result should hold for mixed Hodge complexes. We prove this, for weak mixed Hodge complexes, in Proposition \ref{prop:hodge-extn}. We now describe the application of this proposition. Although it ought to be possible to prove a derived version of Proposition \ref{prop:hodge-extn}, we have not. Thus, we take care to work with honest morphisms of weak mixed Hodge complexes, coming from the functorial weak mixed Hodge complexes for smooth varieties described in Appendix \ref{app:functorial-mhc}. 

We have functorial weak mixed Hodge complexes for smooth varieties, given as above by $C^*(-) := (S^*_{\infty}(-(\C), \Q), F^{\cdot})$. The weak mixed Hodge complex for the singular variety $Z$  can be represented by $C^*(Z) := \Cone(\oplus_i C^*(Z_i) \to \oplus_{i,j} C^*(Z_i \cap Z_j))[-1]$, as the varieties $Z_i$, $Z_i \cap Z_j$ are smooth. The inclusion $Z \subset X$ induces a map $C^*(X) \to C^*(Z)$, and we have $C^*(X,Z) = \Cone(C^*(X) \to C^*(Z))[-1]$. We have a map $f \from C^*(X,Z) \to C^*(Z)$, with $C^*(Z)$ isomorphic to $\Cone(f)$. We apply Proposition \ref{prop:hodge-extn} with \[ A^* = C^*(X,Z),\ B^* = C^*(X),\ f \from C^*(X,Z) \to C^*(X), \] and $\wt{\alpha} := [D \cap Z] \in H^2_{\DB}(Z, \Q(1)) \isom H^2(\Cone(f))$.

% C^*(Y) := (S^*(Y(\C), \Q), F^{\cdot})$, with a map $C^*(X) \to C^*(Y)$. %Dually, we have $f \from \Hom(C^*(Y), \Q(d)) \to \Hom(C^*(X), \Q(d))$. We apply Proposition \ref{prop:hodge-extn} with $A^* = \Hom(C^*(Y), \Q(d))[d], B^* = \Hom(C^*(X), \Q(d))[d]$, and this map $f$. By Poincare duality, $\Cone(f)$ is quasi-isomorphic to $C^*(Z)$.

%, where the inclusion $. We have a triangle of Hodge complexes $C^*(X,Z) = \Cone(C^*(X) \to C^*(Z))[-1]$. This has a map $f \from C^*(X,Z) \to C^*(X)$, such that $\Cone(f)$ is quasi-isomorphic to $C^*(Z)$. We apply Proposition \ref{prop:hodge-extn} with $A^* = C^*(X,Z),\ B^* = C^*(X)$, $f \from C^*(X,Z) \to C^*(X)$.

% apply Proposition \ref{prop:hodge-extn} with $A^* = S_c^*(Y(\C), \Q\ConeS^*(X(\C), \Q), F^{\cdot}), B^* = (S^*((X-Z)(\C), \Q), F^{\cdot})$, $f \from A^* \to B^*$ the map of these functorial weak mixed Hodge complexes induced by $X - Z \subset Z$. 
\end{proof}

Via the isomorphism $\gamma \from H^1(Z, 1) \isom 1_{\cC}(1),$ we identify $[D \cap Z] \in \Ext^1_{\cC}(1_{\cC}, 1_{\cC}(1))$ with a class $\varphi'(D \cap Z) \in \Ext^1_{\cC}(1_{\cC}, 1_{\cC}(1))$. It is not hard to see that $\varphi'$ is a homomorphism
\[ \varphi' \from \text{ component-wise degree $0$ Cartier divisors on $Z$ } \to \Ext^1_{\cC}(1_{\cC}, 1_{\cC}(1)). \]

Let $D' = D \cap Z$, $\L' = \L|_{Z}$. Note that $D'$ is a Cartier divisor on $Z$ with is degree zero on each component $Z_i$, and $\L'$ is the associated line-bundle. Therefore, to prove Theorem \ref{thm:line-bundle-extn}, it suffices to prove the following proposition, which depends only on $Z$ and not the ambient variety $X$:
\begin{prop}
Consider a component-wise degree zero Cartier divisor $D'$ on $Z$ with associated line-bundle $\L'$. Then
\[ \kappa(\varphi(\L')) = \varphi'(D'). \]
\end{prop}

In other words, we compute the Hodge/\'etale Chern class of certain Cartier divisors on $Z$, and relate them to the homomorphism $\pi_1(Z(\C)) \to L^*$ associated to the local system underlying $\L'$. While this seems like it should be well-known, as a special case of a more general result about Abel-Jacobi maps for singular varieties, we prove this for lack of a reference.

\begin{proof} 

We first compute $\varphi(\L')$ in terms of the divisor $D'$:
\begin{lem}
Consider the divisor $D' = \sum n_i P_i$. Then \[ \varphi(\L') = \prod z_i(P_i)^{n_i}. \]
\end{lem}
\begin{proof} 
Let $s$ be any meromorphic section of $\L$ such that $\div(s)$ is disjoint from $Z^{\mathrm{sing}} = \coprod Z_i \cap Z_{i+1} = \coprod 0_i$. Then $s|_{Z_i}$ corresponds to a function $f_i \in L(z)$ via any choice of $\L|_{Z_i} \isom \O_{Z_i}$. We have \[ \prod_i \frac{f_i(0_i)}{f_i(\infty_i)} = \varphi([\L]) \in L^*, \]
as Lemma \ref{lem:line-bundle-rational-var} relates $\varphi$ to transition functions along the intersection of irreducible components.

For a function $f \in L(z)$ such that $\div(f) = \sum n_j [P_j]$ is disjoint from $\{0, \infty \}$, we have \[ \prod z(P_j)^{n_j} = \frac{f(0)}{f(\infty)}. \] This is because $f(z) = f(\infty)\prod_i (z-P_i)^{n_i}$ with $\sum n_i = 0$, and so $\frac{f(0)}{f(\infty)} = (-P_i)^{n_i} = (-1)^{\sum n_i} P_i^{n_i} = P_i^{n_i}$. Note that this is a special case  of Weil reciprocity.
% $f(z) = f(\infty)\frac{(z-P_i)^n_i}$, so $\frac{f(0)}{f(\infty)} = (-P_i)^{n_i}$.
\end{proof}

We must then show that $\varphi'(D') \in \Ext^1_{\cC}(1_{\cC}, 1_{\cC}(1))$ is given by the product $\prod z_i(P_i)^{n_i}$. As the divisor $D'$ is degree zero on each component of $Z$, it is easy to see that we may reduce to the case that $Z = \P^1/(0 = \infty)$, and assume that $D' = [a] - [b]$. This may require passing to a finite Galois extension $L'/L$, so that all geometric points of $D \cap Z$ are defined over $L'$. This is not a problem as $\Ext^1_{G_L}(\Q_l, \Q_l(1)) = \Ext^1_{G_{L'}}(\Q_l, \Q_l(1))^{\Gal(L'/L)}$. We may further assume that $b = 1$, since the coordinate $z$ on $\P^1$ was only defined up to a constant.

We have an isomorphism 
\[ \Ext^1_{\cC}(1_{\cC}, H^1(\P^1/(0 = \infty))) \isom \Ext^1_{\cC}(1_{\cC}, H^1(\P^1/(0 = \infty) - \{ 1 \})) \] By the isomorphism $\P^1/(0 = \infty) - \{ 1 \} \isom \A^1/(0=1)$, this extension class is identified with the image of \[ a \in L^* \tensor \Q = H^1_{\Mot}(\Spec(L), \Q(1))\] under the regulator \[r_{\cC} \from H^1_{\Mot}(\Spec(L), \Q(1)) \to \Ext^1_{\cC}(1_{\cC}, 1_{\cC}(1)),\] following the description of the regulator on Bloch's higher Chow groups given in (\cite{Scholl}). The regulator is well-known to equal the map $\kappa \from L^* \tensor \Q \to \Ext^1_{\cC}(1_{\cC}, 1_{\cC}(1))$ defined above. For completeness, we review this in Appendix \ref{appendix:reg-on-point}.

\end{proof}

\subsection{A conjectural compatibility of Chern classes}
We describe a conjectural generalization of our Theorem \ref{thm:line-bundle-extn}, motivated by the possibility of generalizing results of Looijenga \cite{Looijenga} from the Hodge to the \'etale setting.

Consider a number field $L \subset \C$, with integers $\O_L$. Consider a smooth projective variety $X/L$ equipped with a toroidal embedding $Z \subset X$. By toroidal embedding, we mean a simple normal crossings divisor where every irreducible component $Z_i$ is a smooth proper toric variety over $L$, every intersection $Z_i \cap Z_j$ is a toric divisor in both $Z_i$ and $Z_j$, etc.

Suppose we have a rank $n$ vector bundle $\mathcal{V}$ on $X$ such that $\mathcal{V}|_Z$ is trivial on each irreducible component. The ideas of Section \ref{subsec:hodge-bundle} prove that $\mathcal{V}$ is the scalar extension of a local system $\rho \from \pi_1(Z(\C)) \to \GL_n(L) \into \GL_n(\C)$. Moreover, suppose that $\rho$ factors through $\GL_n(\O_L)$. Abusing notation, we also let $\rho$ denote the map $\rho \from \pi_1(Z(\C)) \to \GL_n(\O_L)$.

As $Z$ is a toroidal embedding, we have a map $s \from Z(\R_{\geq 0}) \into Z(\C)$, generalizing the special case described in Section \ref{sec:boundary-top}. This can be shown to induce a map $H^*(Z, 0) \to H^*(Z(\R_{\geq 0}), \Z) \tensor 1_{\cC}$ in $\cC$. We have a Chern class $c_i(\mathcal{V})|_Z \in \Ext^1_{\cC}(1_{\cC}, H^{2i-1}(Z, i))$, which we can pull-back to $s^*(c_i(\mathcal{V})|_Z) \in \Ext^1_{\cC}(1_{\cC}, 1_{\cC}(i)) \tensor H^{2i-1}(Z(\R_{\geq 0}), \Z)$.

On the other hand, we have the Borel and Soul\'e regulators (\cite{Looijenga}, \cite{Soule})
\begin{align*}
 b_i &\in  \Ext^1_{\MHS}(\Q, \Q(i)) \tensor H^{2i-1}(B\GL_n(\O_L), \Z) \\
 b_i &\in \Ext^1_{G_L}(\Q_l, \Q_l(i)) \tensor H^{2i-1}(B\GL_n(\O_L), \Z). 
\end{align*} These can be pulled back along $f_{\rho} \from Z(\R_{\geq 0}) \to B\pi_1(Z(\R_{\geq 0})) \nmto{B\rho} B\GL_n(\O_L)$ to give classes $f_{\rho}^*(b_i) \in \Ext^1_{\cC}(1_{\cC}, 1_{\cC}(i)) \tensor H^{2i-1}(Z(\R_{\geq 0}), \Z)$.

\begin{conj}\label{conj:toroidal-chern}
$f_{\rho}^*(b_i) = s^*(c_i(\mathcal{V})|_Z)$
\end{conj}

\begin{eg}
Suppose $i = 1$, $\pi_1(Z(\R_{\geq 0})) = \langle u \rangle$, and $\mathcal{V}$ is a line-bundle corresponding to the map $\rho \from \Z \to \O_L^*$, $\rho(u) = \alpha$. The Borel regulator $b_1 \in H^1(B\GL_1(\O_L), \Z) \tensor \C/\Q(1)$ corresponds to the homomorphism $\O_L^* \into \C^* \tensor \Q$, and this conjecture reduces to the claim that $s^*(c_1(\mathcal{V}|_Z))(u) \in \C^* \tensor \Q$ equals $\alpha$.
\end{eg}

Looijenga's methods \cite{Looijenga}, following \cite{Reznikov}, should be enough to prove Conjecture \ref{conj:toroidal-chern} in the Hodge case. The \'etale case seems more difficult. 

Our Theorem \ref{thm:line-bundle-extn} essentially consists of a proof of this conjecture in the case that $Z$ has dimension $1$ and $\mathcal{V}$ is a line-bundle, together with the use of Lemma \ref{lem:geom-extn} to prove that the extension class appears in the cohomology of $H^{2i-1}((X, Z), i)$.

\section{Extensions in Eisenstein cohomology}\label{sec:extn-classes}
We use the notation introduced in Section \ref{sec:line-bundle-extn}, \begin{align*}
\cC &= G_F\text{-mod or }\MHS_w, \\ 
H^*(-,*) &= \text{ geometric \'etale or singular cohomology, } \\
1_{\cC}(i) &= \Q_l(i) \text{ or } \Q(i),
\end{align*}
in order to prove results about extensions of Galois modules or (weak) mixed Hodge structures simultaneously.

\subsection{Non-split extensions in degree $2d-2$}\label{subsec:nonsplit-extn}
Consider the variety $Y := Y(1)_{\Q}$. This is a possibly singular variety, the coarse space of a smooth stack. 

 We studied the singular cohomology $H^*(Y, \Q)$ in Sections \ref{subsec:Eis-MHS}, \ref{subsec:const-residue}, obtaining an exact sequence
\[ 0 \to (\oplus_i \Q \omega_i^*) \tensor \Q(1) \to H^{2d-2}(Y, \Q(d))_{\Eis} \nmto{\del} \O^* \tensor \Q \to 0. \] 
We verified that this exact sequence is compatible with mixed Hodge structures, where the abelian groups $\oplus_i \Q \omega_i^*$, $\O^* \tensor \Q$ are given the trivial mixed Hodge structure, and $\Q(1)$ is the Tate mixed Hodge structure. A similar result holds in the \'etale setting:

\begin{prop}\label{prop:boundary-actions}
After tensoring with $\Q_l$, this is an exact sequence of $G_F$-modules
\[ 0 \to (\oplus_i \Q_l \omega_i^*) \tensor \Q_l(1) \to H^{2d-2}_{\Et}(Y_{\bar{F}}, \Q_l(d))_{\Eis} \nmto{\del} \O^* \tensor \Q_l \to 0, \] where the abelian groups $\oplus_i \Q \omega_i^*$, $\O^* \tensor \Q$ are given the trivial Galois action.
\end{prop}
\begin{comment}
\begin{enumerate}
\item The classes $\omega_i^* = \omega_1 \wedge \ldots \wedge \wh{\omega_i} \wedge \ldots \wedge \omega_d \in H^{2d-2}(Y,\Q(d-1))_{\univ}$ give an isomorphism
\[ \oplus_{i=1}^d 1_{\cC}(1) \omega_i^* \isom H^{2d-2}(Y,d)_{\univ}. \]
\item The quotient $\frac{H^{2d-2}(Y, 1_{\cC}(d))_{\Eis}}{H^{2d-2}(Y, 1_{\cC}(d))_{\univ}}$ is isomorphic to $\O^* \tensor 1_{\cC}$.
\end{prop}
\end{comment}
\begin{proof}

It will suffice to check that the modules \begin{enumerate}
\item $H^{2d-2}_{\Et}(Y_{\bar{F}}, \Q_l(d-1))_{\univ}$
\item $H^{2d-2}(Y_{\bar{F}}, \Q_l(d))_{\Eis}/H^{2d-2}(Y_{\bar{F}}, \Q_l(d))_{\univ}$
\end{enumerate} have trivial $G_F$-action. We can even check this after restriction to $G_{F(\mu_N)}$ for all $N \geq 3$.

(1) follows from the remark at the end of \ref{subsec:hodge-bundle} about the Chern classes of the Hodge bundles descending to level 1, by their $\SL_2(\O/N\O)$-invariance.

(2) We pass to auxiliary level $N \geq 3$, and use the results of Section \ref{subsec:const-residue} to relate the map $\del$ to the algebraically defined residue maps. We have a variety $Y(N)/\Q(\mu_N)$, and a finite map $\phi_N \from Y(N) \to Y(1)_{\Q(\mu_N)}$. On \'etale cohomology, this map induces an injective map of $G_{\Q(\mu_N)}$-modules
\[ \phi_N^* \from H^*(Y_{\bar{\Q}}, \Q_l(d))_{\Eis} \to H^*(Y(N)_{\bar{\Q(\mu_N)}}, \Q_l(d))_{\Eis}. \]

It will suffice to show that the $G_{\Q(\mu_N)}$-action on $\frac{H^*(Y(N)_{\bar{\Q(\mu_N)}}, \Q_l(d))_{\Eis}}{H^{2d-2}(Y(N)_{\bar{\Q(\mu_N)}}, \Q_l(d))_{\univ}}$ is trivial. %We have shown (via Section \ref{subsec:const-residue}) that 

 %However, if we prove that the Galois module $H^{2d-2}(Y_{\bar{F}}, \Q_l(d))_{\Eis}/H^{2d-2}(Y_{\bar{F}}, \Q_l(d))_{\univ}$ is trivial when restricted to $G_{\Q(\mu_N)}$ for all $N \geq 3$, then it will be trivial as a $G_{\Q}$-module.

We proved (via \ref{prop:MHS-eis-cohom} and \ref{prop:extn-compat}) that $\ \frac{H^*(Y(N)_{\bar{\Q(\mu_N)}}, \Q_l(d))_{\Eis}}{H^{2d-2}(Y(N)_{\bar{\Q(\mu_N)}}, \Q_l(d))_{\univ}} \into H_1(\del X(N)_{\bar{\Q(\mu_N)}}, \Q_l(0))$ (noting that the toroidal residue maps are compatible with Galois actions).

%As MHS, $H_1(\del X(N), \Q(0))^{\SL_2(\O/N\O)} \isom T(\Z) \tensor \Q(0)$. The same holds for $G_{\Q(\mu_N)}$-modules, and for all of $H_1(\del X(N)_{\bar{\Q(\mu_N)}}, \Q(0))$:
\begin{lem}
There is an isomorphism
\[ H_1(\del X(N)_{\bar{\Q(\mu_N)}}, \Q_l) \isom (\oplus_{\Cusps(\Gamma(N))} T(\Z)) \tensor \Q_l, \]
where the right side is given the trivial Galois action.
\end{lem} 
\begin{proof}
As abelian groups, the identification follows from Lemma \ref{lem:borel-serre-units}. Using Corollary \ref{boundary-circles}, this follows from \ref{lem:weight-of-circle}.
\end{proof}

%We have an injection \[ \frac{H^{2d-2}(Y_{\bar{\Q}}, \Q_l(d))_{\Eis}}{H^{2d-2}(Y_{\bar{\Q}}, \Q_l(d))_{\univ}} \into \frac{H^{2d-2}(Y(N)_{\bar{\Q(\mu_N)}}, \Q_l(d))_{\Eis}}{H^{2d-2}(Y_{\bar{\Q(\mu_N)}}, \Q_l(d))_{\univ}} \isom (\oplus_{\Cusps(\Gamma(N))} T(\Z)) \tensor \Q_l. \]

%This proves the triviality of the $G_{\Q(\mu_N)}$ action on $\frac{H^{2d-2}(Y_{\bar{\Q}}, \Q_l(d))_{\Eis}}{H^{2d-2}(Y_{\bar{\Q}}, \Q_l(d))_{\univ}}$.

\end{proof}

%Passing to $\SL_2(\O/N\O)$-invariants, we obtain
%\[ 0 \to H^{2d-2}(Y(\C), \Q(d))_{\univ} \to H^{2d-2}(Y(\C), \Q(d))_{\Eis} \nmto{\Res \circ \phi_N^*} (\oplus_{\Cusps(\Gamma(N))} T_N(\Z) \tensor \Q)^{\SL_2(\O/N\O)} \to 0. \]

%In order to make it easier to pass between $Y(N)$ and $Y(1)$, we normalize these maps, defining
%\[ \del := \frac{1}{\vol(Y(N))} \del_N, \]
%with $\vol(Y(N)) := \bar{\omega_1} \cup \ldots \cup \bar{\omega_d} \in H^{2d}(X(N), \Q(d)) \isom \Q$.

%We want to describe the sub/quotient of this extension as mixed Hodge structures/$G_{K}$-modules. We do this by passing to auxiliary level $N \geq 3$, where $Y(N)$ is a smooth variety over $\Q(\mu_N)$, with smooth projective compactification $X(N)$. 
%This allows us to study the mixed Hodge structure at level 1 by passing to level $N$. Similarly, we can study the restriction of the $G_{\Q}$-module $H^{2d-2}(Y_{\bar{\Q}}, \Q_l(d))_{\Eis}$ to $\Q(\mu_N)$ by studying the cohomology at level $N$.

%of $\Q$-MHS. As mentioned in Section \ref{subsec:const-residue}, this is also a sequence of $G_{\Q}$-modules.  %Recall that $\O^* \tensor \Q \isom H_1(\del X, \Q(0))$ (Lemma \ref{prop:basic-top}).

%Now, we must base-change from $\Q$ to $K$. Having done this, we have an isomorphism $H^{2d-2}(Y_F, d)_{\univ} \isom \oplus_{i=1}^d 1_{\cC}(1) \omega_i^*$.

We obtain an extension $E = H^{2d-2}(Y, d)_{\Eis}$, with
\[ 0 \to \oplus_{i=1}^d 1_{\cC}(1) \omega_i^* \to E \nmto{\del} \O^* \tensor 1_{\cC} \to 0. \]
It is classified by maps
\begin{align*}
m_{G_F} \from \O^* \to \oplus_i F^* \tensor \Q_l,\\
m_{\MHS} \from \O^* \to \oplus_i \C^* \tensor \Q.
\end{align*}

We have now justified the statement of Theorem \ref{thm:etale-extn-class}. To prove both Theorem \ref{thm:eis-periods-1} (via the reductions \ref{prop:period-to-hodge}, \ref{prop:extn-compat}) and Theorem \ref{thm:etale-extn-class}, it will suffice to show:
\begin{thm}\label{thm:classify-eis-extns}
The maps $m_{G_F}$ and $m_{\MHS}$ are both equal to the map
\[ u \mapsto -(\sigma_1(u),\ldots, \sigma_d(u)) \in \oplus_i F^* \tensor \Q. \]
\end{thm}

\begin{proof}
\textbf{Reduction to $Y(N)$:}

We begin by reducing the theorem to a statement about the smooth variety $Y(N)/\Q(\mu_N)$. There is a finite map $\phi_N \from Y(N) \to Y(1)_{\Q(\mu_N)}$. On singular cohomology, this map induces
\[ \phi_N^* \from H^*(Y(\C), \Q)_{\Eis} \to H^*(Y(N)(\C), \Q)_{\Eis}, \]
identifying \[ H^*(Y(\C), \Q)_{\Eis} = H^*(Y(N)(\C), \Q)_{\Eis}^{\SL_2(\O/N\O)}. \]

We want to rewrite the quotient map in a way which is independent of the auxiliary level $N$. For $N \geq 3$, we have a map 
\begin{align*}
\del_N \from H^{2d-2}(Y(N), \Q(d)) &\nmto{\Res} \oplus_{x \in \Cusps(\Gamma(N))} H_1((\del X)_x, \Q(0)) \\ &\isom \oplus T(\Z) \tensor \Q \nmto{\prod} \O^* \tensor \Q.
\end{align*}
Note $T(\Z) \subset \O^*$, so that $T(\Z) \tensor \Q \isom \O^* \tensor \Q$. Here, $(\del X)_x$ denotes the connected component of $\del X$ corresponding to the cusp $x \in \Cusps(\Gamma(N))$.

%For $N = 1$, we have only 1 cusp (by our assumption that $h_{K} = 1$), and this map induces an isomorphism $\del_1 \from H^{2d-2}(Y(1), \Q(d))_{\del} \isom \O^* \tensor \Q$.

The map \[ \del_N \circ \phi_N^* \from H^{2d-2}(Y(1), \Q(d)) \to \O^* \tensor \Q \] induces an isomorphism
$\del_N \circ \phi_N^* \from H^{2d-2}(Y(1), \Q(d))_{\del} \isom \O^* \tensor \Q$ (using the compatibility \ref{prop:residue-compat}).

\begin{lem}\label{lem:bdry-compat-level}
There is an equality
\[ \del = \frac{1}{\vol(Y(N))}\del_N \circ \phi_N^* \] 
of maps $H^{2d-2}(Y(1), \Q(d)) \to \O^* \tensor \Q$.
\end{lem} 
The proof is easy, using the definition of $\del_N$ in terms of Poincare duality.

%Recall that, by Theorem \ref{thm:mumford-proportionality} (for $N \geq 3$), $\vol(Y(N)) = |\zeta_F(-1)||\SL_2(\O/N\O)|$.

We obtain an extension
\[ 0 \to H^{2d-2}(Y(N)(\C), \Q(d))_{\univ} \to H^{2d-2}(Y(N)(\C), \Q(d))_{\Eis} \nmto{\del_N} \O^* \tensor \Q \to 0. \]
of $\Q$-vector spaces. It is easy to see that Proposition \ref{prop:boundary-actions} holds for this extension as well, giving an extension
\begin{align}\label{extn-in-aux-level} 0 \to \oplus_{i=1}^d 1_{\cC}(1) \omega_i^* \to E' \nmto{\del_N} \O^* \tensor 1_{\cC} \to 0 
\end{align}
in $\cC = \MHS_w$ or $G_{F(\mu_N)}$-modules. By Lemma \ref{lem:bdry-compat-level}, $\phi^*_N$ induces an isomorphism of this extension with the extension $E = H^{2d-2}(Y, 1_{\cC}(d))$ (after restricting from $G_F$ to $G_{F(\mu_N)}$ in the \'etale setting). However, this restriction in the \'etale setting is not a problem, as $\Ext^1_{G_F}(\Q_l, \Q_l(1)) = Ext^1_{G_{F(\mu_N)}}(\Q_l, \Q_l(1))^{\Gal(F(\mu_N)/F)}$. Thus we may forget about $Y(1)$, and prove a result about an extension occuring in the cohomology the variety $Y(N)$.

\textbf{Chern classes of $\omega_i$ on $Y(N)$:} 

We now start the main part of the proof. Let $L = F(\mu_N)$, $Y = Y(N)_{L}$, $X$ a smooth projective toroidal compactification. We consider the exact sequence \[ 0 \to H^{2d-2}(X, d)^{\perp \del X} \to H^{2d-2}(Y, d) \to H_1(\del X, 0) \to \ldots. \]

%This sequence is the Poincare dual of the LES of compactly supported cohomology: \[ \ldots \to H^1_c(\del X, 0) \to H^{2}_c(Y,0) \to H^2_c(X,0) \to \ldots. \] 

%In singular cohomology, we have a factorization \[ H^{2d-2}(Y, \Z(d)) \to H^{2d-2}(\H^d/P, \Z(d)) \to H_1(\del X, 0). \] This final map is an isomorphism $H^{2d-2}(\H^d/P, \Q(d)) \isom H_1(\del X, \Q(0)) = T(\Z) \tensor \Q$.

The map $\Res \from H^{2d-2}(Y, \Q(d)) \to H_1(\del X, \Q)$ is surjective for weight reasons (Lemma \ref{lem:weight-of-circle}). This gives an exact sequence:
\[ 0 \to H^{2d-2}(X, d)_{\perp \del X} \to H^{2d-2}(Y, d) \to H_1(\del X, 0) \to 0. \] 

We have $H_1(\del X, \Q(0)) \isom \oplus_{\Cusps(\Gamma(N))} T(\Z) \tensor \Q$. Consider a cusp $x \in \Cusps(\Gamma(N))$ corresponding to the connected component $(\del X)_x$ of $\del X$, and $\gamma_{u,x} \in H_1((\del X)_x, \Z)$, corresponding to $u \in T(\Z) \subset \O^*$ via the isomorphism $H_1((\del X)_x, \Z) \isom T(\Z)$. By \ref{boundary-circles}, $\gamma_{u,x}$ is in the image of $H_1(Z, \Z(0)) \to H_1((\del X)_x, \Z(0))$, for $Z \subset \del X$ a circle of $\P^1$'s. We pull back the above extension by $H_1(Z, 0) \to H_1((\del X)_x, 0) \to H_1(\del X, 0)$, and obtain an extension
\[ 0 \to H^{2d-2}(X, d)_{\perp \del X} \to E \to H_1(Z, 0) \to 0. \]
This extension is Poincare dual to the pull-back of
\[ 0 \to H^1(Z, 1) \to H^2(X, Z) \to \ker(H^2(X, 1) \to H^2(Z, 1)) \to 0 \]
along $\ker(H^2(X, 1) \to H^2(\del X, 1)) \to \ker(H^2(X, 1) \to H^2(Z, 1))$. We know that \[c_1(\bar{\L_i}) \in \ker(H^2(X, 1) \to H^2(\del X, 1)), \] for $\bar{\L_i}$ the Hodge line-bundles on $X$ (defined over $F$). Thus we may apply Theorem \ref{thm:line-bundle-extn} to see that the extension $0 \to 1_{\cC}(1) \to E' \to 1_{\cC} \to 0$ given by the following ``framing"
\[ \begin{tikzcd}[column sep = small]
  & & & 1_{\cC} \arrow{d}{c_1(\bar{\L}_i)} & \\
 0 \arrow{r} & H^1(\del X, 1) \arrow{r} \arrow{d}{\gamma_{u,x}} & H^2(X, \del X) \arrow{r} & \ker(H^2(X, 1) \to H^2(\del X, 1)) \arrow{r} & 0 \\
  & 1_{\cC}(1) & & & \\
\end{tikzcd} \]
corresponds to $\sigma_i(u) \in F^*$, since $\bar{\L_i}$ restricted to any connected component of $\del X$ is, by Proposition \ref{hodge-on-boundary}, the scalar extension of the local system corresponding to the homomorphism $\sigma_i \from T(\Z) \to F^*$. 

Using Poincare duality and twisting by $1_{\cC}(1)$, we find that the extension framed by
\[ \begin{tikzcd}[column sep = small]
  & & & 1_{\cC} \arrow{d}{\gamma_{u,x}} & \\
 0 \arrow{r} & H^{2d-2}(X, d)_{\perp \del X} \arrow{r} \arrow{d}{\cup \bar{\omega_i}} & H^{2d-2}(Y, d) \arrow{r}{\Res} & H_1(\del X, 0) \arrow{r} & 0 \\
  & 1_{\cC}(1) & & & \\
\end{tikzcd} \]
corresponds to $-\sigma_i(u) \in F^*$ (the sign arises from \ref{lem:dual-extn}). 

Now, consider the framing 
\begin{align*}\psi := \oplus_{i=1}^d \frac{1}{\vol(Y(N))} (\cdot \cup \bar{\omega_i}) \from & H^{2d-2}(X, d)_{\perp \del X} \to \oplus_{i=1}^d 1_{\cC}(1), \\ 
\end{align*}
\begin{align*}
\O^* \tensor \Q &\to \oplus_{\Cusps(\Gamma(N))} T(\Z) \tensor \Q \isom H_1(\del X, \Q) \\
u &\mapsto \frac{\vol(Y(N))}{|\Cusps(\Gamma(N))|} (u,\ldots,u). \\
\end{align*} 
 The sub-quotient of $H^{2d-2}(Y, d)$ obtained from this framing is an extension
\[ 0 \to \oplus_{i=1}^d 1_{\cC}(1) \to E'' \to \O^* \tensor 1_{\cC} \to 0, \]
classified (in both $G_L$-mod and $\Q$-MHS) by the map \[-(\sigma_1,\ldots, \sigma_d) \from \O^* \to \oplus_{i=1}^d F^* \tensor \Q. \]

It remains to prove $E'' \isom E'$, with $E'$ the extension defined above (\ref{extn-in-aux-level}), compatible with the identification of the sub/quotient. As $E' = H^{2d-2}(Y,d)_{\Eis} \subset H^{2d-2}(Y,d)$, this involves showing that the framing used to defined $E''$ induces the identity on the sub/quotient of $E'$.

\textbf{Sub:}
Identify $H^{2d-2}(X, d-1)_{\perp \del X} = \im(H^{2d-2}(X, d-1) \to H^{2d-2}(Y, d-1))$. The latter group contains the classes $\omega_1^*,\ldots,\omega_d^*$. We need to show that the map $\psi \from H^{2d-2}(X, d-1)_{\perp \del X} \to \oplus_{i=1}^d 1_{\cC}$ sends $\omega_i^*$ to the i-th standard vector $v_i := (0,\ldots,0,1,0,\ldots, 0)$. Lift $\omega_i^*$ to $\bar{\omega_i}^* := \bar{\omega_1} \wedge \ldots \wedge \wh{\omega}_i \wedge \ldots  \wedge \bar{\omega_d}$ in $H^{2d-2}(X, d-1)$. We see that $\psi(\bar{\omega_i}^*) = \vol(Y(N))/\vol(Y(N)) v_i = v_i$ as required. 

\textbf{Quotient:} The composition $E \subset H^{2d-2}(Y, d) \nmto{\del_N} H_1(\del X, 0)$ induces a map \[ \O^* \tensor 1_{\cC} \isom E/\ker(\del) \to H_1(\del X, 0), \] which can be easily seen to be the same as the map $\O^* \tensor 1_{\cC} \to H_1(\del X,0)$ of the above framing.
%, it suffices to note that $\frac{1}{\vol(Y(N))} \del_N \from \from E \to \O^* \tensor \Q$ induces a map $ is a sub-module of $H^{2d-2}(Y(N), \Q(d))$ which is isomorphic to $E'$.

%With this framing, the extension is classified by $\frac{\vol(Y(N))}{|\Cusps(\Gamma(N))|}  \sum_{\Cusps(\Gamma(N))} \frac{1}{\vol(Y(N))} - \

% What framing is $\del$ related to? We have a map $\O^* \to \oplus_{\Cusps(\Gamma(N))} \O^*_N \to \O^*$ which is multiplication by $|\SL_2(\O/N\O)|$. The first map is then $u \mapsto \frac{|\SL_2(\O/N\O)|}{|\Cusps(\Gamma(N))|} (u,\ldots,u)$

\begin{comment}
The surjection $H^{2d-2}(X, d)_{\perp \del X} \nmto{\cup c_1(\bar{\L}_i)} 1_{\cC}(1)$ differs from $H^{2d-2}(X, d)_{\perp \del X} \isom \oplus 1_{\cC}(1) \cdot \omega_i^* \to 1_{\cC}(1) \cdot \omega_i^*  \nmto{\frac{1}{\omega_i^*}} 1_{\cC}(1)$ by a factor of $\omega_i^* \wedge \omega_i = 2 |\zeta_F(-1)|$. If we use the latter map in our framing, the full extension is classified by
\[ m \from \O^* \tensor \Q \to \oplus_{v \mid \infty} F^* \tensor \Q, \ m(u) = -\frac{1}{2 |\zeta_F(-1)|} (\sigma_1(u),\ldots, \sigma_d(u)). \] 
\end{comment}
\end{proof}

\subsection{Split extensions in lower degrees}
Let $Y = Y(1)_F$. We show:
\begin{thm}
For $i < 2d-2$, $H^i(Y(1)_F,d)$ is a semisimple $G_F$-module/MHS.
\end{thm}

\begin{proof}
Similarly to in Theorem \ref{thm:classify-eis-extns}, we can immediately reduce to showing that $H^i(Y(N)_L, d)$ is semisimple as a $G_L$-module/MHS for $N \geq 3$. We check this for $G_L$-modules, and the proof for MHS is essentially the same.

Let $Y = Y(N)_L, X = X(N)_L$. We consider the extension of $G_L$-modules 
\[ 0 \to \frac{H^{2i-1}(\del X, i)}{H^{2i-1}(X, i)} \to H^{2i}((X, \del X), i) \to \ker(H^{2i}(X,i) \to H^{2i}(\del X, i)) \to 0. \] By Poincare duality, to prove the theorem it suffices to show that this extension splits for $i > 1$, as both the sub and quotient modules are semisimple. 
 \begin{rmk}
 The fact the the quotient is semisimple follows from a theorem of \cite{Nekovar-Semisimplicity}. The semisimplicity for mixed Hodge structures follows from general results about the semisimplicity of polarized pure Hodge structures.
\end{rmk}

For $i = 1$, the extension \[ 0 \to H^{1}(\del X, 1) \to H^{2}((X, \del X), 1) \to \ker(H^{2}(X,1) \to H^{2}(\del X, 1)) \to 0 \]
is non-split, as we proved in Theorem \ref{thm:classify-eis-extns}. We take the $G_L$-invariant class $\bar{\omega_i} \in \ker(H^{2}(X,1) \to H^{2}(\del X, 1))$, and let $\wt{\omega_i}$ denote any choice of (necessarily not $G_L$-invariant) lift to $H^2((X, \del X), 1)$.

We will show that for $i > 1$, the class $\wt{\omega_{j_1}} \cup \ldots \cup \wt{\omega_{j_i}} \in H^{2i}((X, \del X), i)$ is $G_L$-invariant.

%This extension is the cup-product of the extension 
%\[ 0 \to H^{2i-3}(\del X, i-1) \to E \to \omega_{j_2} \cup \ldots \cup \omega_{j_i} \to 0 \]
%by $\omega_{j_1}$. 

We have a $G_L$-equivariant map
\[ \cup \bar{\omega_{j_1}} \from H^{2i-2}((X, \del X), i-1) \to H^{2i}((X, \del X), i). \]
This map is trivial on the image of $H^{2i-3}(\del X, i-1)$, as the boundary map 
\[ \delta \from H^{2i-3}(\del X, i-1) \to H^{2i-2}(\del X, i-1) \] satisfies $\delta(\eta|_{\del X} \cup \cdot ) = \eta \cup \delta(\cdot)$. We obtain a $G_L$-equivariant map
\[ \cup \bar{\omega_{j_1}} \from H^{2i-2}(X,i-1) \to H^{2i}((X, \del X), i). \] 

Take any (not necc. $G_L$-invariant) lift of $\omega_{j_2} \cup \ldots \cup \omega_{j_i}$ to $H^{2i-2}((X, \del X), i-1)$. We could take $\wt{\omega_{j_2}} \cup \ldots \cup \wt{\omega_{j_i}}$. Its cup-product with $\cup \bar{\omega_{j_1}}$ is independent of this lift, hence is $G_L$-invt.

Therefore the extension $H^{2i}((X, \del X), i)$ splits when pulled back along \[ H^{2i}(X)_{\univ} \subset \ker(H^{2i}(X,i) \to H^{2i}(\del X, i)). \]

This inclusion is an isomorphism, unless $i = d/2$. In this case, note that the sequence \[ H^{d-1}(\del X, d/2) \to H^d((X, \del X), d/2) \to H^d(X, d/2) \] is dual to $H^d(X, d) \to H^d(Y, d) \to H_{d-1}(\del X, 0)$. By the above splitting, we are left with the possibility of a non-trivial extension
\[ 0 \to \im(H^d(X) \to H^d(Y))^{\perp \univ} \to E \to \im(\Res \from H^d(Y, d) \to H_{d-1}(\del X, 0)) \to 0. \]
However, this extension splits for Hecke eigenvalue reasons, as in Proposition \ref{prop:MHS-eis-cohom} (1). \end{proof}

\begin{comment}
TODO: Can we show that there are non-trivial Massey products of the form \[ H^2_{\DB}(X,\Z(1))^{\tensor n} \tensor H^{2d-n}_{\DB}(X, \Z(d)) \to H^{2d+1}_{\DB}(X, \Z(d+n)) \to \Ext^1_{\MHS}(\Z(0), \Z(n)). \]

The Massey product is well-defined modulo $[\L_1] \cup H^{2d-1}_{\DB}(X, \Z(d+n-1)) \subset H^{2d+1}_{\DB}(X, \Z(d+n))$. We have $H^{2d-1}_{\DB}(X, \Z(d+n-1)) \isom \Ext^1_{\MHS}(\Z(0), H^{2d-2}(X, d+n-1)) \nmto{\cup [\L_1]} \Ext^1_{\MHS}(\Z(0), \Z(n))$.

If we instead use $F_{\infty}-\MHS$ with $\R$-coefficients, we can get Massey products as long as $n$ odd. It seems like the motivic Massey product should be well-defined as long as the Hodge one is. Are there classes in $H^1_{\Mot}(\Spec(F), \Z(n))$ related to wedge-products of fundamental units $\log(u_1) \wedge \ldots \wedge \log(u_i)$?
%, as long as $H^{2d+1}_{\Mot}(X, \Z(d+n))$ injects into $H^{1}_{\Mot}(\Spec(K), \Z(n))$.
\end{comment}

\section{Eisenstein series and $(\g,K)$-modules}\label{sec:gK-cohom}
In this section, we will recall some standard facts in the theory of principal series representations, residues of Eisenstein series, and $(\g,K)$-cohomology, specialized to the case $G = \SL_2(\R)^d$. We will treat the relation of this with the group $\GL_2(\R)^d$ in an ad-hoc manner in Section \ref{subsec:partial-conjugation}. 

\subsection{Principal series for $\SL_2(\R)^d$}\label{subsec:principal-series}
Define $G = \SL_2(\R)^d, \g = (\sl_2)^d, \g_{\C} = \g \tensor_{\R} \C, K = (\SO(2))^d, \kk = (\so_2)^d$. We assume, as always, that $F$ is a totally real field of degree $d$, narrow class number one, and we consider $\Gamma := \SL_2(\O_F) \subset G$. We have the Iwasawa decomposition $NAK = G$. Write $P = NA$.

Let $I(s) := \Ind_{P}^{G}(|\cdot|^{s}, |\cdot|^{-s})$ be a principal series representation\footnote{This is the non-normalized, smooth induction. } of $G$. Here $|\cdot| \from (\R^*)^d \to \C^*$ is the character $|(t_1,\ldots, t_d)| := |t_1 \cdots t_d|$.

Via the right-action of $\SL_2(\R)$ on the vector $(1,0) \in \R^{2}- \{0\}$, we identify $N \backslash G \isom (\R^2 - \{0\})^d$. We define an action of $(\R^*)^d$ on $(\R^2)^d$, where $\R^*$ acts via scaling in $\R^2$. We have: \begin{lem}As a $G$-representation, \begin{align*}
I(s) = \{ f \from (\R^2- \{ 0\})^d \to \C \mid f(a \cdot v) = |a|^{-2s} f(v) &\text{ for } a \in (\R^*)^d, \\
& f \text{ smooth }\}.
\end{align*} 
%where $G = \SL_2(\R)^d \subset \GL_{2d}(\R)$ acts on $\R^{2d}$ in the usual way, and hence by right-translation on functions $f \from (\R^2- \{ 0\})^d \to \C$.
\end{lem}

Identify $\R^2 \isom \R + i \R = \C$, so that we have complex-valued coordinate functions $w_1, \ldots, w_d$ on $(\R^2)^d \isom \C^d$. Let $e_0 = 1/|w_1 \cdots w_d|^{2s}$, and for $J = (j_1,\ldots,j_d) \in (2\Z)^d$, define \[ e_J = \left(\left(\frac{\bar{w_1}}{w_1}\right)^{j_1/2} \cdots \left(\frac{\bar{w_d}}{w_d}\right)^{j_d/2} \right)\cdot e_0. \] These vectors form a basis for the $K$-finite vectors: \[ I(s)^{\Kfin} = \oplus_{J \in (2\Z)^d} \C e_J. \] 

As a representation of $\g_{\C} = \sl_{2,\C}^d$, we have operators $L_i,\ R_i \in \g_{\C}$, $i = 1,\ldots,d$, which act on $I(s)$ by the differential operators \[  L_i = -w_i\frac{\del}{\del \bar{w_i}},\ R_i = -\bar{w_i}\frac{\del}{\del w_i}. \] For example, \[ L_1(e_{(2,0,\ldots,0)}) = (s-1) e_0, \] while 
\[ L_2(e_{(2,0,\ldots,0)}) = s e_{(2,2,0,\ldots,0)}. \]

\subsection{Eisenstein series and analytic continuation}
For any $v \in (\R^2-\{0\})^d$, $g \in \SL_2(\R)^d = G$, we have $v \cdot g \in (\R^2-\{0\})^d$. Thus for any $f \in I(s)$, we can evaluate $f$ on $v \cdot g$. This gives an inclusion of $G$-representations
\[ I(s) \to C^{\infty}(G),\ f \mapsto f(v \cdot g). \] 

Eisenstein series are created by averaging this map over certain collections of vectors $v$ so that its image is contained in $C^{\infty}(\Gamma \backslash G)$\footnote{By definition, $C^{\infty}(\Gamma \backslash G) := C^{\infty}(G)^{\Gamma}$.}. For $\Re(s) > 1$, we have an map \[ \Eis_s \from I(s) \to C^{\infty}(\Gamma \backslash G) \] of $G$-representations, defined as follows.
Consider the lattice \[ \Lambda_0 = \O_F^2 \subset \O_F \tensor \R \subset (\R^2)^d. \] We then define \[ \Eis_s(f) := \left(g \mapsto \sum_{\substack{v \in \Lambda_0/\O_F^*,\\ v \text{ primitive}}} f(v \cdot g)\right). \] For $\Re(s) > 1$, this summation is abolutely convergent.

To relate these to classical formulas for Eisenstein series, we use the Iwasawa decomposition $G = NAK$ to give a section $s \from G/K \isom NA \to G$ of the projection $p \from G \to G/K$. Explicitly, for $(z_i) \in \H^d \isom G/K$, we have
\[ s((z_i)) = \left(\begin{pmatrix} y_i^{1/2} & x_i y_i^{-1/2} \\ 0 & y_i^{-1/2} \\ \end{pmatrix}\right) \in \SL_2(\R)^d. \] We view $f \in I(s)$ as a function on $\C^d$, so in particular in can be evaluated on $\H^d$. We then have
\begin{lem}\label{lem:classical-eis}
$s^*(\Eis_s(f)) = N(y)^s \sum_{\substack{(c,d) \in (\O)^2/\O^*,\\ (c,d) = 1}} f(cz + d)$.
\end{lem}

\begin{eg}
For $s = 2k$, $k \in \N$, $k > 1$, we have
\[ s^*(\Eis_{2k}(e_{(2k,\ldots,2k)})) = N(y)^{2k} \sum_{\substack{(c,d) \in (\O)^2/\O^*,\\ (c,d) = 1}} \frac{1}{N(cz + d)^{2k}}, \]
recovering (up to a normalization) the classical holomorphic Eisenstein series of weight $2k$.
\end{eg}

%\todo{Why mod $\O_F^*$ and not totally positive units, say? Compute $G(\Z)/P(\Z)$}

We recall the meromorphic continuation of Eisenstein series in the $s$-variable. 
\begin{thm}\label{thm:residue-eis}
\

\begin{enumerate}
\item After possibly subtracting a term of the form $\frac{A}{s-1} + \frac{B}{s}$, for each fixed $g \in G$ the function $\Eis_s(e_J)(g)$ analytically continues to $\C$. This analytic continuation is a smooth function on $\C \times G$.
\item For $0 \neq J\in (2\Z)^d$, $g \in G$, $\Eis_s(e_J)(g)$ has holomorphic continuation in $s$.
\item For all $g \in G$, the meromorphic function $\Eis_s(e_0)(g)$ has a simple pole at $s = 1$, with
\[ \Res_{s=1}(\Eis_s(e_0)(g)) = \Res_{s=1} \frac{\xi_F(2s-1)}{\xi_F(2s)}. \]
\end{enumerate}
\end{thm}

Here $\xi_F(s)$ is the completed zeta function of the field $F$, $\xi_F(s) = |\Delta_F|^{s/2} (\pi^{-s/2} \Gamma(s/2))^d \zeta_F(s),$ satisfying the functional equation $\xi_F(s) = \xi_F(1-s)$.

\begin{proof}
%1) It is easy to reduce the meromorphic continuation of $\Eis_s(e_J)$ to that of $\Eis_s(e_0)$ by applying opeators $L, R$ to $\Eis_s(e_0)$. The meromorphic continuation is well-known \todo{Due to Hecke?}.

This is well-known, but for lack of a good reference we sketch the proof. We focus on the case of $\Eis(e_0)$ - the proof for other vectors $e_J$ is similar, but easier.% (it does not follow from the Langland's constant term formula for split reductive groups, as $\Res_{F/\Q} \GL_2$ is not a split reductive group). 

First, we need to know the constant term of the Eisenstein series. The constant term of $\Eis_s(e_0)$ along the standard parabolic $\begin{pmatrix}* & * \\ 0 & * \end{pmatrix} \subset \SL_2(\O)$ is a smooth function $f$ on $(\O^* \tensor \R)/\O^* \isom \R^d/\O^*$ such that $(y_i \frac{\del}{\del y_i})^2(f) = s(s-1)f$ for all $i$. As $f$ must also be invariant under $\O^*$, we find that it is of the form $A(s) N(y)^s + B(s) N(y)^{1-s}$ for functions $A(s), B(s)$. A direct computation of $\lim_{y \to \infty} \frac{1}{N(y)^s}\Eis_s(e_0)$ shows that $A(s) = 1$, while the computation $B(s) = \frac{\xi_F(2s-1)}{\xi_F(2s)}$ is more involved (\cite{Sorensen}).

Then, one must prove that the non-constant terms of the Fourier expansion of $\Eis_s(e_0)$ are holomorphic in $s$. For this, see (\cite{Freitag}, Ch. 3, Prop. 4.6). Therefore \[ \Eis_s(e_0) = N(y)^s +  \frac{\xi_F(2s-1)}{\xi_F(2s)}N(y)^{1-s} + (\text{holom. in } s), \]
and so $\lim_{s \to 1}\Eis_s(e_0) =  \lim_{s \to 1}(s-1)\frac{\xi_F(2s-1)}{\xi_F(2s)}$.

Note that, as (\cite{Freitag}), (\cite{Sorensen}) use classical formulas for Eisenstein series, we use Lemma \ref{lem:classical-eis} to apply their results.
\end{proof}

\begin{comment} 
 While well-known, for lack of a reference this will be verified in the course of proving the following theorem. %One reference for this is (\cite{Freitag}, Ch. 3, Prop.  4.6), which proves that, except for the constant term, all terms in the Fourier expansion of the Eisenstein series have an analytic continuation to $\C$. The constant coefficient is computed in (\cite{Sorensen}), and has the required meromorphic continuation by standard facts about zeta functions. As (\cite{Freitag}), (\cite{Sorensen}) use classical formulas for Eisenstein series, we use Lemma \ref{lem:classical-eis} to apply their results.

 The function $\Eis_s(e_0)$ has meromorphic continuation in $s$. , and the residue is a constant $\in C^{\infty}(\Gamma \backslash G)$. This constant has been computed:
\begin{thm}\label{thm:residue-eis}
%\begin{enumerate}
%\item $\Eis_s(e_J)$ has meromorphic continuation in $s$.
%\item $\Eis_s(e_0)$ has a simple pole at $s = 1$, and 
%\end{enumerate}
\end{thm}
\end{comment}

 Using these continuations, we have a partially defined map
\[ \Eis_1 \from I(1)^{\Kfin} \dasharrow C^{\infty}(\Gamma \backslash G),\ \Eis_1(e_J) := \lim_{s \to 1}  \Eis_s(e_J)\ \text{ for $J \neq 0$. } \]
This map is defined on $\sum_{J \neq 0} \C e_J \subset I(1)^{\Kfin}$. For $\Re(s) > 1$, the map $\Eis_s \from I(s)^{\Kfin} \to C^{\infty}(\Gamma \backslash G)$ is a map of $\g_{\C}$-representations. For $s = 1$, $\Eis_1$ is almost a map of $\g_{\C}$-representations, in the following sense:
\begin{lem}\label{lem:lowering-residue}
\hfill

\begin{enumerate}
\item Let $J \in (2\Z)^d - \{ 0 \}$, $X \in \g_{\C}$, such $X.e_J \in \sum_{J' \neq 0} \C e_{J'} \subset I(1)$. Then both $\Eis_1(e_J)$, $\Eis_1(X.e_J)$ are defined, and \[ X.\Eis_1(e_J) = \Eis_1(X.e_J). \] 
\item Let $J \in (2\Z)^d - \{ 0 \}$, $X \in \g_{\C}$, such that $X.e_J \in \C e_0 \subset I(1)$. Then \[ X(\Eis_1(e_{J})) = \Res_{s=1}(\Eis_s(e_0)) \frac{X.e_J}{(s-1)e_0} \in \C. \] 
\end{enumerate}
\end{lem}
\begin{proof}
(1) Both $\Eis_s(e_J)$, $\Eis_s(X.e_J)$ are holomorphic at $s = 1$. For $\Re(s) > 1$, the summation $\Eis_s(e_J)$ absolutely converges, and so we can commute derivatives with summation. This gives $X.\Eis_s(e_J) = \Eis_s(X.e_J)$. Taking $s \to 1$, we find $\lim_{s \to 1} X.\Eis_s(e_J)(g) = \Eis_1(X.e_J)$  (as a point-wise limit of functions on $G$.)

We need to know that, as $\Eis_s(e_J)$ is holomorphic at $s = 1$, then \[ \lim_{s \to 1} X.\Eis_s(e_J) = X.\Eis_1(e_J). \] This follows from the fact that the function $\Eis_s(e_J)$ is a smooth function on $\C \times G$. 

(2) For $\Re(s) > 1$, $X(\Eis_s(e_J)) = C(s-1)\Eis_s(e_0)$, for $C = \frac{X.e_J}{(s-1)e_0} \in \C$. Taking limits, we obtain $\lim_{s \to 1} X(\Eis_s(e_J)) = C\Res_{s=1}\Eis_s(e_0)$. The left side equals $X(\Eis_1(e_J))$ for the same reason as in (1).

%
%\todo{We need to show $\lim_{s \to 1} X.\Eis_s(e_J) = X.\Eis_1(e_J)$. This probably uses the integral representation and the functional equation - after subtracting off the simple pole, want a jointly $C^1$ function on $s$ and $g$. Does it help to prove it adelically? We still have an issue in choosing the correct test vector at infinity, as it will not be Schwartz.}
\end{proof}

Write $F_J := \Eis_1(e_J), F_0 := 1$. We define \[ C^{\infty}_{\Eis} := \sum_{J \in (2\Z)^d} \C F_{J} \subset C^{\infty}(\Gamma \backslash G). \]

\subsection{Differential forms and $(\g,K)$-cohomology}
\

See \cite{Borel-Wallach} for generalities on $(\g, K)$-modules and $(\g,K)$-cohomology.

Let $M$ be a $(\g, K)$-module. There is a complex $\Hom_K(\wedge^*(\g/\kk) \tensor \C, M)$, with differential defined as follows. For $f \in M,\ \phi \in \wedge^*(\g/\kk)^{\dual}$, we have
\begin{align*}
f \tensor \phi & \in \Hom_K(\wedge^*(\g/\kk) \tensor \C, M) \subset M \tensor \wedge^*(\g/\kk)^{\dual}, \\
 d(f \tensor \phi) & := \sum_i L_i(f) \tensor (L_i^{\dual} \wedge \phi) + R_i(f) \tensor (R_i^{\dual} \wedge \phi). 
\end{align*} In this formula, $L_i^{\dual}$, $R_i^{\dual} \in (\g/\kk)^{\dual} \tensor \C$, are the dual basis of $L_i$, $R_i \in \g/\kk \tensor \C$. We view $L_i$, $R_i$ as operators on the $\g_{\C}$-module $M$. 

\begin{rmk}
\

The complex $\Hom_K(\wedge^*(\g/\kk) \tensor \C, M)$ computes the $(\g, K)$-cohomology group $\Ext^*_{(\g,K)}(\C, M)$, although we do not need this.\end{rmk}

\begin{lem}\label{lem:gk-into-forms}
$\wedge^*(\g/\kk)^{\dual} \tensor \C$ is isomorphic to the space of left $G$-invariant differential forms on $G$ which are trivial on the vectors in the fiber of the map of tangent bundles $T(G) \to T(G/K).$
\end{lem}

We will frequently identify $\wedge^*(\g/\kk)^* \tensor \C$ with a subspace of $\Omega^*(G)$ ($C^{\infty}$ complex-valued differential forms on $G$) via this lemma. This lemma implies the following basic result in the theory of $(\g,K)$-cohomology:
\begin{prop}\label{prop:gK-complex}
\

\begin{enumerate}
\item There is an isomorphism of complexes \[ i \from \Hom_K(\wedge^*(\g/\kk) \tensor \C, C^{\infty}(G)) \isom \Omega^*(G/K), \]
given by taking the product of a form $\wedge^*(\g/\kk)^{\dual} \tensor \C \subset \Omega^*(G)$ with a function on $G$, noting that the condition the the product is $K$-invariant ensures that the form descends to a form on $G/K$. 
\item Restricting to the $\Gamma$-invariants in the above isomorphism gives an isomorphism \[ i \from \Hom_K(\wedge^*(\g/\kk) \tensor \C, C^{\infty}(\Gamma \backslash G)) \isom \Omega^*(\Gamma \backslash G/K). \]
\end{enumerate}
\end{prop}

\begin{comment}
We recall the definition of the differential on the domain of $i$: 
\begin{align*}
f \tensor \phi & \in \Hom_K(\wedge^*(\g/\kk) \tensor \C, C^{\infty}(G)) \subset C^{\infty}(G) \tensor \wedge^*(\g/\kk)^{\dual}, \\
 d(f \tensor \phi) & := \sum_i L_i(f) \tensor (L_i^{\dual} \wedge \phi) + R_i(f) \tensor (R_i^{\dual} \wedge \phi). 
\end{align*} In this formula, we consider $L_i^{\dual}$, $R_i^{\dual}$ as elements of $(\g/\kk)^{\dual} \tensor \C$, and $L_i$, $R_i$ as operators on the $\g_{\C}$-module $C^{\infty}(G)$.
\end{comment}

Lemma \ref{lem:lowering-residue} implies
\begin{cor}\label{cor:gK-sub}
$C^{\infty}_{\Eis}$ is a $(\g, K)$-submodule of $C^{\infty}(\Gamma \backslash G)$.
\end{cor}

Restricting to the $(\g, K)$-submodule $C^{\infty}_{\Eis} \subset C^{\infty}(\Gamma \backslash G)$, we obtain a subcomplex \[ \Omega^*(Y)_{\Eis} := i(\Hom_K(\wedge^*(\g/\kk) \tensor \C, C^{\infty}_{\Eis})) \subset \Omega^*(\Gamma \backslash G/K). \] Our goal in this section is to relate $\Omega^*(Y)_{\Eis}$ to classical Eisenstein series. 

\begin{comment}
of the projection $p \from G \to G/K$. Explicitly, for $(z_i) \in \H^d \isom G/K$, we have
\[ s((z_i)) = \left(\begin{pmatrix} y_i^{1/2} & x_i y_i^{-1/2} \\ 0 & y_i^{-1/2} \\ \end{pmatrix}\right) \in \SL_2(\R)^d. \] Then the map $i$ is equal to
\[ \Hom_K(\wedge^*(\g/\kk) \tensor \C, C^{\infty}(G)) \into \Omega^*(G) \nmto{s^*} \Omega^*(G/K). \]
\end{comment}

For $j \in 1,\ldots, d$, define $\eta_{2,j} = \frac{1}{2i}\frac{dz_j}{y_j}$, $\eta_{0,j} = 1$, $\eta_{-2,j} = \frac{1}{-2i} \frac{d\bar{z_j}}{y_j}$. 
\begin{lem}\label{left-NA-invariance}
\

\begin{enumerate}
\item For $j = 1,\ldots,d$, \[ s^*(R_j^{\dual}) = \eta_{2,j},\ s^*(L_j^{\dual}) = \eta_{-2,j}, \]
where we consider $L_j^{\dual},\ R_j^{\dual},$ as elements of $\Omega^*(G)$ via Lemma \ref{lem:gk-into-forms}.
\item The map $\wedge^* (\g/\kk)^{\dual} \tensor \C \subset \Omega^*(G) \nmto{s^*} \Omega^*(G/K)$ is an isomorphism onto the left $NA$-invariant forms on $G/K$.
\end{enumerate}
\end{lem}
%\begin{rmk}
%Note that $p^*(\eta_{2,j})$ is right $K$-invariant, while $R_j^{\dual}$ is not.
%\end{rmk}

\begin{proof}
It is easy to see that (2) follows from (1). We prove (1).

Let $g = nak,\ k = (k_j) \in \SO(2)^d$. For each $j$, we have a homomorphism $\phi_j \from \SO(2) \to \C^*$ sending $\begin{pmatrix} \cos(\theta) & \sin(\theta) \\ - \sin(\theta) & \cos(\theta) \end{pmatrix}$ to $e^{2 i \theta}$. We have forms $R_j^{\dual}$ and $\phi_j(k_j)p^*(\eta_j)$ on $G = NAK$. These are left $NA$-invariant forms on $G$ with the same transformation under the right action of $K$, and hence are equal up to a non-zero constant $C$. As $s^*(\phi_j(k_j)) = 1$, we obtain $s^*(R_j^{\dual}) = C s^*(\phi(k_j)p^*(\eta_j)) = C \eta_j$. Similarly, $s^*(L_j^{\dual}) = C' \bar{\eta_j}$.

To compute the constants $C$, $C'$, we must do a little computation. It is clear that for this computation, we may assume that $G = \SL_2(\R)$.

Consider $e_0 \in I(1)$. We obtain a function $f(g) := e_0((0,1)\cdot g)$ on $G$. Then $f(s(z)) = (\frac{1}{|w|^2})((0,1)\cdot g) = \Im(z).$ We compute $s^*(R(f)) = R(e_0((0,1) \cdot g)) = e_2((0,1) \cdot g) = (\frac{1}{w^2})((0,1)\cdot g) = \Im(z)$, and similarly $s^*(L(f)) = \Im(z)$. Therefore
\begin{align*}
d(\Im(z)) &= d(s^*(f)) \\ 
&= s^*(d(f)) \\ 
&= s^*(L(f) L^{\dual} + R(f) R^{\dual}) \\ &= \Im(z) (C \frac{dz}{2i \Im(z)} + C' \frac{d\bar{z}}{-2i \Im(z)}). 
\end{align*} But we can directly compute
\[ d(\Im(z)) =  \Im(z) (\frac{dz}{2i \Im(z)} + \frac{d\bar{z}}{-2i \Im(z)}), \]
so we find $C = C' = 1$.
\end{proof}

We introduce some notation. For $J \in (2 \Z)^d$, define $\eta_J = \eta_{j_1,1} \wedge \ldots \wedge \eta_{j_d,d}$. We define, for $i = 1,\ldots, d$,  $\omega'_i = \eta_{2,i} \wedge \bar{\eta_{-2,i}} = \frac{dz_i \wedge d\bar{z_i}}{4y_i^2} = -\frac{i}{2} \frac{dx_i \wedge dy_i}{y_i^2}$, and for $\epsilon \subset [d]$,  $\omega'_{\epsilon} = \wedge_{i \in \epsilon} \omega'_i$. Define $\supp(J) = \{ i \in [d] \mid j_i \neq 0 \}$. We see that, using the identification from Lemma \ref{left-NA-invariance}, that
\[ \wedge^* (\g/\kk)^{\dual} \tensor \C \text{ has basis } \eta_J \wedge \omega'_{\epsilon}, \text{ for } J \in \{-2,0,2\}^d, \epsilon \subset [d] - \supp(J). \] 

\begin{prop}\label{prop:classical-eis}
\
\begin{enumerate}
\item The vector space $\Hom_K(\wedge^*(\g/\kk), C^{\infty}_{\Eis})$ has basis \[ F_J \tensor \eta_J \wedge \omega'_{\epsilon} \text{ for } J \in \{-2,0,2\}^d,\ \epsilon \subset [d] - \supp(J). \]

\item Under the map $i \from \Hom_K(\wedge^*(\g/\kk), C^{\infty}_{\Eis}) \to \Omega^*(Y)$, $F_J \tensor \eta_J \wedge \omega'_{\epsilon}$ goes to the form
\[ s^*(F_J) \eta_J \wedge \omega'_{\epsilon}. \]

\item The functions $E_J := s^*(F_J) \in C^{\infty}(Y)$ are classical weight 2 Eisenstein series. More precisely, for $J \in \{-2,0,2\}^d$, 
\begin{align*} E_J = N(y) \lim_{s \to 0^+} \sum_{\substack{(c,d) \in (\O)^2/\O^*,\\ (c,d) = 1}} &\frac{1}{|N(cz+d)|^{2s}}\prod_{j_i = 2}\frac{1}{(c_i z_i + d_i)^2} \\
 & \cdot \prod_{j_i = 0}\frac{1}{|c_i z_i + d_i|^2} \prod_{j_i = -2}\frac{1}{(c_i \bar{z_i} + d_i)^2} 
\end{align*}
\end{enumerate}
\end{prop}
Here, $y = (y_i) = (\Im(z_i))$, and $N(y) = \prod y_i$.
\begin{proof}
(1) and (2) are easy. (3) follows immediately from Lemma \ref{lem:classical-eis} by taking limits. 

\begin{comment}
We verify (3) in the case $J = (2,\ldots,2)$.

% We write $\tau = (\tau_1,\ldots,\tau_d) \in \H^d$, as opposed to the $z_i$'s used in the statement, to avoid conflict with the $z_i$'s used in our model for the principal series representation $I(s)$.

Associated to $z = (z_1,\ldots,z_d) \in \H^d$, we have a matrix \[ g = (g_i) = \left(\begin{pmatrix} y_i^{1/2} & x_i y_i^{-1/2} \\ 0 & y_i^{-1/2} \\ \end{pmatrix}\right) \in \SL_2(\R)^d. \] Then $\Lambda_0 \cdot g = N(y)^{-1/2}(\O_F + (z_1,\ldots,z_d) \O_F)$, where $N(y) = \prod_{i=1}^d y_i$.

Take the vector $e_J \in I(s)$. This corresponds to the function $e_J = \frac{1}{|N(w)|^{2s-2}N(w)^2}$ on $(\C^*)^d$. For $\Re(s) > 1$, we have
\begin{align*} 
\Eis_s(e_J)(g) &= \sum_{\substack{v \in \Lambda_0 \cdot g/\O^* \\ v \text{ primitive}}} e_J(v) \\
&= \sum_{\substack{(c,d) \in (\O)^2/\O^*,\\ (c,d) = 1}} \frac{N(y)^{s}}{|N(cz+d)|^{2s-2} N(c z + d)^2} \\
%&= N(y)^{s} \sum_{\substack{(c,d) \in (\O)^2/\O^*,\\ (c,d) = 1}} \frac{1}{|N(cz+d)|^{2s-2}}\frac{1}{ N(c z + d)^2} \\
\end{align*}
\end{comment}

\end{proof}

\subsection{Partial complex conjugation}\label{subsec:partial-conjugation}
The group $\pi_0(O(2)^d) = (\Z/2\Z)^d$ acts on the space $Y = \Gamma \backslash G/K$ in the following way. By our hypothesis on the narrow class number of $F$, there are isomorphisms
\[ \Gamma \backslash G \isom (\R^*)^d\GL_2(\O_F)\backslash \GL_2(\R)^d,\ \Gamma \backslash G / K \isom (\R^*)^d\GL_2(\O_F) \backslash \GL_2(\R)^d/SO(2)^d. \] The group $O(2)^d$ acts on $\GL_2(\O_F) \backslash \GL_2(\R)^d$ by right-multiplication, descending to an action of $\pi_0(O(2)^d)$ on $Y$. This agrees with the action of $\pi_0(\GL_2(\R)^d) = \pi_0(O(2)^d)$ defined by Harder (\cite{Harder}, 1.2).

For concreteness, we reformulate this action in terms of the space $\H^d/\Gamma$. To do this, we first recall the following easy facts about units in $\O_F$, which follow from our narrow class number one hypothesis:
\begin{lem}
\

\begin{enumerate}
\item $\O_F^*/(\O_F^*)^2 \isom (\Z/2\Z)^d$ by the map sending a unit to its signs under the real embeddings of $F$.
\item The totally positive units $U \subset \O_F^*$ equals the subgroup of squares $(\O_F^*)^2 \subset \O_F^*$.
\item The totally negative units are of the form $-u$ for $u \in (\O_F^*)^2.$
\end{enumerate}
\end{lem}

 For $i = 1,\ldots, d$, let $\epsilon^{(i)} \in \O_F^*$ be a unit such that $\epsilon^{(i)}_i < 0$, $\epsilon^{(i)}_j > 0$ for $j \neq i$. We have maps \[ c_i \from \H^d \to \H^d,\ c_i(z_1,\ldots, z_d) = (\epsilon^{(i)}_1 z_1, \ldots, \epsilon^{(i)}_i \bar{z_i}, \ldots, \epsilon^{(i)}_d z_d). \] The map $c_i$ descends to a map $c_i \from \H^d/P(\Z) \to \H^d/P(\Z)$, which does not depend on the choice of $\epsilon^{(i)}$ by the above lemma. Moreover, these maps are commuting involutions of the spaces $\H^d/P(\Z)$ and $Y = \H^d/\SL_2(\O)$ (compare \cite{VanDerGeer}, pg. 125). The composition $c := c_1 \circ \cdots \circ c_d$ acts as $(z_1,\ldots, z_d) \mapsto (-\bar{z_1}, \ldots, -\bar{z_d})$ on $\H^d/P(\Z)$ and $Y = \H^d/\SL_2(\O)$. 

\begin{lem}
The action of $\langle c_1, \ldots, c_d \rangle \isom (\Z/2\Z)^d$ on $Y$ agrees with the action of $\pi_0(O(2)^d) \isom (\Z/2\Z)^d$.
\end{lem}
 
 We compute the action of $\pi_0(O(2)^d)$ on $\Omega^*(Y)_{\Eis}$. Define involutions $c_i \from \Z^d \to \Z^d$ acting as identity on all but the $i$-th coordinate, where it acts as multiplication by $-1$. We record the actions of the geometric complex conjugations $c_i^*$ as well as the standard complex conjugation on functions/forms $\bar{(\cdot)}$:%Then $c_i(e_J) = e_{c_i(J)}$. 

\begin{lem}\label{lem:cplx-conj-eis}
\
\begin{enumerate}
\item $c_i^*(E_J) = E_{c_i(J)}$, $\bar{E_J} = E_{-J}$ for $J \in (2\Z)^d$
\item $c_i^*(\omega_i) = - \omega'_i,\ c_i^*(\omega'_j) = \omega'_j$ for $i \neq j$, $\bar{\omega'_i} = - \omega'_i$
\item $c_i^*(\eta_J) = -\eta_{c_i(J)}$, $\bar{\eta_J} = \eta_{-J}$
\end{enumerate}
\end{lem}
\begin{proof}
(1) If $J = 0$, $E_J = 1$, so the claim is obvious. 

There is an action of $\pi_0(O(2)^d) \isom (\Z/2\Z)^d$ on $I(s)$, via our identification of $I(s)$ with certain functions on $(\C^*)^d$, using complex-conjugation on each factor. In terms of the standard basis of $K$-finite vectors, we have $c_i(e_J) = e_{c_i(J)}$.

Thus, if $J \neq 0$, (1) is equivalent to the claim that, for $s = 1$, the partially defined map \[ I(1)^{\Kfin} \nmtodash{\Eis_1} C^{\infty}(\Gamma \backslash G) \nmto{s^*} C^{\infty}(G/K) \] is $\pi_0(O(2)^d)$-equivariant when defined. The proof of this is similar to that of Lemma \ref{lem:lowering-residue}, in that we first prove it for $I(s)^{\Kfin} \nmto{\Eis_s} C^{\infty}(\Gamma \backslash G) \nmto{s^*} C^{\infty}(G/K)$, $\Re(s) > 1$, obtaining the case $s = 1$ by taking limits. 

When $\Re(s) > 1$, this can be checked this directly, using the formulas for $c_i$ acting on $\H^d/\Gamma$ and the classical formulas for Eisenstein series (Lemma \ref{lem:classical-eis}).

\begin{comment}
 A more theoretical proof would involve: 
\begin{enumerate}
\item[(i)] considering $I(s)$ as the $\GL_2(\R)^d$-representation, via the action of $\GL_2(\R)^d$ on $(\R^2-\{ 0 \})^d$,
\item [(ii)] computing the action on $I(s)$ of the complex-conjugations $(\Z/2\Z)^d \subset \GL_2(\R)^d$ (coming from the identification $\R^2 \isom \C$), \\
\item[(iii)] checking that $I(s)^{\Kfin} \tensor |\det|^{-s}  \nmto{\Eis_s} C^{\infty}((\R^*)^d\GL_2(\O_F) \backslash \GL_2(\R)^d)$ is $\GL_2(\R)^d$-equivariant, 
\item [(iv)] checking that \[ C^{\infty}((\R^*)^d\GL_2(\O_F) \backslash \GL_2(\R)^d) \nmto{s^*} C^{\infty}(\GL_2(\R)^d/(\R^*)^dSO(2)^d)\] is $O(2)^d$-equivariant.
\item [(v)] using that $s^*(\Eis_s(e_J))$ is $P'(\Z) := \begin{pmatrix}
\O_F^* & \O_F \\ 0 & \O_F^*
\end{pmatrix}$-invariant to see that it descends to $P'(\Z)\backslash \GL_2(\R)^d/(\R^*)^dSO(2)^d \isom \H^d/P(\Z)$, where the two actions of $(\Z/2\Z)^d$ agree.
\end{enumerate}
\end{comment}

\begin{comment}
This can be checked in two steps: 
\begin{enumerate}
\item[(i)] $C^{\infty}(\Gamma \backslash G) \nmto{s^*} C^{\infty}(\Gamma \backslash G/K)$ is equivariant for the map $O(2)^d \nmto{\det} (\Z/2\Z)^d$
\item[(ii)] $I(1)^{\Kfin} \nmtodash{\Eis_1} C^{\infty}(\GL_2(\O_F) \backslash \GL_2(\R)^d)$ is $O(2)^d$-equivariant where defined.
\end{enumerate}
\end{comment}

% E_J should be considered as a $\GL_2(\O_F)$-invariant function on $\GL_2(\R)^d/SO(2)^d$.

(2) and (3) follow from the formulas defining $\omega'_i$ and $\eta_J$ and our explicit description of the operators $c_i^*$ on $\H^d$.
% d(\epsilon \bar{z_i})/(- \epsilon y_i) = - d\bar{z_i}/y_i.
\end{proof}

\subsection{Exterior derivatives of Eisenstein series}
We now compute all differentials on the subcomplex $\Omega^*(Y)_{\Eis} \subset \Omega^*(Y)$. 

Define $|J| := \sum |j_i|$. Define $v_i := (0,\ldots,2,\ldots,0)$, zero except in the $i$-th coordinate.
\begin{prop}\label{prop:deriv-via-gK}
\hfill
\begin{enumerate}
\item If $|J| > 2$, then
\begin{align*} d(E_J \eta_J \wedge \omega'_{\epsilon}) &= \left( \sum_{i \text{ s.t. } j_i = 0, \epsilon_i = 0} E_{J + v_i} \eta_{v_i} + E_{J - v_i} \eta_{-v_i} \right) \wedge \eta_{J} \wedge \omega'_{\epsilon} \\
&= \left( \sum_{i \text{ s.t. } j_i = 0, \epsilon_i = 0} (1+c_i)^*(E_{J + v_i} \eta_{v_i}) \right) \wedge \eta_{J} \wedge \omega'_{\epsilon}. 
\end{align*}
\item If $|J| = 2$, then $|j_i| = 2$ for a unique $i$. We have
\[ d(E_J \eta_J \wedge \omega'_{\epsilon}) = -(j_i/2) \Res_{s=1}(\Eis_s(e_0)) \omega_i \wedge \omega'_{\epsilon} + \text{ the RHS of (1) }. \]

\item If $|J| = 0$, then $E_J \eta_J \wedge \omega'_{\epsilon} = \omega'_{\epsilon}$, and \[ d(\omega'_{\epsilon}) = 0. \]
\end{enumerate}
\end{prop}
\begin{proof}
% d(E_J \eta_J \wedge \omega_{\epsilon}) = \left(\sum L_i(E_J) \wedge \eta_{v_i} \wedge \eta_J + R_i(E_J) \eta_{-v_i}\right) \eta_J \wedge \omega_{\epsilon}.
% L_1(E_{(0,2,2)})

This follows easily from the definition of the differential on $\Hom_K(\wedge^*(\g/\kk), C^{\infty}_{\Eis})$, Lemma \ref{left-NA-invariance}, the structure of $C^{\infty}_{\Eis}$ as a $(\g, K)$-module (Lemma \ref{lem:lowering-residue}, Corollary \ref{cor:gK-sub}), and the fact that $i \from \Hom_K(\wedge^*(\g/\kk), C^{\infty}_{\Eis}) \to \Omega^*(Y)$ is a map of complexes (Proposition \ref{prop:gK-complex}). Note that in the second equality of (1), we simply rewrite the first line using Lemma \ref{lem:cplx-conj-eis}.
\end{proof}

%\begin{rmk}
%If it were not for issues of absolute convergence, (1) would follow from the identity $dy = \frac{dz}{2i} + \frac{d\bar{z}}{2i}$. The fact that such issues arise is what allows for (2).
%\end{rmk}

We introduce some notation which will be convenient in the following sections. For $J \subset [d]$, define $\phi(J) \subset \{0,2\}^d$ by $\phi(J)(j) = \begin{cases} 2 & \text{if $j \in J$ }\\
0 & \text{otherwise } \end{cases}$. 
Define $E'_J := E_{\phi(J)} \eta_{\phi(J)} \wedge \omega'_{[d] - J}.$

We record the following consequence:
\begin{cor}\label{cor:cplx-conj-on-defective-eis}
For $1 \leq i < j \leq d$, let $I = \{i,j\}$, and consider the form $E'_{I}$. It is a closed $(2d-2)$-form, satisfying
\[ (E'_I - (-1)^{d}\bar{E'_I}) = \frac{\xi_F(2s-1)}{\xi_F(2s)} \cdot (\omega'_{[d] - i} - \omega'_{[d]-j}) + d\eta \]  
for some $2d-3$-form $\eta$.
\end{cor}
\begin{proof}
We apply Proposition \ref{prop:deriv-via-gK} (2) with $J = v_j, \epsilon = [d] - \{i,j\}$ to obtain (up to exact forms)
\[ \left( E_{v_i+v_j} \eta_{v_i} + E_{- v_i + v_j} \eta_{-v_i} \right) \wedge \eta_{v_j} \wedge \omega'_{[d]-\{i,j\}} = \frac{\xi_F(2s-1)}{\xi_F(2s)} \omega'_{[d]-i}.\]

Similarly, for $J = -v_i, \epsilon = [d] - \{i,j\}$, we obtain
\[ \left( E_{-v_i+v_j} \eta_{v_j} + E_{-v_i - v_j} \eta_{-v_j} \right) \wedge \eta_{-v_i} \wedge \omega'_{[d]-\{i,j\} } = -\frac{\xi_F(2s-1)}{\xi_F(2s)} \omega'_{[d]-j}.\]

Adding these together, we obtain 
\[ E_{v_i+v_j} \eta_{v_i} \wedge \eta_{v_j} \wedge  \omega'_{[d]-\{i,j\}} - E_{-v_i - v_j}  \eta_{-v_i} \wedge \eta_{-v_j} \wedge \omega'_{[d]-\{i,j\} } = \frac{\xi_F(2s-1)}{\xi_F(2s)} ( \omega'_{[d]-i} -  \omega'_{[d]-j}). \]

Via Lemma \ref{lem:cplx-conj-eis}, this gives \[  (E'_I - (-1)^{d}\bar{E'_I}) = (1 - (-1)^{d}c^*)(E'_I) =  \frac{\xi_F(2s-1)}{\xi_F(2s)} ( \omega'_{[d]-i} -  \omega'_{[d]-j}). \]
\end{proof}

\section{Eisenstein periods via differential forms}\label{sec:weak-eis-periods}
In this section, we give a proof of a weak form of Theorem \ref{thm:eis-periods-1}. Let $Y = \H^d/\SL_2(\O)$. We use the notation for Eisenstein series of Section \ref{subsec:eis-series}. With this notation, Corollary \ref{cor:cplx-conj-on-defective-eis} becomes: % (noting a change in sign)
\begin{cor}
\[ (E_{[d] - \{i,j\}} - (-1)^{d}\bar{E_{[d] - \{i,j\}}}) = \frac{\xi_F(2s-1)}{\xi_F(2s)} \cdot (\omega_j^* - \omega_i^*) + d\eta \]
\end{cor}

%Note that $c^*(E'_I) = \bar{E'_I}$ as differential forms (this is similar to Lemma \ref{lem:cplx-conj-eis} (1)).

Recall from Section \ref{subsec:Eis-MHS} the map
\[ \Eis \from \O^* \to H^{2d-2}(Y, \C). \]
 
\begin{lem} For $u \in \O^*$,
\[ \Eis(u) = \pm(2 \pi i)^d \left(\frac{1}{\vol(\del X)} \sum_i 2\log(|\sigma_i(u)|) \sum_j E_{[d] - \{i,j\}}  \right) \in H^{2d-2}(Y, \C). \]
\end{lem}
\begin{proof}
Compute the image of the right side under $\frac{1}{(2\pi i)^d}\PD \circ i_{\infty}^* \from H^{2d-2}(Y, \C(d)) \to H_1(\H^d/P(\Z), \C)$, and see that it is $\pm u$. The sign ambiguity arises because we did not specify an orientation on $(\H^d/P(\Z))^{N=1}$.
\end{proof}

Here, 
\begin{align*}\vol(\del X) := & \int_{(\H^d/P(\Z))^{N=1}} dx_{[d]} \wedge \frac{dy_{[d-1]}}{y_{[d-1]}} \\
\vol(X) := & \int_{\H^d/\SL_2(\O)} \omega_1 \wedge \ldots \wedge \omega_d \in \Q(d).
\end{align*}

% Old ones:
%vol(X) = 2\zeta_F(-1) \cdot (2 \pi i)^{d}
%vol(\del X) = 2^{d-1} |Reg_K| |\Delta_K|^{1/2}$. 
% I have now modified this by some power of i

% For $d = 1$, E_2 has constant term $1 - 3/(\pi y)$

\begin{lem}\label{lem:residue-volume}
$\Res_{s=1}\frac{\xi_F(2s-1)}{\xi_F(2s)} = \frac{|\vol(\del X)|}{|\vol(X)|}$
\end{lem}
\begin{proof}
A Stokes' Theorem argument for the Borel-Serre compactification and the Eisenstein series $E'_{1}$, a differential form of degree $2d-1$, proves this directly. However, we instead note that all of the terms can be related to the arithmetic of the field $F$:

\begin{align*}
%\Res_{s=1}(\Eis_s(e_0)) &= \Res_{s=1}\left(\frac{\xi_F(2s-1)}{\xi_F(2s)}\right) \\
|\vol(X)| &= 2|\zeta_F(-1)| (2 \pi)^d \\
|\vol(\del X)| &= 2^{d-1} |\Reg_F| |\Delta_F|^{1/2}, \\  
\end{align*}

% \xi_F(s) = |\Delta_F|^{s/2}(\pi^{-s/2} \Gamma(1/2 s))^d \zeta_F(s)

% Class number:
% $\lim _{s\to 1}(s-1)\zeta _F(s)={\frac {2^{r_{1}}\cdot (2\pi )^{r_{2}}\cdot \operatorname {Reg} _F\cdot h_F}{w_F\cdot {\sqrt {|D_F|}}}}$
% For totally real, class number 1, this is $2^d \frac{\Reg_F}{2 \sqrt{|D_F|}}$
% \Res_{s=1} \xi_F(2s-1) = 2^d \Reg_F 

% \xi_F(-1) = |\Delta_F|^{-1/2}(\pi^{1/2} \Gamma(-1/2))^d \zeta_F(-1) = (-2\pi)^d |\Delta_F|^{-1/2} \zeta_F(-1)

% \Res_{s=1} \xi_F(2s-1) = (\pi^{-1/2} \Gamma(1/2))^d \Res_{s=1} \zeta_F(2s-1) = \Res_{s=1} \zeta_F(2s-1) = 1/2 \Res_{s=1} \zeta_F(s-1) = 2^{d-2} R_F/\sqrt{|D_F|}

%\Res_{s=1}\left(\frac{\xi_F(2s-1)}{\xi_F(2s)}\right) = \xi_F(2)^{-1} \Res_{s=1} \xi_F(2s-1) = (-\pi)^{-d} \frac{\Reg_F\sqrt{|D_F|}}{ \zeta_F(-1)}

% The infinity factor of $\frac{\xi(2s-1)}{\xi(2s)}$ is
% \frac{\pi^{-(s-1/2)}\Gamma(s-1/2)}{ \pi^{-s} \Gamma(s)}

The first equality is Theorem \ref{thm:mumford-proportionality}, while the second is an elementary computation, relating the regulator and discriminant to the volumes of $T(\R)^{N=1}/T(\Z)$ and $N(\R)/N(\Z)$. We have \[ \frac{\vol(\del X)}{|\vol(X)|} = \frac{2^{d-1} |\Reg_F| |\Delta_F|^{1/2}}{2|\zeta_F(-1)| (2 \pi)^d} = \frac{|\Reg_F||\Delta_F|^{1/2}}{4 |\zeta_F(-1)| \pi^d} . \]

%The first equality is Theorem \ref{thm:residue-eis}, the second is Theorem \ref{thm:mumford-proportionality}, and the third is an elementary computation, relating the regulator and discriminant to the volumes of $T(\R)^{N=1}/T(\Z)$ and $N(\R)/N(\Z)$. We have \[ \frac{\vol(\del X)}{|\vol(X)|} = \frac{2^{d-1} |\Reg_F| |\Delta_F|^{1/2}}{2|\zeta_F(-1)| (2 \pi)^d} = \frac{|\Reg_F||\Delta_F|^{1/2}}{4 |\zeta_F(-1)| \pi^d} . \]

\begin{rmk}When $d = 1$, this reduces to $\frac{3}{\pi}$.\end{rmk}

 An easy calculation with the analytic class number formula proves $\Res_{s=1} \xi_F(2s-1) = 2^{d-2} \Reg_F$, while the functional equation for $\xi_F(s)$ proves that $\xi_F(2) = \xi_F(-1) = (-2\pi)^d |\Delta_F|^{-1/2} \zeta_F(-1)$. Thus  $\Res_{s=1}\left(\frac{\xi_F(2s-1)}{\xi_F(2s)}\right) = \frac{\Reg_F |\Delta_F|^{1/2}}{ 4\zeta_F(-1)(-\pi)^d}$, proving the lemma. 
\end{proof}

Putting these together, and using $\vol(X) = 2|\zeta_F(-1)| (2 \pi i)^d$ (\ref{thm:mumford-proportionality}), we have:
\begin{thm}\label{thm:eis-periods-weak}
\[ \frac{1}{2}(\Eis(u)/(2 \pi i)^d - \bar{\Eis(u)/(2 \pi i)^d}) = \pm \frac{1}{2|\zeta_F(-1)| (2 \pi i)^d}\sum_{i=1}^d \log(|\sigma_i(u)|)\omega_i^*. \]
\end{thm}

This implies the following weak version of Theorem \ref{thm:eis-periods-1}, using $\R$ as opposed to $\Q$ coefficients:
\begin{cor} For all $u \in \O_F^* \tensor \Q$, we have
\[ \Eis(u) = \pm\frac{1}{2|\zeta_F(-1)|} \sum_{i = 1}^d \log(|\sigma_i(u)|) \omega_i^* \mod H^{2d-2}(Y, \R(d)). \]
In particular, for $1 \neq u \in \O_F^* \tensor \Q$, $\Eis(u) \notin H^{2d-2}(Y, \R(d)).$
\end{cor}

\section{Harder's Eisenstein section}\label{sec:harder-section}
In this section, we study the relationship of our results with Harder's Eisenstein cohomology \cite{Harder}. Let $Y = \H^d/\SL_2(\O_F)$. %The rationality of the subspace $H^*(Y)_{\univ} \subset H^*(Y)_{\Eis}$ implies that there is \emph{some} complement of $H^*(Y, \Q)_{\univ} \subset H^*(Y,\Q)_{\Eis}$ defined over $\Q$. 
Harder showed that there is a unique $\pi_0(\GL_2(F_{\infty}))$-equivariant splitting of 
$H^{*}(Y, \Q)_{\Eis} \surj H^*(Y, \Q)_{\del},$ which we will denote \[ \Eis' \from H^{d+k}(Y, \Q)_{\del} \to H^{d+k}(Y, \Q)_{\Eis}. \] We give a proof of this, which should be essentially the same as Harder's, in Section \ref{subsec:harder-conjugation}. We then prove:
\begin{thm}\label{thm:harder-filtration}
\

\begin{enumerate}
\item For $k = d-1$, $\Eis'$ is not compatible with the Hodge filtration.
\item For $0 \leq k < d-1$, $\Eis'$ is compatible with the Hodge filtration.
\end{enumerate}
\end{thm}
\begin{proof}
(1) follows immediately from \ref{cor:rational-sections}. We will prove (2) in the remainder of this section. 
\end{proof}

(2) is equivalent to:
\begin{cor}
For $0 \leq k < d-1$, the maps $\Eis \from H^{d+k}(Y, \Q)_{\del} \to H^{d+k}(Y, \Q)_{\Eis}$, characterized by their compatibility with the Hodge filtration, are $\pi_0(\GL_2(F_{\infty}))$-equivariant.
\end{cor}

\subsection{Partial complex conjugation}\label{subsec:harder-conjugation}
In this section, we reproduce some results of Harder \cite{Harder}. Recall that in Section \ref{subsec:partial-conjugation} we defined an action of $\pi_0(\GL_2(F_{\infty})) = (\Z/2\Z)^d$ on $H^*(\H^d/\SL_2(\O_F))$.%, although we will not take that perspective.

\begin{lem}\label{lem:conj-on-boundary}
For $0 \leq k \leq d-1$, $c_j^* = -1$ on $H^{d+k}(\H^d/P(\Z), \Q)$.
\end{lem}
\begin{proof}
Recall that, in degrees $0 \leq i \leq d-1$, the map $s_* \from H_i(\R^{d-1}/T(\Z),\Q) \to H_i((\H^d/P(\Z))^{N = 1},\Q)$ is an isomorphism (Lemma \ref{lem:homology-from-base}). The involutions $c_j$ do not restrict to the identity on $\R^{d}/T(\Z)$, but they act trivially on homology, as the action of $c_i$ on $\R^d$ commutes with the action of $T(\Z)$. Thus $(c_j)_* = 1$ on $H_i((\H^d/P(\Z))^{N = 1},\Q)$. 

The map $c_j$ does not preserve the orientation of $(\H^d/P(\Z))^{N = 1}$. Thus the Poincare duality $H_i(\H^d/P(\Z)) \isom H^{2d-i-1}(\H^d/P(\Z))$ identifies $(c_j)_*$ with $-(c_j)^*$. Hence $c_j^*$ acts by $-1$ on $H^{2d-i-1}(\H^d/P(\Z))$.
\end{proof}

\begin{lem}
$H^*(Y)_{\univ}^{c_1^* = -1,\ldots,c_d^* = -1} = 0$.
\end{lem}
\begin{proof}
The only product of the $\omega_i$'s which could have these eigenvalues is $\omega_1 \wedge \ldots \wedge \omega_d$, but this product is zero.
\end{proof}

We consider the endomorphisms \[ \frac{1}{2}(1 + c_j^*) \from H^{d+k}(Y) \to H^{d+k}(Y). \] They commute with the Hecke operators (see the proof of Proposition \ref{prop:hecke-on-eis}) and hence act on the subspace $H^{d+k}(Y,\Q)_{\Eis}$. Consider the map \[ \Eis' := \left( \frac{1}{2^d}\prod_{j=1}^d (1+c_j)^* \right) \circ \Eis'' \from H^{d+k}(Y,\Q)_{\del} \to H^{d+k}(Y,\Q)_{\Eis}, \]
where $\Eis''$ is any section $H^{d+k}(Y,\Q)_{\del} \surj H^{d+k}(Y,\Q)_{\Eis}$.

\begin{prop}
The map $\Eis'$ is characterized among sections of $H^{d+k}(Y,\Q)_{\Eis} \to H^{d+k}(Y,\Q)_{\del}$ by the property \[ \im(\Eis') \subset H^{2d-1-i}(Y, \C)^{c_1^* = \ldots = c_d^* = -1}.\]
\end{prop}
\begin{proof}
\begin{align*}H^{d+k}(Y, \Q(d))_{\Eis}^{c_1^* = -1,\ldots,c_d^* = -1} &\isom H^{d+k}(\H^d/P(\Z), \Q(d)) \\
&\isom H_{d-1-k}(\H^d/P(\Z), \Q). \\
\end{align*}
\end{proof}

The maps $\Eis'$ are those referred to in Theorem \ref{thm:harder-filtration}. Essentially by definition, the sections $\Eis'$ are Hecke-equivariant and $\pi_0(\GL_2(F_{\infty}))$-equivariant. However, they might not be compatible with the Hodge filtration on cohomology, as the maps $c_i$ are not holomorphic.

\subsection{Proof of Theorem \ref{thm:harder-filtration}}
As in Section \ref{sec:weak-eis-periods}, for $I \subset [d]$, we have a $(2d-|I|)$-form $E'_I = E_{\phi(I)} \eta_{\phi(I)} \wedge \omega_{[d] - I}$. To prove Theorem \ref{thm:harder-filtration} (1), it suffices to show that, for all $I \subset [d]$ with $|I| > 2$,
\[ c_i^*(E'_I) + E'_I = d \eta_J \]
for some differential form $\eta_J$ on $Y$.
\begin{lem}
Fix $I \subset [d]$, $|I| > 2$.
\begin{enumerate}
\item If $i \in I$,
\[ (c_i^* + 1) \cdot (E'_I) = d(E_{\phi(I-\{i\})} \eta_{\phi(I-\{i\})} \wedge \omega_{[d] - I}). \]
\item If $i \notin I$, 
\[ (c_i^* + 1) \cdot (E'_I) = 0. \]
\end{enumerate}
\end{lem}
\begin{proof}
This is an easy consequence of Proposition \ref{prop:deriv-via-gK}.
\end{proof}

 The fact that this lemma does not hold when $|I| = 2$, due to the residue term in Proposition \ref{prop:deriv-via-gK} (2), gives another proof of Theorem \ref{thm:harder-filtration} (2).

%\begin{lem}
%For $I \subset [d], |I| = 2$, we have 
%\[ (c_i^* + 1) \cdot (E_{\phi(I)} \eta_{\phi(I)} \wedge \omega_{[d] - I}) = d(E_{\phi(I - \{ i \})} \eta_{\phi(I - \{ i \})} \wedge \omega_{[d] - I}) + 
%\end{lem}

\appendix
%\addcontentsline{toc}{section}{Appendices}

\section{Deligne-Beilinson cohomology}\label{appendix:weak-MHS}
\subsection{Weak mixed Hodge structures}
\begin{defn}A weak MHS is a pair $(V,F^{\dot})$, for $V$ a $\Z$-module and $F^{\dot}$ a decreasing (finite, exhaustive, separating) filtration on $V_{\C} = V \tensor_{\Z} \C$. Morphisms of weak MHS are homomorphisms $f \from V_1 \to V_2$ such that $f_{\C}$ is strictly compatible with the filtrations $F^{\dot}_1, F^{\dot}_2$.
\end{defn}

We let $\MHS_{w}$ denote the category of weak MHS. Some basic properties:
\begin{enumerate}
\item $\MHS_{w}$ is an abelian category.
\item We have exact functors \begin{align*}F^0 \from & \MHS_w \to Vec_{\C}, \\
(-)_{\Z} \from & \MHS_w \to Ab.
\end{align*}
\item We have a Tate twist $V(1) := V \tensor \Z(1)$, with \[ \Z(1) := (2 \pi i \Z), F^1\Z(1)_{\C} = \Z(1)_{\C}, F^2\Z(1)_{\C} = 0. \]
\item $\Ext^1_{\MHS_w}(\Z(0), \Z(1)) \isom \C^*$ 
\item The forgetful functor $\MHS \to \MHS_{w}$ is exact.
\end{enumerate}

Note that while $(V, F^{\dot}) \mapsto F^0V_{\C}$ and $(V, F^{\dot}) \mapsto V$ are exact, their intersection $V \mapsto F^0V_{\C} \cap V_{\Z}$ is not.

% es which compute $R\Hom_{\MHS}(\Z(0),-)$ and  
%$\Gamma \from \Ch^+(\MHS) \to \Ch^+(Ab)$, 

Beilinson \cite{Beilinson-AH} defines a functor $\Gamma_{w} \from \Ch^+(\MHS_w) \to \Ch^+(\mathrm{Ab})$ such that, for $A^* \in \Ch^+(\MHS_w)$, $\Gamma_w(A^*) \in \Ch^+(\mathrm{Ab})$ is quasi-isomorphic to $R\Hom_{\MHS_{w}}(\Z(0),A^*)$. For example, for $V \in \MHS_w$, $R\Hom_{\MHS_w}(\Z(0), V)$ is quasi-isomorphic to the complex \[ \Gamma_{w}(V): \ \ V_{\Z} \to F^0V_{\C}. \]

\subsection{Weak mixed Hodge complexes}
%Beilinson also produces a contravariant functor $X \mapsto C^*(X)$ from varieties over $\C$ to $D^+(\MHS)$. This can be composed with $D^+(\MHS) \to D^+(\MHS_{w})$. He proves that Deligne-Beilinson cohomology $H^i_{\DB}(X, \Z(j))$ is the cohomology of $R\Hom_{\MHS_w}(\Z(0), C^*(X) \tensor \Z(j))$. 

There is a category of \emph{weak mixed Hodge complexes}, $\MHC_w$, consisting of diagrams 
\[ C^*_{\Z} \to C^*_{\C} \leftarrow (C^{' *}_{\C}, F^{\dot}) \]
of (filtered) chain complexes, such that 
\begin{enumerate}
\item $C^*_{\Z} \tensor \C \to C^*_{\C}$, $C^{' *}_{\C} \to C^*_{\C}$ are quasi-isomorphisms, 
\item The differential $d$ on $C^{' *}_{\C}$ is strictly compatible with the filtration $F^{\dot}$
\item The cohomology (with its induced filtration) $(H^*(C^*_{\Z}), F^{\dot})$ lives in $\Ch(\MHS_w)$
\end{enumerate}
Morphisms in $\MHC_w$ are morphisms of complexes which induce morphisms of $\MHS_w$ on cohomology. There is a functor $\Ch(\MHS_w) \to \MHC_w$ - note however that it is not full.

Importantly, (\cite{Beilinson-AH} 3.1) implies that, given a morphism $f \from A \to B$ of weak mixed Hodge complexes, the cone of chain complexes $\Cone(f)$ is naturally a weak mixed Hodge complex. Analogously to Beilinson's definition of $D^{+}_{\MHC}$, the triangulated category of (bounded below) mixed Hodge complexes, we have a triangulated category of (bounded below) weak mixed Hodge complexes, $D^{+}_{\MHC_w}$, by inverting quasi-isomorphisms of complexes. 

The same proof shows that
\begin{thm}[\cite{Beilinson-AH}]\label{thm:mhc-cmh}
$D^+_{\MHC_w} \isom D^+(\MHS_w)$. 
\end{thm}

There is a functor $\Gamma_w \from \MHC^+_w \to \Ch^+(\mathrm{Ab})$, such that for a complex of mixed Hodge structures $ \in \Ch^+(\MHS_w)$, $\Gamma_w$ is the same as the map $\Gamma_w$ above.

\subsection{Weak mixed Hodge complexes associated to smooth varieties}
\

Let $j \from X \into \bar{X}$ be a good compactification (i.e.\ $\bar{X}$ is smooth and proper, with $\bar{X} - X$ a NCD). We have a diagram of sheaves on $\bar{X}$ in the classical topology:
\[ \mathcal{D}_{(X, \bar{X})}: \ \ j_*\underline{\Z} \to \Omega_{\bar{X}}^*(\del X) \leftarrow (\Omega_{\bar{X}}^*(\del X), F^{\cdot}). \]
Then $R\Gamma(\bar{X}, \mathcal{D}_{(X, \bar{X})})$ is a weak mixed Hodge complex, well-defined and functorial as an object of $D^{+}_{\MHC_w}$. % \isom D^+(\MHS_w)$. % the triangulated category of weak mixed Hodge complexes.

Beilinson shows that Deligne-Beilinson cohomology $H^i_{\DB}(X, \Z(j))$ is the cohomology of the complex $\Gamma_w(\mathcal{D}_{(X, \bar{X})} \tensor \Z(j))$. 
%Further, for a complex of mixed Hodge structures $ \in \Ch^+(\MHS_w)$, $\Gamma_w$ is the same as the map $\Gamma_w$ above.
%TODO: This element of $D^+(\MHS_w)$ is the same as going through mixed Hodge complexes.

\subsection{Functorial weak mixed Hodge complexes}\label{app:functorial-mhc}
It is convenient to have a functorial complex computing Deligne-Beilinson cohomology on smooth varieties. For example, this is required to apply results of Bloch \cite{Bloch-Regulators}, Scholl \cite{Scholl}, and the method of Geisser-Levine \cite{Geisser-Levine} to construct and study regulators from motivic cohomology to Deligne-Beilinson cohomology.

To do this, it will suffice to define a functor
\[ C^*(-) \from Sm_{\C}^{op} \to \MHC_w^+ \]
factoring $Sm_{\C}^{op} \to D^+_{\MHC_w}$, $X \mapsto R\Gamma(\bar{X},\mathcal{D}_{(X, \bar{X})})$. The functor $\Gamma_w(C^*(-)) \from Sm_{\C}^{op} \to \Ch(Ab)$ is then a functorial complex computing Deligne-Beilinson.
 
%ial weak mixed Hodge complex, equal to $\Gamma_w(\mathcal{D})$ in $D^+_{\MHC_w}$.

%We consider pairs $(X, \bar{X})$, for $X$ a smooth variety and $\bar{X}$ a smooth compactification such that $\bar{X} - X$ is a SNCD. We consider maps of pairs $f \from (X_1, \bar{X_1}) \to (X_2, \bar{X_2})$.

%If we restrict to smooth varieties, we can produce a functor $X \mapsto C^*(X)$ valued in $\Ch^+(\MHS_w)$, which is equal to Beilinson's functor upon composition with $\Ch^+(\MHS_w) \to D^+(\MHS_w)$. For $X$ smooth, Bloch-Scholl only need this functoriality in order to construct the regulator
%\[ r_{\DB} \from H^i_{\Mot}(X, \Q(j)) \to H^i_{\DB}(X, \Q(j)) \]
%and relate it to geometrically constructed extensions as in Section \ref{subsec:reg-and-extn}. %

%$\Ext^1_{\MHS_w}(\Q(0), H^{i-1}(X, j))$.

We let the underlying complex be the $C^{\infty}$ singular cochains $S^*_{\infty}(X,\Z)$, i.e.\ the dual of $C^{\infty}$ singular simplices. Integrating differential forms over such simplices $\Delta^i \to X$, we have a quasi-isomorphism of complexes \[ \int \from A^*_X \to S^*_{\infty}(X,\C),\] which is degree-wise injective.

For $X \subset \bar{X}$ a good compactification (i.e.\ $\bar{X}$ smooth, $\del\bar{X} := \bar{X} - X$ a SNCD), define \begin{align*}
A^n_{X}(\log \del X) := & \sum_{p+q = n} \Omega^p_X(\log \del X) \tensor_{\O_X} A^{0,q}_X \\
F^{p'}(A^*_{X}(\log \del X)) := & \sum_{p \geq p'} \Omega^p_X(\log \del X) \tensor_{\O_X} A^{0,q}_X.
\end{align*}

 We define $F^p(S^*_{\infty}(X,\C)) = S^*_{\infty}(X,\C)$ for $p \leq 0$, and for $p > 0$, we define \[ F^p(S^*_{\infty}(X,\C)) \subset S^*_{\infty}(X,\C) \] to be the image of \[ \cup_{X \subset \bar{X}} F^p(A^*_{\bar{X}}(\log X)) \subset A^*_X \nmto{\int} S^*_{\infty}(X,\C), \]
where the union is taken over all good compactifications. 
 
\begin{lem}
\
\begin{enumerate}
\item $(S^*_{\infty}(X,\Z), F^{\dot}) \in \Ch^+(\MHS_w)$
\item Given a map $f \from X_1 \to X_2$, the map $f^* \from S^*_{\infty}(X_2,\Z) \to S^*_{\infty}(X_1,\Z)$ is compatible with the filtrations, strictly so if we pass to $H^*(S^*_{\infty}(-,\Z))$. Hence we obtain a functor
\[ Sm_{\C}^{op} \to \MHC_w^+. \]
\end{enumerate}
\end{lem}
\begin{proof}
(1) To prove that this is a chain-complex of weak MHS, we must invoke results of Deligne (\cite{HodgeII}).  %whose proofs require using the weight filtration, which we have deliberately thrown out when defining weak MHS
The differentials $d$ on $A^*_{\bar{X}}(\log X)$ are strictly compatible with the filtration (loc. cit.). As a morphism $(\bar{X}, X) \to (\bar{X}', X)$ of good compactifications of $X$ (extending the identity on $X$) induces a filtered quasi-isomorphism \[ (A^*_{\bar{X}'}(\log X), F^{\dot}) \to (A^*_{\bar{X}}(\log X), F^{\dot}),\] and the system of all compactifications of $X$ is filtered (in the sense of limits/colimits, not filtrations), we find that the map $(A^*_{\bar{X}}(\log X), F^{\dot}) \to (S^*_{\infty}(X, \C), F^{\dot})$ is a filtered quasi-isomorphism. This implies that $d$ is strictly compatible with the filtration $F^{\cdot}$ on $S^*_{\infty}(X,\C)$ (loc. cit., Prop. 1.3.2). 
Therefore \[ (S^*(X,\Z), F^{\dot}) \in \Ch^+(\MHS_w). \]

(2) This follows from the resolution of singularities. In particular, for any map $X_1 \to X_2$ of smooth varieties, and any good compactification $(X_2, \bar{X_2})$, there exists a good compactification $\bar{X_1}$ of $X_1$ so that the map extends to $\bar{X_1} \to \bar{X_2}$. 

That the map is strictly compatible with filtrations after passing to cohomology follows from (loc. cit.)
\end{proof}

%TODO: We also need that for any compactification , there exists a compactification $(X_1, \bar{X_1})$,

We relate this to Beilinson's object in $D^+(\MHS_w)$ as follows.
Let $j \from X \into \bar{X}$ a good compactification. We have a diagram of sheaves on $\bar{X}$ in the classical topology:
\[ \mathcal{D}: \ \ j_*\underline{\Z} \to \Omega_{\bar{X}}^*(\del X) \leftarrow (\Omega_{\bar{X}}^*(\del X), F^{\dot}). \]
Then $R\Gamma(\bar{X}, \mathcal{D})$ is the standard weak mixed Hodge complex.
%, well-defined in $D^{+}_{\mathcal{H},w}$, the triangulated category of weak mixed Hodge complexes. % (the existence of an injective resolution of filtered complexes of sheaves follows from Godemont resolutions)
%Beilinson's method proves that it is quasi-isomorphic (in $D^+_{\mathcal{H},w}$) to a complex of weak MHS. This is Beilinson's complex of weak MHS associated to $X$, and it is well-defined and functorial in $D^+(\MHS_w)$.

Instead of taking an injective resolution of the above diagram of sheaves, we can take the following flasque resolution
\[ \begin{tikzcd}
 j_*\underline{\Z} \arrow{r}\arrow{d}& \Omega_{\bar{X}}^*(\del X) \arrow{d}{\int} & \arrow{l} (\Omega_{\bar{X}}^*(\del X), F^{\dot}) \arrow{d} \\
j_* \underline{S}^*_{\infty, \Z} \arrow{r} & j_* \underline{S}^*_{\infty, \C} & \arrow{l}{\int} (\underline{A}^*_{\bar{X}}(\del X), F^{\dot}), \\
\end{tikzcd} \]
where we have a sheaf of $C^{\infty}$ cochains $\underline{S}^*_{\infty, \Z}$ on $X$, and a sheaf of $C^{\infty}$ differential forms with log poles $\underline{A}^*_{\bar{X}}(\del X)$. The global sections of the bottom row are quasi-isomorphic to the standard weak MHC above. We could further replace $(\underline{A}^*_{\bar{X}}(\del X), F^{\dot})$ with its image in $\underline{S}^*_{\infty, \C}$ (adjusting $F^0$ as above), and take the co-limit over all compactifications $X \subset \bar{X}$. We then obtain our complex $(S^*_{\infty}(X, \Z), F^{\dot})$ above.  %We can further replace the bottom row with 
%\[ j_* \underline{S}^*_{\infty, \Q} \arrow{r} & j_* \underline{S}^*_{\infty, \C} & \arrow{l}{\int} (\underline{A}^*_{\bar{X}}(\del X), F^{\cdot})

%ISSUE: The individual terms are quasi-isomorphic, but we do not have a quasi-isomorphism of diagrams??? We need a diagram of objects up to q-isom, not a diagram up to q-isom!

% H^1(\G_m, \Z) should be pure of weight $2$. Is $F^p \cap \bar{F^q} = 0$ for $p + q > 2$? It might be the case that the shifted weight filtration on the complex actually endows $S^*_{\infty}$ with a MHS, as opposed to a weak MHS.

\subsection{An analog of a result of Jannsen}
\begin{comment}
Scholl \cite{Scholl} defines a map $H^i_{\Mot}(X, \Z(j))_{hom} \to \Ext^1_{\MHS_w}(\Z(0), H^{i-1}(X, j))$, in terms of extensions in the cohomology of simplicial varieties. We explained this in Section \ref{subsec:reg-and-extn}. However, as the hypotheses of (\cite{Scholl}, 2.) are not obviously satisfied for $\MHS_w$, one needs to use (weak) mixed Hodge complexes in order to prove Theorem \ref{thm:scholl-regulator}, (2).

To literally apply (\cite{Scholl}, 2.), we would need a functor $Sm_{\C}^{op} \to \Ch^+(\MHS_w)$ whose cohomology computes $H^*(X)$. As a substitute, we use a functor $C^*(-) \from Sm_{\C}^{op} \to \MHC_w^+$ (for example, the functor $X \mapsto (S^*_{\infty}(X, \Z), F^{\cdot})$). This extends a functor on pairs of simplicial smooth varieties.

We have, in the notation of Section \ref{subsec:reg-and-extn}, an extension given by pulling back the following sequence along $cl(\alpha)$:
\[ \begin{tikzcd}[column sep = small]
 & & \Z(0) \arrow{d}{cl(\alpha)} \\
 0 \to H^{2p-1}(\Sigma^q(X), \Z(p)) \arrow{r}& H^{2p-1}(\Sigma^q(X) - Z, \Z(p)) \arrow{r}& H^{2p}_{Z}(\Sigma^q(X), \Z(p)). \\
\end{tikzcd} \]

We have a triangle in $\MHC_w^{+}$:
\[ C^*(\Sigma^q(X)) \tensor \Z(p) \to C^*(\Sigma^q(X) - Z) \tensor \Z(p) \to C^{*+1}_Z(\Sigma^q(X)) \tensor \Z(p) \]
where the right-most complex is the cone of a genuine morphism of complexes, with its natural structure of a weak mixed Hodge complex.

The following abstract proposition then implies Theorem \ref{thm:scholl-regulator}, (2). It is an easy analog of (\cite{Scholl} 2.7).
\end{comment}

We prove an analog of a result of Jannsen's (\cite{Jannsen-MixedRealizations} 9.4). For a complex $A \in \MHC^+_w$, there is an exact sequence
\[ 0 \to \Ext^1_{\MHS_w}(\Z(0), H^{i-1}(A)) \to H^i(\Gamma_w(A)) \to \Hom_{\MHS_w}(\Z(0), H^i(A)) \to 0.\] In particular, we have a map 
\[ \ker(H^i(\Gamma_w(A)) \to H^i(A)) \isom \Ext^1_{\MHS_w}(\Z(0), H^{i-1}(A)). \] If $A' \in \Ch(\MHS_w)$ is isomorphic to $A$ in $D^+_{\MHC_w}$, this exact sequence corresponds to the Leray-Serre spectral sequence for $R\Hom_{\MHS_w}(\Z(0), A')$. 

\begin{prop}\label{prop:hodge-extn}
Consider a morphism $f \from A \to B$ in $\MHC^+_w$ such that $H^{i-1}(A) \to H^{i-1}(B)$ is injective. Consider also a class \[ \wt{
\alpha} \in \ker(H^i(\Cone(\Gamma_w(f))) \nmto{\delta} H^i(\Gamma_w(A)) \to H^i(A)). \] Denote the image of $\wt{\alpha}$ in $H^i(\Cone(f))$ by $\alpha$.

Then the image of $\wt{\alpha}$ under 
\[ \ker(H^i(\Gamma_w(A)) \to H^i(A)) \to \Ext^1_{\MHS_w}(\Z(0), H^{i-1}(A)) \] classifies the extension obtained by the following pull-back:
\[\begin{tikzcd}[column sep = small]
 & & \Z(0) \arrow{d}{\alpha} \\
 0 \to H^{i-1}(A) \arrow{r}& H^{i-1}(B) \arrow{r}& H^0(\Cone(f)). \\
\end{tikzcd}, \]
\end{prop}

In the rest of this section, we prove this proposition. We can assume, by shifting the complexes, that $i = 1$.

\todo{I have not check sign conventions for cones here, although it is clear that it does not matter.}

Suppose have have a morphism $f \from A \to B$ in $\MHC_w^{+}$, such that $H^0(A) \to H^0(B)$ is injective. We have a triangle $A \to B \to \Cone(f)$, giving a LES
\[ 0 \to H^0(A) \to H^0(B) \to H^0(\Cone(f)) \to H^1(A) \to \ldots \]
of $\MHS_w$. We also have a triangle $\Gamma_w(A) \to \Gamma_w(B) \to \Cone(\Gamma_w(f))$ in $Ch^+(Ab)$. %Noting that $Cone(\Gamma_w(f))$ is quasi-isomorphic (?) to $\Gamma_w(

Suppose we have an element $\alpha \in H^0_{\MHS_w}(\ker(H^0(\Cone(f)) \to H^1(A)))$. We lift $\alpha \in H^0(\Cone(f))$ in two different ways: we lift $\alpha_{\Z}$ to a class $[\phi] \in H^0(B)_{\Z}$ represented by $\phi \in B^0_{\Z}$, and $\alpha_{\C}$ to a class $[\omega] \in F^0H^0(B)_{\C}$, represented by $\omega \in B^0_{\C}$. Since $[\phi] - [\omega] \in H^0(A)_{\C}$, we have that $\phi - \omega = dh + \eta$, for some $h \in B^{-1}_{\C}, \eta \in A^0_{\C}, d\eta = 0$. The class $[\eta] \in \frac{H^0(A)_{\C}}{F^0H^0(A)_{\C} + H^0(A)_{\Z}}$ represents the extension class of
\[ 0 \to H^0(A) \to E \to \Z(0) \cdot \alpha \to 0 \]
in $\Ext^1_{\MHS_w}(\Z(0), H^0(A))$. This is the extension given by the pull-back \[\begin{tikzcd}[column sep = small]
 & & \Z(0) \arrow{d}{\alpha} \\
 0 \to H^0(A) \arrow{r}& H^0(B) \arrow{r}& H^0(\Cone(f)). \\
\end{tikzcd}. \]

We have a cochain \[(\phi, \omega, h) \in \Gamma_w(B)^0 = B^0_{\Z} \oplus B^0_{\C} \oplus B^{-1}_{\C}. \] It is not necessarily closed: \[ d(\phi, \omega, h) = (0, 0, (\phi - \omega) - dh) = (0,0, \eta) \in B^1_{\Z} \oplus B^1_{\C} \oplus B^{0}_{\C}. \] However, since $\eta \in B^0_{\C}$ lifts to $A^0_{\C}$, we find that the class
\[ ((\phi, \omega, h), (0,0, \eta)) \in \Cone(\Gamma_w(f))^0 = \Gamma_w(B)^0 \oplus \Gamma_w(A)^1 \] is closed. We denote its cohomology class by
$\wt{\alpha} \in H^0(\Cone(\Gamma_w(f)))$.

We have maps \begin{align*}
H^0(\Cone(\Gamma_w(f))) & \nmto{\delta} H^1(\Gamma_w(A)), \\
H^0(\Cone(\Gamma_w(f))) & \isom H^0(\Gamma_w(\Cone(f))) \to H^0(\Cone(f)). \\
\end{align*}

The following two lemmas prove the proposition. The first proves the proposition for cocycles of the form $((\phi, \omega, h), (0,0, \eta))$ in $\Cone(\Gamma_w(f))$, and the second says that every cocycle in question is cohomologous to one in this form.
\begin{lem}
\
\begin{enumerate}
\item The image of $\wt{\alpha}$ under \[ H^0(\Cone(\Gamma_w(f))) \isom H^0(\Gamma_w(\Cone(f))) \to H^0(\Cone(f)) \] is $\alpha$.
\item The image of $\wt{\alpha}$ in $H^0(\Cone(\Gamma_w(f))) \to H^1(\Gamma_w(A))$ is represented by $(0,0, \eta)$.
\item The class $[(0,0, \eta)]$ is contained in $\ker(H^1(\Gamma_w(A)) \to H^1(A))$. % \isom \Ext^1_{\MHS_w}(\Z(0), H^0(A))$.
\item Under the isomorphism \[ \ker(H^1(\Gamma_w(A)) \to H^1(A)) \isom \Ext^1_{\MHS_w}(\Z(0), H^0(A)),\]
$[(0,0, \eta)]$ corresponds to the extension $[\eta]$.
\end{enumerate}
\end{lem}  
\begin{proof}
These statements are easy to verify using the isomorphism $\Cone(\Gamma_w(f)) \isom \Gamma_w(\Cone(f))$.

% ((\phi, \omega, h), (0,0, \eta))
\end{proof}

\begin{lem}
Every class in the kernel of \[ H^0(\Cone(\Gamma_w(f))) \nmto{\delta} H^1(\Gamma_w(A)) \to H^1(A) \] may be represented by a cochain of the form $((\phi, \omega, h), (0,0, \eta))$.
\end{lem}
\begin{proof}
Conside such a class $((\phi_B, \omega_B, h_B), (\phi_A,\omega_A,h_A)) \in \Cone(\Gamma_w(f))^0$. We have:
\[ \delta((\phi_B, \omega_B, h_B), (\phi_A,\omega_A,h_A)) = (\phi_A,\omega_A,h_A) \in \Gamma_w(A)^1. \] %For this class to be closed, we would need $d\phi_A = 0, d\omega_A = 0, dh_b = \phi_A - \omega_A$.
For it to be in the kernel of $H^1(\Gamma_w(A)) \to H^1(A)$, we need $\phi_A = d\gamma_1, \omega_A = d\gamma_2$. We have $d(\gamma_1, \gamma_2,0) = (\phi_A, \omega_A, \gamma_2 - \gamma_2)$.

We find that, using $((0,0,0), (\gamma_1, \gamma_2,0))$, our original class is cohomologous to a class 
\[ ((\phi'_B, \omega'_B, h'_B), (0,0,h'_A)). \]
\end{proof}

\begin{comment}

Now, Deligne-Beilinson cohomology of these spaces is computed by the complex $\Gamma_w(C^*(-) \tensor \Z(p))$. We let $(\phi, \omega, h)$ denote an element of this complex: 
\begin{align*}
\phi \in (C^*(X) \tensor \Z(p))_{\Z}, \\
\omega \in (C^*(X) \tensor \Z(p))_{\C}, \\ 
h \in (C^{*-1}(X) \tensor \Z(p))_{\C}.
\end{align*}
\end{comment}

%Scholl says that his technique nevertheless works for absolute Hodge cohomology, but I don't know why. 
%It might be that a small amount of homological algebra is all that is required to circumvent the need for functorial chain complexes. However, we will simply define the required functor $Sm_K^{op} \to \Ch^+(\MHS_w)$.

\section{Regulators for fields}\label{appendix:reg-on-point}

We use the notation of Section \ref{sec:extn-classes}. Let $\Delta^1 = \Spec(\Z[t_1, t_2])/(t_1 + t_2 = 1) \isom \A^1_{\Z}$ be the algebraic 1-simplex, $\del \Delta^1 = \{0, 1\} \subset \A^1(\Z)$.

\begin{thm}[\cite{Bloch}]
The map \begin{align*}
L^* &\to H^1_{\Mot}(\Spec(L), \Z(1)), \\
a &\mapsto \{ a \} \subset \P^1_{L} - \{ 1 \} \isom \Spec(L) \times \Delta^1
\end{align*}
is an isomorphism.
\end{thm} 

\begin{lem}\label{extn-classes}
\hfill
\begin{enumerate}
\item $\Ext^1_{G_L}(\Z_l, \Z_l(1)) = L^* \tensor \Z_l$
\item $\Ext^1_{\MHS_w}(\Z(0), \Z(1)) = \C/\Z(1) \isom \C^*$
\end{enumerate}
\end{lem}
\begin{proof}
1) follows from Kummer theory.

2) is proven in \cite{Carlson} for $\MHS$, as opposed to $\MHS_w$, but the proof is the same. 

We give an slight reformulation of the isomorphism, which is more convenient for us. Given an extension \[ 0 \to \Z(1) \to E \nmto{\alpha} \Z \to 0, \] we can apply $\Fil^0$ to $E \tensor \C$, to obtain
\[ 0 \to 0 \to \Fil^0(E \tensor \C) \isom \C \to 0. \] We denote this isomorphism $s \from \C \to \Fil^0(E \tensor \C).$ For the extension
\[ 0 \to \Z(1) \to E^{\dual}(1) \to \Z \to 0, \]
we can take any lift of $1 \in \Z$ to an element $\gamma \in E^*(1)$. This gives a map
$\gamma \from E \tensor \C \to \C.$
This depends on the choice of $\gamma$, but if we restrict to $\alpha^{-1}(\Z)$, we get a well-defined map $\gamma \from \alpha^{-1}(\Z) \to \C/\Z(1).$ Composing with the map $s$, we obtain
\[ \Z \nmto{s} \alpha^{-1}(\Z) \nmto{\gamma} \C/\Z(1). \]

This is easily seen to be the same as the definition in \cite{Carlson}: the element $s(1) - x$ lives in $\Z(1) \tensor \C = \C$, and is well-defined in $\C/\Z(1)$ (i.e.\ independent of the choice of $x$). The map is $[E] \mapsto s(1)-x \in \C/\Z(1)$. 
%This map measures how far the section $s$ is from respecting both the integral structure and filtration, hence being a map of MHS.
\end{proof}

\begin{rmk}
Note that $\Ext^1_{\MHS}(\Z(0), \Z(1)) = \Ext^1_{\MHS_w}(\Z(0), \Z(1))$, so we do not lose any information by working with $\MHS_w$ as opposed to $\MHS$.
\end{rmk}

%QUESTION: Is this a general fact about $\Ext^1$ classes? Seems like it would at least work for twists by $\Z(n)$ in the Hodge setting.

There is an involution $\iota \from \Ext^1_{\cC}(1_{\cC}, 1_{\cC}(1)) \to \Ext^1_{\cC}(1_{\cC}, 1_{\cC}(1))$, sending $E \mapsto E^{\dual}(1)$.
\begin{lem}\label{lem:dual-extn}
$\iota([E]) = -[E]$
\end{lem}
\begin{proof}
\textbf{1) $\cC = G_L$-mod}

Consider an extension $E$ of the form \[ 0 \to \mu_{n} \to \{ \zeta_n^i a^{j/n} \} \to \Z/n\Z \to 0, \] generated by the $n$-th roots of some $a \in L^*$ (really, all $a^i$, $i = 0,\ldots,n-1$). 
Our candidate for the dual extension is $E'$:
\[ 0 \to \mu_{n} \to \{ \zeta_n^i a^{-j/n} \} \to \Z/n\Z \to 0. \]

We need to exhibit a non-degenerate, Galois-equivariant pairing
\[ (\cdot,\cdot) \from E \times E' \to \mu_n. \]
Moreover, to match up the sub/quotient of $E^*(1)$ with those of $E'$, we need that $\mu_n \times \mu_n \to \mu_n$ is trivial, with the induced map $
\mu_n \times \Z/n\Z \to \mu_n$ being the map $(\zeta_n^i, j) \mapsto \zeta_n^{ij}$ (and similarly for $\Z/n\Z \times \mu_n \to \mu_n$). 

Define the pairing by \begin{align*}
(\zeta_n, \zeta_n) &= 1 \\
(a^{1/n}, a^{-1/n}) &= 1, \\
(\zeta_n, a^{-1/n}) &= \zeta_n, \\
(a^{1/n}, \zeta_n) &= \zeta_n.
\end{align*}

It is easy to check Galois-equivariance. For example, \begin{align*}
(\sigma(a^{1/n}), \sigma(a^{-1/n})) &= (\zeta_n^{c(\sigma)}a^{1/n}, \zeta^{-c(\sigma)}a^{-1/n}) \\ 
&= (\zeta_n, a^{-1/n})^{c(\sigma)} \cdot (\zeta_n, a^{-1/n})^{-c(\sigma)} \\ 
&= \zeta_n^{c(\sigma) + c(-\sigma)} = 1 = \sigma((a^{1/n}, a^{-1/n})).
\end{align*} 

\textbf{2) $\cC = \MHS_w$}

We use the notation of \ref{extn-classes}. Let $\phi \from E \tensor \C \to \C$ be any map such that $\phi(s(1)) = 0$, $\phi|_{\Z(1)}$ is the inclusion $\Z(1) \into \C$. We choose an element $x \in E$ such that $\alpha(x) = 1$. The extension $\iota(E)$ corresponds to the number $\phi(x) \in \C/\Z(1)$. It is easy to see that $\phi(x) = \phi(x - s(1)) = -[E]$.
\end{proof}

There exist regulator maps \[ r_{\mathcal{C}} \from H^1_{\Mot}(\Spec(L), \Q(1)) \to \Ext^1_{\mathcal{C}}(1_{\mathcal{C}}, 1_{\mathcal{C}}(1)). \] Following \cite{Scholl}, $r_{\cC}(a)$ is the extension classifying
\[ 0 \to H^1(\Sigma^1, 1) \to H^{1}(\Sigma^1 - \{ a \}, 1) \to H^{2}_{ \{ a \}}(\Sigma^1, 1) \to 0, \]
for $\Sigma^1$ the affine nodal curve $\A^1/(0 = 1) = \Delta^1/\del \Delta^1$.

The following is well-known (eg. \cite{Geisser-Levine} in the \'etale case), but we write it down for completeness:
\begin{prop}\label{field-regulator}
\
\begin{enumerate}
\item The map $r_{\Et}$ is the inclusion $L^* \tensor \Q \into L^* \tensor \Q_l$.
\item The map $r_{\MHS_w}$ is the inclusion $L^* \tensor \Q \into \C^* \tensor \Q$. 
\end{enumerate}
\end{prop}

\begin{proof}
Fix $a \in L^* = H^1_{\Mot}(\Spec(L), \Z(1))$. Using $\P^1 - \{ 1 \} = \Delta^1$, we have a cycle $Z = \{ a \} \subset \Sigma^1$. We get an extension \[ 0 \to H^1(\Sigma^1, 1) \to H^{1}(\Sigma^1 - \{ a \}, 1) \to H^{2}_{ \{ a \}}(\Sigma^1, 1) \to 0. \] Our goal is to compute the $G_F$-action/weak MHS on this extension.

This extension is isomorphic to
\[ 0 \to H^0(\del \Delta^1, 1)/H^0(\Delta^1, 1) \to H^{1}(\Sigma^1 - \{ a \}, 1) \to H^1(\Delta^1 - \{a\}, 1) \to 0. \] As the cross-ratio of the points $(0,\infty,1,a)$ equals the cross-ratio of $(1,a,0,\infty)$, we have an isomorphism \[ \Sigma^1 - \{ a \} = (\P^1 - \{1, a \})/(0 = \infty)  \isom \G_m/(1 = a). \] We obtain the extension
\[ 0 \to H^0(\{ 1, a \}, 1)/H^0(\G_m, 1) \to H^{1}((\G_m, \{1, a \}), 1) \to H^1(\G_m, 1) \to 0. \]

For $\cC = G_L$-mod, we dualize and twist by $1_{\cC}(1)$, to obtain get the extension 
\[ 0 \to H_1(\G_m, 0) \to H_{1}((\G_m, \{ 1, a \}), 0) \to 1_{\mathcal{C}} \to 0. \]
The class of this extension in $\Ext^1_{G_L}(\Z_l, \Z_l(1))$ is the Kummer class $\kappa(a)$\todo{Sign?}. To see this, use $H^{\Et}_1(\G_m, \Z/n\Z) \isom \mu_n$ while $H_{1}^{\Et}((\G_m, \{1, a \}), \Z/n\Z) \isom \langle \mu_n, a^{1/n} \rangle \subset (L^{1/n})^*/L^*$ as $G_L$-modules (see \cite{Deligne-3Points}, where this is proven via path torsors). By \ref{lem:dual-extn}, the extension class of $H^{1}(\G_m/(1 = a), 1)$ is $-\kappa(a)$.

%Or, using Poincare duality (and a twist), we get 
%\[ 0 \to H^1(\G_m, 1) \to H^1(\G_m - \{1,a\}, 1) \to H^2_{1 \cup a}(\G_m,1)^0 \to 0.\]
%The Hodge extension class is described by the integral along $\gamma_0$ of a 1-form with residues $1$ at $1$, $-1$ at $a$. Similarly, we could integrate $dz/z$ 

For $\cC = \MHS_w$, we proceed as follows. The cohomology $H^*((\G_m, \{1, a \}), 1)$ is computed by the bi-complex 
\[ \begin{tikzcd}
\O_{1} \oplus \O_{a} & \\
\O_{\G_m} \arrow{u} \arrow{r}{d} & \Omega^1_{\G_m} \\
\end{tikzcd}. \] The isomorphism $\Fil^0 H^1(\G_m, 1) \isom \Fil^0 H^{1}((\G_m, \{1, a \}), 1)$ sends $dz/z$ to $(0,0, dz/z) \in \Fil^0((\O_{1} \oplus \O_{a}) \oplus \Omega^1_{\G_m}) = \Omega^1_{\G_m}$.

%We want to describe the isomorphism $Fil^0 H^{1}((\G_m, \{1, a \}), 1) \isom Fil^0 H^1(\G_m, 1) = \C \frac{dz}{z}$, in terms of the complex computing $H^{1}((\G_m, \{1, a \}), \C)$.

The integral structure on $H^{1}((\G_m, \{1, a \}), \C)$ is determined by that on $H_{1}((\G_m, \{1, a \}), \C)$ via the integration pairing 
\[ \int \from H^{1}((\G_m, \{1, a \}), \C) \times H_{1}((\G_m, \{1, a \}), \Z) \to \C, \]
equal to $\int_{\gamma} (x_0, x_1, \omega) = x_1 - x_0 + \int_{\gamma} \omega$ for $\gamma$ a path from $1$ to $a$.

%Choose a path $\gamma$ from $1$ to $a$. We want to pair $\gamma$ with a class $(x_0, x_1, \omega) \in (\O_{1} \oplus \O_{a} \oplus \Omega^1_{\G_m})/(\O_{\G_m})$. The pairing is $. 

By the description in Lemma \ref{extn-classes}, we see that the extension class in $\Ext^1_{\MHS_w}(\Z(0), \Z(1)) \isom \C/\Z(1)$ equals the integral $\int_{1}^{a} \frac{dz}{z} = \log(a) \in \C/\Z(1)$.
\end{proof}

\bibliography{mybib}{} 
\bibliographystyle{plain}

\end{document}